\documentclass[12pt]{amsart}

\usepackage{amssymb}
\usepackage{amscd}
\usepackage[dvips]{graphicx}

\usepackage{epsfig}

\usepackage[margin=0.66in,top=0.3in]{geometry}

\usepackage{color}

\newcommand{\nequation}{\setcounter{equation}{0}}

\newtheorem{lemma}{Lemma}

\newtheorem{theorem}{Theorem}
\newtheorem{remark}{Remark}

\def\eqn {\begin{equation}}
\def\eeqn {\end{equation}}

\newcommand{\e}{\varepsilon}
\def\real{{\mathbb R}}
\def\R{\real}

\def\ep{\varepsilon}

\def\lb{\lambda}
\def\al{\alpha}

\def\pa{\partial}

\def\C{\mathcal C}

\def\G{\mathcal G}

\def\K{\mathcal K}

\def\S{\mathcal S}

\input epsf

\newcommand{\X}{\mathbb X}

\newcommand{\Om}{\Omega}

\newcommand{\la}{\lambda}

\newcommand{\be}{\begin{equation}}
\newcommand{\ee}{\end{equation}}

\newcommand{\mcs}{{\mathcal S}}

\newcommand{\mca}{{\mathcal A}}
\newcommand{\mcb}{{\mathcal B}}

\newcommand{\mcg}{{\mathcal G}}
\newcommand{\mcl}{{\mathcal L}}
\newcommand{\mcr}{\mathcal{R}}

\newcommand{\mcx}{X}
\newcommand{\mcy}{Y}

\newcommand{\mck}{{\mathcal K}}

\newcommand{\bdc}{\mathbb{C}}

\newcommand{\bdr}{\mathbb{R}}
\newcommand{\bdz}{\mathbb{Z}}

\newcommand{\mcc}{\mathcal{C}}

\newcommand{\N}{\mathbb N}

\newcommand{\bese}{\begin{subequations}}
\newcommand{\ese}{\end{subequations}}
\newcommand{\non}{\nonumber}

\begin{document}
\title[Global bifurcation of steady gravity water waves with critical layers]{Global bifurcation of steady gravity water waves\\ with critical layers}
\author{Adrian Constantin, Walter Strauss and Eugen Varvaruca}

\address{King's College London, Department of Mathematics, Strand, WC2L 2RS, London, United Kingdom and
Faculty of Mathematics, University of Vienna, Oskar-Morgenstern-Platz 1, 1090 Vienna, Austria}
\email{adrian.constantin@kcl.ac.uk; adrian.constantin@univie.ac.at}
\address{Brown University, Department of Mathematics and Lefschetz Center for
Dynamical Systems, Box 1917, Providence, RI 02912, USA}
\email{wstrauss@math.brown.edu}%
\address{University of Reading, Department of Mathematics and Statistics, Whiteknights, PO Box 220, Reading RG6 6AX, United Kingdom}
\email{e.varvaruca@reading.ac.uk}%

\begin{abstract}
We construct large families of two-dimensional travelling
water waves propagating under the influence of gravity in a flow of constant vorticity over a
flat bed. A Riemann-Hilbert problem approach is used to recast the governing equations
as a one-dimensional elliptic pseudo-differential equation with a scalar constraint. The structural
properties of this formulation, which arises as the Euler-Lagrange equation of an energy functional,
enable us to develop a theory of analytic global bifurcation.
\end{abstract}

\maketitle

\tableofcontents

\section{Introduction}

We construct large families of two-dimensional periodic steady water waves propagating
under the influence of gravity at the surface of a layer of an incompressible, inviscid and homogeneous fluid with a flat bed.
In contrast to most previous mathematical studies,
our waves can have internal stagnation points and critical layers.
They can also have overturning profiles, that is, profiles that are not graphs.
Such phenomena cannot occur for the much-studied irrotational flows.
For an irrotational steady flow the wave profile is necessarily the graph
of a single-valued function \cite{Sp,T,V0,V2}, and there are no interior stagnation points or
critical layers.

In this paper we construct families of water waves on flows with non-vanishing constant vorticity.
Vorticity is the hallmark of underlying non-uniform currents. Flows with constant vorticity
are of interest because of their analytic tractability, being representative of a wide  range of physical scenarios.
For example,
if the waves are long compared with the average water depth, then the existence of a non-zero mean vorticity
is more important than its specific distribution \cite{DP}.  In particular, constant vorticity gives a good description of tidal currents; that is, the alternating, horizontal movement of water associated with the rise and fall of the tide, with positive/negative vorticity being appropriate
for the ebb and flood, respectively \cite{Cb}. On areas of the continental
shelf and in many coastal inlets these are the most significant currents;
the fact that they are the most regular and predictable currents adds to their appeal.
The presence of an underlying non-uniform current gives rise to new dynamically rich flow phenomena.
In particular, critical layers have been observed in many
numerical studies of water flows with an underlying current of constant vorticity, a typical
feature of the flow dynamics being the `cat's eye' flow pattern of Thomson (Lord Kelvin) \cite{Ke}.
The analysis pursued for instance in \cite{C, CSp} shows that no such patterns
can occur in the flow beneath an irrotational periodic steady water wave. General overviews of the wave-current interaction
theory can be found in \cite{Cb, Str} and of the numerics in \cite{DP, KS1, KS2, OS, V}.
In the present setting, positive vorticity $\Upsilon>0$ characterizes pure
currents $(\Upsilon Y+U_0,\,0)$ whose
surface horizontal velocity exceeds the horizontal current velocity $U_0$ on the bed $Y=0$. This is appropriate for the ebb current,
while $\Upsilon<0$ captures flood currents \cite{Cb}. Note that tides refer to the vertical motion of water
caused by the gravitational forces due to the relative motions of Moon,
Sun and Earth, whereas the tidal current flood/ebb is the horizontal movement of water
(mostly one-directional) associated with the rise and fall of the tide, respectively. A spectacular example of tidal
currents effects on sea waves is encountered at the Columbia River entrance, where appreciable tidal currents make
it one of the most hazardous navigational regions in the world because the wave height
can easily double in just a few hours cf. \cite{Jon}.

The waves investigated in the present paper
are constructed by means of bifurcation from laminar (flat) flows.
Construction of small-amplitude waves via local bifurcation has already been
accomplished in \cite{W, CV}.
What is novel here is the construction of waves of {\it large amplitude}
via global bifurcation. Now, global bifurcation has also been carried out for general vorticity under the limitation
that there are no stagnation points or critical layers \cite{CS}, that is,
under the condition that $\psi_Y<0$, where the vector $(\psi_Y,-\psi_X)$ is the
fluid velocity in the frame moving with the speed of the traveling wave.
The reason for this limitation in past studies is because they rely upon
the semi-Lagrangian transformation of Dubreil-Jacotin \cite{DJ}, which requires $\psi_Y<0$.
The critical layers are comprised of the points
where $\psi_Y=0$, while the stagnation points are those at which both $\psi_Y=0$
and $\psi_X=0$. On the other hand, a vertical scaling transformation was used in \cite{W} to construct
waves of small amplitude with critical layers. However, the intricate nature of the reformulated free-boundary problem
precludes an attempt to develop an effective global bifurcation approach from this perspective. While for
waves of small amplitude, the vertical scaling transformation induces tractable structural changes that can be captured
by linear theory, as soon as the waves cease to represent small perturbations of a flat free surface, the inherent loss of
structure cannot accommodate the pursuit of a physically tangible analysis. Two of the present authors developed in \cite{CV} a
new conformal transformation to make a different construction of waves with critical layers.
This transformation is further developed in the present  paper to construct waves of {\it large amplitude}.
We also prove that they are critical points of a natural functional derived from the energy.

We begin  Section 2 by introducing the governing equations of water waves.
Then, as presented in \cite{CV}, we use a conformal mapping from a fixed infinite strip in the plane
onto the fluid domain with its free boundary (see Figure 1). For the sake of completeness, some of the technical
background for the Hilbert transform on a strip is provided in Appendix A. By means of Riemann-Hilbert theory on the strip,
the problem is further transformed to an elliptic pseudo-differential equation (\ref{vara}) on a
horizontal line, where the unknown is the surface profile expressed as a function $v(x)$
of a conveniently transformed version of the horizontal coordinate.
The pseudo-differential equation is also coupled to a scalar equation (\ref{meana}) for the mass flux $m$.
As distinguished from the formulation in \cite{CV}, this new formulation is crucial to our global
bifurcation analysis in the subsequent sections. The equations appear rather technical as
they involve the Hilbert transform for a strip and several parameters.
Nevertheless, we show that the equations in the new formulation are precisely
the Euler-Lagrange equations of a certain functional $\Lambda$.
As explained in Appendix B, the functional $\Lambda$ comes naturally from the physical energy
$E  = \displaystyle\iint \left\{\frac12 |\nabla\psi|^2 +\Upsilon\psi + \frac Q2 - gY \right\} dYdX$,
where $\Upsilon$ is the vorticity ($\Upsilon = -\gamma$ in some references), ${Q}/{2g}$ is the total head,
which is the greatest possible height of
the wave, as discussed further at the beginning of Section 2, and $g$ is gravity.
The formulation is new even for irrotational flows of finite depth.
It could be regarded as an analogue of Babenko's formulation \cite{Ba} for the irrotational
water-wave problem of infinite depth.

In Section 3 we construct for every integer $n \ge 1$ and both choices of sign $\pm$ a
continuous solution curve $\K_{n,\pm}$.
The solution curve $\K_{n,\pm}$ consists of triples $(Q,m,v)$ belonging to a
function space $\real\times\real\times C_{per}^{2,\alpha}(\real)$, where $0<\alpha<1$ is the H\"older space index.
Each point on the curve $\K_{n,\pm}$ corresponds to a water wave, with
$n=1$ standing for the bifurcation from the lowest eigenvalue and $n \ge 2$ for the higher modes.
While local bifurcation from a simple eigenvalue is a well-known technique,
most often referred to by the names Crandall-Rabinowitz or Liapunov-Schmidt, we go beyond local bifurcation
and prove the existence of a global solution curve.
We thus obtain the existence of solutions that are not merely
small perturbations of a laminar flow (with a flat free surface).
We show that each curve $\K_{n,\pm}$ has locally a real-analytic parametrization,
by employing the ``analytic theory of bifurcation" due to Dancer, Buffoni and
Toland \cite{BT}.  The basic idea of continuation of the local bifurcation curve originates
with the work of Rabinowitz \cite{Ra}, who used Leray-Schauder
degree to extend the local bifurcation curve.  Shortly thereafter Dancer \cite{D} showed that
analytic continuation can provide a similar construction.
The two approaches provide somewhat different sets of solutions.
The main advantage of the analytic theory is that it provides a continuous curve of solutions,
whereas the degree-theoretic approach
only ensures the existence of a global continuum of solutions that might lose its character as a curve
beyond a small neighbourhood of the local bifurcation point.
The constructed solution
curve has a locally analytic parametrization around each point, even as it passes
through its branch points. Theorem \ref{glbp} provides three alternatives for how the curve ``ends" in either direction.
One is that the curve is unbounded in $\real\times\real\times C_{per}^{2,\alpha}(\real)$.
Another is that the solution curve $\K_{n,\pm}$ approaches a wave of greatest height $Q/2g$.
The third is that the curve is a closed loop that returns to the original laminar flow.

Section 4 is a key part of the analysis.  We eliminate the third
alternative and replace it by the alternative that the curve of solutions contains a wave for which the surface $\S$
intersects itself at a point directly above the trough.  This possibility
is supported by numerical computations \cite{DP, V}.  Although the surface
might overturn (that is, $Y$ could be a multivalued function of $X$), we
prove that $Y$ decreases from the crest to the trough.  We prove in
particular that a `loop' could not occur but that a limit of such
monotone waves might have a surface that intersects itself.  Moreover,
the self-intersection could happen only on the vertical line above the
trough.  These statements are proven by an intricate series of arguments
based on a maximum principle that keeps track of the possible nodes of $Y$.

Section 5 highlights some more detailed features
of the small-amplitude waves that illustrate the profound effect of
vorticity.  In particular, there exist flows of which the streamlines are
in the shape of ``cat's eyes".  In Section 6 we mention some further
results that will appear in a forthcoming paper (most of these results were already
announced in \cite{SS}), as well as some
conjectures that gained some reasonable degree of credence due to
some partial analytical results and available numerical simulations.

\section{The free-boundary problem}

The problem of periodic travelling gravity water waves in a flow of constant
vorticity $\Upsilon$ over a flat bed can be formulated as the free-boundary problem of finding

\begin{itemize}
\begin{subequations}\label{fs}
\item A laterally unbounded domain $\Omega$ in $(X,Y)$-plane, whose boundary consists of the real axis
\begin{equation}{\mathcal B}:=\{(X,0):\ X \in \R\}\end{equation}
representing the flat impermeable water bed, and an a priori unknown curve with a
parametric representation
\be{\mathcal S}:=\{(u(s),v(s)):\ s \in \R\}\label{fs1}\ee
where
\be\text{the map $s \mapsto (u(s)-s, v(s))$ is periodic of period $L$},\label{fs2}\ee representing
the water's free surface.
\end{subequations}
\item A function $\psi(X,Y)$ that is $L$-periodic in $X$ throughout $\Omega$, representing the stream
function giving the velocity field $(\psi_Y,-\psi_X)$ in a frame
moving at the constant wave speed, which satisfies the following
equations and boundary conditions:
\begin{subequations}\label{g}
\begin{align}
& \Delta\psi=\Upsilon\quad\hbox{in}\quad \Omega\,,\label{g1}\\
& \psi=0\quad \hbox{on}\quad {\mathcal S}\,,\label{g3}\\[0.1cm]
&  \psi=-m\quad \hbox{on}\quad {\mathcal B}\,,\label{g4}\\[0.1cm]
& \vert\nabla\psi\vert^{2}+2gY=Q \quad\hbox{on}\quad {\mathcal S}\,.\label{g2}
\end{align}
\end{subequations}
\end{itemize}
Here $g$ is the gravitational constant of acceleration and the constant $m$ is the relative mass flux.
The constant $Q/2g$ is  the total head, which is the greatest possible height of $\S$.
The notion of relative mass flux captures the fact that
in the moving frame the amount of water passing any vertical line is constant
throughout the fluid domain, since
$$\int_0^{v(s)} \psi_Y(X,Y)\,dY = m\,,\qquad s \in \R\,.$$
Neglecting friction and surface tension effects,
the Bernoulli equation (\ref{g2}) is a form of conservation of energy; cf. the discussion in \cite{Cb}.
Note that all the terms in
${\vert\nabla\psi\vert^{2}}/{2g}+Y  =  {Q}/{2g}$ on $\S$
have the dimension of length, the first being called the velocity head and representing the
elevation needed for the fluid to reach the velocity $\vert\nabla\psi\vert$ during frictionless
free fall, and the second term being the elevation head.
The level sets of $\psi$ are the streamlines.  A point where the gradient
of $\psi$ vanishes is called a {\it stagnation point}. A {\it critical layer} is a
curve along which $\psi_Y=0$. Such curves arise in the context of flow-reversal,
when a fluid region where the flow is oriented towards the propagation direction of
the wave is adjacent to a fluid region in which the flow is adverse to it.

\begin{figure} \centering  \includegraphics[width=16cm]{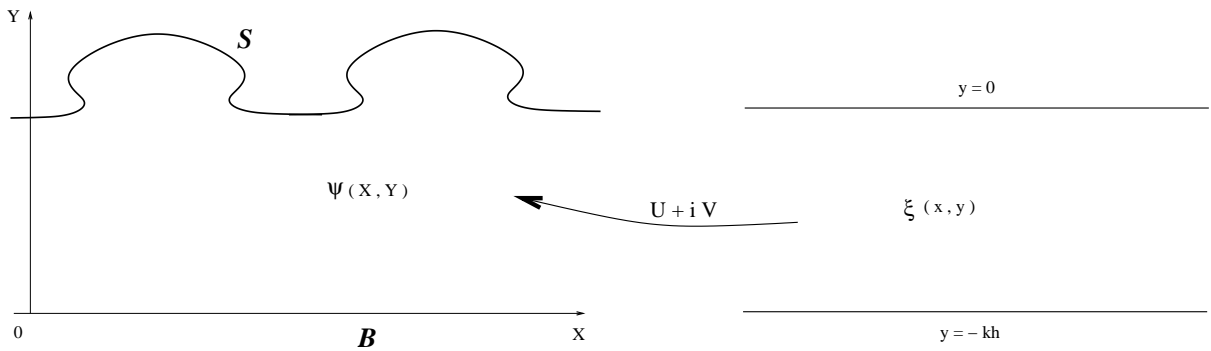}
\caption{\small The conformal parametrization of the fluid domain.}  \end{figure}

For an integer $p \ge 0$ and $\alpha\in (0,1)$, we denote by
$C^{p,\alpha}$ the space of functions whose partial
derivatives up to order $p$ are H\"{o}lder continuous with exponent
$\alpha$ over their domain of definition. By
$C^{p,\al}_{\textnormal{loc}}$ we denote the set of functions of
class $C^{p,\alpha}$ over any compact subset of their domain of
definition. Let $C^{p,\alpha}_L({\mathbb R})$ be the space of functions of
one real variable which are $L$-periodic and of class
$C^{p,\alpha}$. Throughout this paper we are interested in solutions $(\Om,\psi)$ of
the water-wave problem (\ref{g}) of class $C^{1,\alpha}$ for some
$\alpha\in (0,1)$, in the sense that $\mcs$ has a
parametrization (\ref{fs1}) with $u,\, v$ functions of class
$C^{1,\alpha}$, such that (\ref{fs2}) holds and
\begin{equation}\label{rc}
u'(s)^2+v'(s)^2\neq 0\quad\text{for all }s\in\bdr,
\end{equation}
while $\psi\in C^\infty(\Om)\cap C^{1,\alpha}(\overline\Om)$.

\subsection{The first reformulation via conformal mapping}

In this subsection we present, following \cite{CV}, the reformulation of the free-boundary
problem (\ref{g}) as the quasilinear pseudodifferential equation
(\ref{maineq}) for a periodic function of a single variable. This
involves the periodic Dirichlet-Neumann operator $\G_{kh}$ and the periodic
Hilbert transform $\C_{kh}$ for a strip, for whose definition and detailed properties we refer to Appendix A.

For any $d>0$, we denote by $\mcr_{d}$ the horizontal strip $\{(x,y) \in \R^2:\ -d < y < 0\}.$
We consider functions of period $L=2\pi/k$ where $k$ is the wave number.
Given an $L$-periodic strip-like domain $\Omega$ of class $C^{1,\alpha}$ for some $\alpha \in (0,1)$, having
the real axis as its lower boundary, the
considerations in \cite{CV} show that there exists a unique constant $h>0$,
called the {\it conformal mean depth} of $\Omega$, and a conformal map $U+iV$ from
$\mcr_{kh}$ onto $\Omega$ with the following properties.
It admits an extension as a homeomorphism of class $C^{1,\alpha}$
between the closures of these domains, such that
\begin{equation}\label{v}
\left\{\begin{array}{l}
U(x+2\pi,y)=U(x,y)+\displaystyle\frac{2\pi}{k}\quad\hbox{for}\quad (x,y) \in \mcr_{kh},\\[0.2cm]
V(x+2\pi,y)=V(x,y)\quad\hbox{for}\quad (x,y) \in \mcr_{kh},\\[0.2cm]
V(x,-kh)=0\quad\hbox{for all}\quad x \in \R
\end{array}\right.
\end{equation}
See Figure 1. Moreover, $U,\,V \in C^{1,\alpha}(\overline{\mcr_{kh}})$ and
\be U_x^2(x,0)+ V_x^2(x,0)\neq
0\qquad\text{for all }x\in\bdr,\label{mapbd}
\ee
\be \text{the mapping $x\mapsto (U(x,0),\,V(x,0))$ is injective on $\bdr$},\label{inj}\ee
\be \mcs=\{(U(x,0),\,V(x,0)):x\in\bdr\}\quad\hbox{is the upper boundary of}\ \Omega\,.\label{S}\ee
We define
\begin{equation}\label{vV}
v(x)=V(x,0)\quad\text{for all }x\in\R \quad \text{and } h=[v],
\end{equation}
where throughout this paper $[f]$ denotes the mean over one period of a $2\pi$-periodic function $f$. Note that $V$ may be recovered uniquely from $v$  as the solution of
\begin{subequations}\label{VV}
\begin{align}
&\Delta V =0\quad\hbox{in}\quad \mcr_{kh},\\[0.1cm]
&V(x,0)=v(x),\qquad x \in \R,\\[0.1cm]
&V(x,-kh)=0,\qquad x \in \R\,.
\end{align}
\end{subequations}
Of course, $U$ is a harmonic conjugate of $-V$ on $\mcr_{kh}$ and, moreover, by (\ref{H}) with $d=kh$, we have that, up to an additive constant that can be neglected,
\begin{equation}\label{U}
U(x,0)=\frac{x}{k}+ \Big(\mcc_{kh}(v-h)\Big)(x),\quad x \in \bdr\,.
\end{equation}
Thus $U_x(x,0)=V_y(x,0)=\frac1k + \C_{kh}(v')$.

Let $\Omega$ be any $L$-periodic strip-like domain of class $C^{1,\alpha}$, with
the real axis as its lower boundary. Classical elliptic theory ensures that (\ref{g1})--(\ref{g4})
has a unique solution. Thus, one may regard
 (\ref{fs})--(\ref{g}) as the problem of finding a domain $\Omega$ for which the corresponding solution  of (\ref{g1})--(\ref{g4})
satisfies also (\ref{g2}). In what follows, we aim to recast this problem in terms of the function $v$ of
one variable defined above. If  $\psi$ is the unique solution of (\ref{g1})--(\ref{g4}),
we define $\xi:\mcr_{kh}\to \R$ by
\begin{equation}\label{x}
\xi(x,y)=\psi(U(x,y),V(x,y)),\qquad (x,y)\in\mcr_{kh}\,,
\end{equation}
and $\zeta:\mcr_{kh}\to \R$ by
\begin{equation}\label{z}
\zeta=\xi+\,m-\,\frac{\Upsilon}{2}\,V^2\,.
\end{equation}
Then we may recast (\ref{g1})--(\ref{g4}) as
\begin{subequations}\label{gc}
\begin{align}
& \Delta\zeta=0\quad\hbox{in}\quad\mcr_{kh},\label{gc1}\\[0.1cm]
& \zeta(x,0)=m\,-\,\displaystyle\frac{\Upsilon}{2}\,v^2(x)\quad\hbox{for all}\quad x \in \R,\label{gc3}\\[0.2cm]
& \zeta(x,-kh)=0\quad\hbox{for all}\quad x \in \R.\label{gc4}
 \end{align}
\end{subequations}
Moreover, by calculating $|\nabla\xi|^2 = (Q-2gY)|\nabla V|^2$, we see that
(\ref{g2}) is satisfied if and only if
\begin{equation}
(\zeta_y\,+\,\Upsilon VV_y)^2=(Q-2gV)(V_x^2+V_y^2)\quad\hbox{at}\quad (x,0) \quad\hbox{for all}\quad x \in \R.\label{gc2}
\end{equation}
					From (\ref{gc})
and the definition (\ref{G}) of the Dirichlet-Neumann operator $\mcg_{kh}$ we infer that
\begin{equation}\label{zy}
\zeta_y=\displaystyle\frac{m}{kh}\,-\,\frac{\Upsilon}{2} \,\mcg_{kh}(v^2)\quad\hbox{on}\quad y=0,
\end{equation}
while clearly
\begin{equation}\label{vy}
V_y=\mcg_{kh}(v)\quad\hbox{on}\quad y=0\,.
\end{equation}
Therefore (\ref{gc2}) may now be rewritten as
\be
\left\{\frac{m}{kh} -\Upsilon \Big(
\mathcal{G}_{kh}( v^2/2)- v \mathcal{G}_{kh}(v)\Big)\right\}^2=
(Q-2gv)\,\Big(v'^2 + \mcg_{kh}(v)^2\Big),\label{ffor}\ee
		where $[v]=h$.
Equation (\ref{ffor}) was first derived in \cite{CV},  $\Upsilon=-\gamma$, due to a
change of sign  of the two-dimensional vorticity.

The above considerations show that the free-boundary problem
(\ref{g}) leads to the problem of finding a positive
number $h$ and a function $v\in C^{1,\alpha}_{2\pi}({\mathbb R})$ which satisfy
\begin{subequations}\label{maineq}
\begin{align}&\left\{\frac{m}{kh} -\Upsilon \Big(
\mathcal{G}_{kh}( v^2/2)- v \mathcal{G}_{kh}(v)\Big)\right\}^2=
(Q-2gv)\,\Big(v'^2 + \mcg_{kh}(v)^2\Big),\label{m0}\\
& [v]=h,\label{m1}\\
&v(x)>0\quad\text{for all }x\in\bdr,\label{pos}\\
&\text{the mapping $x\mapsto \left(\frac{x}{k}+\mcc_{kh}(v-h)(x), v(x)\right)$ is
injective on $\bdr$},\label{m2}\\
& v'(x)^2+\mcg_{kh}(v)(x)^2\neq 0\quad\text{for all }x\in\bdr\,.\label{m3}
\end{align}
\end{subequations}
						Indeed, (\ref{m0}) is precisely (\ref{ffor}).
The condition (\ref{m1}) is just the definition of $h$, and
(\ref{pos}) is ensured by the assumption that $\mcs$ lies in the upper half-plane. As for (\ref{m2}) and (\ref{m3}), they
are obtained from (\ref{inj}) and (\ref{mapbd}), respectively, in view of (\ref{vy}), (\ref{vV}), (\ref{U}),
and the Cauchy-Riemann equations for the analytic function $U+iV$. Note also that the relations (\ref{zy})-(\ref{vy}), in combination with (\ref{CG}), may be written as
\begin{equation}\label{vzc}
\left\{\begin{array}{l}
V_y = \displaystyle\frac{1}{k}+\mcc_{kh}(v'),\\[0.3cm]
\zeta_y = \displaystyle\frac{m}{kh}\,-\, \displaystyle\frac{\Upsilon}{2kh}\,[v^2]\,-\,\Upsilon\,\mcc_{kh}(vv'),
\end{array}\right. \quad\hbox{on}\quad y=0\,,
\end{equation}
since $[v]=h$.  Thus (\ref{m0}) may be equivalently rewritten as
\begin{equation}\label{add}
\Big\{ \frac{m}{kh} - \frac{\Upsilon}{2kh}[v^2]-\Upsilon\mcc_{kh}(vv') +\Upsilon v\,\Big(\frac{1}{k}+\mcc_{kh}(v')\Big)\Big\}^2
\, -\,(Q-2gv)\,\Big\{(v')^2
+\Big(\frac{1}{k}+\mcc_{kh}(v')\Big)^2\Big\}=0.\end{equation}

Conversely, given a positive number $h$ and a function
$v\in C^{1,\alpha}_{2\pi}$ satisfying (\ref{maineq}) with $k=2\pi/L$, one can construct a solution of
(\ref{g}) by reversing the process which led from (\ref{g}) to (\ref{maineq}). Indeed,
let $V$ be the unique solution of (\ref{VV}). If $U:\mcr_{kh}\to\bdr$ is a harmonic function such that $U+iV$ is
holomorphic, then $U,\, V\in
C^{1,\alpha}(\overline{\mcr_{kh}})$ cf. \cite{CV}. Condition (\ref{m1}) ensures
that the first two relations in (\ref{v}) hold. In combination with (\ref{m2}) and (\ref{pos}),
we infer that the curve $\mcs$ defined by (\ref{S}) is non-self-intersecting and contained in the upper half-plane.
The map $U+iV$ is a conformal mapping from $\mcr_{kh}$ onto an $L$-periodic strip-like domain $\Om$
of conformal depth $h$. Moreover, cf. \cite{CV}, this map admits an extension as a homeomorphism
between the closures of these domains, with $\{(x,0):x\in\bdr\}$ being mapped onto $\mcs$
and $\{(x,-kh):x\in\bdr\}$ being mapped onto $\mcb$. Due to (\ref{m3}), $\mcs$ is
a $C^{1,\alpha}$ curve. If $\zeta$ is the unique solution
of (\ref{gc}), then $\zeta\in C^{1,\alpha}(\overline{\mcr_{kh}})\cap C^\infty(\mcr_{kh})$.
Defining $\xi$ by (\ref{z}), and then $\psi$ by (\ref{x}), we see that
$\psi\in C^{1,\alpha}(\overline\Om)\cap C^\infty(\Om)$
satisfies (\ref{g1})-(\ref{g4}). Finally, since (\ref{m0}) holds, we obtain that $\psi$ satisfies (\ref{g2}).

Note that the  velocity  at the location
$(X,Y)=\Big(U(x,y),\,V(x,y)\Big) \in \Omega$, where $(x,y) \in
\mcr_{kh}$, is given by
\be\label{vel}
(\psi_Y,\,-\,\psi_X)=\Big(\frac{V_x\zeta_x+V_y\zeta_y}{V_x^2+V_y^2}\,
+\,\Upsilon V,\, \frac{V_x\zeta_y-V_y\zeta_x}{V_x^2+V_y^2}\Big)
\ee
in terms of $\zeta(x,y)$ and of the conformal map $U+iV$ from $\mcr_{kh}$ to $\Omega$.
The formula (\ref{vel}) is obtained by differentiating
(\ref{x}), solving the resulting linear system for $\psi_X$ and
$\psi_Y$ and taking (\ref{z}) into account.

\subsection{The new reformulation of Riemann-Hilbert type}

In the theory of irrotational waves of infinite depth, it is known \cite{ST} that the equation analogous to (\ref{add}) admits an equivalent reformulation, via Riemann-Hilbert theory.
We now derive a new reformulation in the present setting as well, although the argument is more intricate
and the new equation is coupled to a scalar constraint.
The new formulation is  the system
\begin{subequations}\label{sys}
\begin{align}\label{vara}
&\displaystyle\mcc_{kh}\Big( (Q-2gv-\Upsilon^2 v^2)\,v'\Big) +
(Q-2gv-\Upsilon^2 v^2)\,\Big( \displaystyle\frac{1}{k} +
\mcc_{kh}(v')\Big)\\
&\qquad\qquad
- \displaystyle\,2\Upsilon\,v \,\Big( \frac{m}{kh} - \frac{\Upsilon}{2kh}\,[v^2] - \Upsilon\,\mcc_{kh}(vv')\Big) -
\frac{Q-2\Upsilon m -2gh}{k} + 2g\,[v\,\mcc_{kh}(v')]=0\,,\nonumber
\end{align}
\begin{align}\label{meana}
&\displaystyle\left[\Big\{ \frac{m}{kh} - \frac{\Upsilon}{2kh}[v^2]-\Upsilon\mcc_{kh}(vv') +\Upsilon v\,\Big(\frac{1}{k}+\mcc_{kh}(v')\Big)\Big\}^2\right] \qquad\qquad\qquad\qquad\qquad\\
&\qquad\qquad -\displaystyle\left[(Q-2gv)\,\Big\{(v')^2
+\Big(\frac{1}{k}+\mcc_{kh}(v')\Big)^2\Big\}\right]=0\,.\nonumber
\end{align}
\end{subequations}
for unknowns $(m,Q, v)\in\R\times\R\times C^{1,\alpha}_{2\pi}(\R)$,
Recall that $m \in \bdr$ is the relative mass flux,
$h=[v]>0$ is the conformal mean depth of
the fluid domain $\Omega$, $k>0$ is the wave number corresponding to the wave
period $L=2\pi/k$, and $Q$ is the total head.
The square bracket $[\ \ ]$ means the average over a period, so \eqref{meana} is a scalar equation.

\begin{theorem}[Equivalence of Formulations]      \label{equiv}
Let $(m,Q, v) \in \R\times\R\times C^{1,\alpha}_{2\pi}(\R)$, with $[v]=h$.

(i) If (\ref{sys}) holds, then (\ref{add}) holds.

(ii) If (\ref{add}) holds and, in addition,
\be V_x^2+V_y^2\neq 0\quad\text{in }\overline\mcr_{kh},\label{noz}\ee
where $V$ is the solution of (\ref{VV}),  then (\ref{sys}) holds.
\end{theorem}

\begin{proof}[Proof of Theorem \ref{equiv}]
It is convenient to begin by introducing the following general notation.
A $2\pi$-periodic function $z+iw:\R\to\bdc$ is said to {\it belong to the class} ${A}^{p,\alpha}_{d}$
for some integer $p\geq 0$, $\alpha\in (0,1)$ and $d>0$,
if $[w]=0$ and there exists a analytic function $Z+iW: \mcr_d\to \bdc$ such that $Z,W\in C^{p,\alpha}_{2\pi}(\overline{\mcr_d})$, $W$ satisfies $(\ref{BC})$, and $Z(x,0)=z(x)$ for all $x\in \R$.
Then ${A}^{p,\alpha}_{d}$ is an algebra.
Indeed, if $Z_j + iW_j \in {A}^{p,\alpha}_{d}$ for $j=1,2$,
then the product $Z+iW=(Z_1+iW_1)(Z_2+iW_2)$ is harmonic due to the Cauchy-Riemann equations and
the fact that each component is harmonic.
The Cauchy-Riemann equation $Z_x=W_y$ implies that $[W]$ is independent of $y$.
But $[W(\cdot,-d)]=0$ so it follows that $[w]=0$ as well.  Thus ${A}^{p,\alpha}_{d}$ is an algebra.
The discussion in Appendix A shows that
\be
z+iw\in {A}^{p,\alpha}_{d} \Leftrightarrow z=[z]+\mcc_d(w).\label{htrans}\ee

Now for any $(m,Q, v)\in\R\times\R\times C^{1,\alpha}_{2\pi}(\R)$ with $[v]=h$, let $V$ be
the solution of (\ref{VV}) and $\zeta$ the solution of (\ref{gc}). In particular,
\begin{equation}\label{zeta_x}
\zeta_x(x,0)=-\Upsilon V(x,0) V_x(x,0)\quad\hbox{for all }\quad x\in\R
\end{equation}
and $\zeta_y$ is given by \eqref{vzc}.
Also, recall from the previous section that (\ref{add}) is merely (\ref{gc2}).
As a consequence (\ref{meana}) is equivalently expressed as
\begin{equation}
[ (Q-2gV(\cdot,0))(V_x^2(\cdot,0)+V_y^2(\cdot,0))-(\zeta_y(\cdot,0)\,
+\,\Upsilon V(\cdot,0)V_y(\cdot,0))^2]=0.\label{megc2}
\end{equation}
				 Notice that
\begin{align}& V_y(\cdot,0)+iV_x(\cdot,0)   \in A^{0,\alpha}_{kh},\label{vd}\\
& \zeta_y(\cdot,0)+i\zeta_x(\cdot,0)   \in A^{0,\alpha}_{kh},\label{zd}\\
& \{\zeta_y^2-\zeta_x^2+2i\zeta_y\zeta_x\}\big|_{\displaystyle(x,0)} \in A^{0,\alpha}_{kh}.\label{zd2}
\end{align}

We now claim that (\ref{vara}) may be expressed in an equivalent way as
\be \label{rhh}\left(x\mapsto \{(Q-2gV-\Upsilon^2V^2)(V_y-iV_x)-2\Upsilon \zeta_yV\}\big|_{\displaystyle(x,0)}\right)\in A^{0,\alpha}_{kh}.\ee
To justify this claim, let us consider the function in \eqref{rhh} in more detail.
Its imaginary part is
\begin{equation}\label{w}
w=-(Q-2gv-\Upsilon^2v^2)\,v'\,,
\end{equation}
which satisfies $[w]=0$.  Because of \eqref{vzc}, its real part is
\be \label{zed}z=(Q-2gv-\Upsilon^2v^2)\, \Big( \displaystyle\frac{1}{k}+\mcc_{kh}(v')\Big) \\[0.3cm]
-\,2\Upsilon\,v\, \Big(\displaystyle\frac{m}{kh}- \displaystyle\frac{\Upsilon}{2kh}\,[v^2]
-\Upsilon\,\mcc_{kh}(vv') \Big),\ee
which has the average
$$[z]=\frac{Q-2gh-2\Upsilon m}{k} -2g\,[v\,\mcc_{kh}(v')] - \Upsilon^2\,[v^2\,\mcc_{kh}(v')] +
2\Upsilon^2\,[v\,\mcc_{kh}(vv')].$$
But $f \mapsto \mcc_{kh}(f')$ being self-adjoint, we have
\begin{eqnarray}\label{aux}
[v^2\,\mcc_{kh}(v')] &=& \displaystyle\frac{1}{2\pi}\,\int_{-\pi}^{\pi} v^2\,\mcc_{kh}(v')\,dx =
\displaystyle\frac{1}{2\pi}\,\int_{-\pi}^{\pi} v\,\mcc_{kh}\Big((v^2)'\Big)\,dx\\[0.3cm]
&=& \displaystyle\frac{1}{\pi}\,\int_{-\pi}^{\pi} v\,\mcc_{kh}(vv')\,dx = 2\,[v\,\mcc_{kh}(vv')] \nonumber
\end{eqnarray}
so that
\begin{equation}\label{alpha}
[z]=\frac{Q-2gh-2\Upsilon m}{k} -2g\,[v\,\mcc_{kh}(v')].
\end{equation}
According to (\ref{htrans}), the statement (\ref{rhh}) is equivalent to $z=[z]+\C_d(w)$.
We express $z$ by (\ref{zed}), $[z]$ by (\ref{alpha}) and $w$ by (\ref{w}).
Then all the terms in the equation that result from the equation $z=[z]+\C_d(w)$ are compared with the
terms in (\ref{vara}).  This leads to the conclusion that (\ref{rhh}) is equivalent to (\ref{vara}), as we claimed.

$(i)$
Suppose that (\ref{vara}) and (\ref{meana}) hold. Then, as we just pointed out,
(\ref{rhh}) and (\ref{megc2}) hold.
Since  $A^{0,\alpha}_{kh}$ is an algebra, it follows from (\ref{rhh}) and (\ref{vd}) by multiplying by $V_y+iV_x$  that
\be
\left(x\mapsto \{(Q-2gV-\Upsilon^2V^2)(V_x^2+V_y^2)
-2\Upsilon \zeta_yV(V_y+iV_x)\}\big|_{\displaystyle(x,0)}\right)\in A^{0,\alpha}_{kh}.         \label{qel}\ee
Upon using (\ref{zeta_x}), this may be rewritten as
\be
 \label{rhhp}\left(x\mapsto \{(Q-2gV-\Upsilon^2V^2)(V_x^2+V_y^2)
 -2\Upsilon\zeta_y V V_y+2i\zeta_x\zeta_y\}\big|_{\displaystyle(x,0)}\right)\in A^{0,\alpha}_{kh}.   \ee
It now follows, taking (\ref{zd2}) into account, that
\be
 \label{dhp}\left(x\mapsto \{(Q-2gV-\Upsilon^2V^2)(V_x^2+V_y^2)
 -2\Upsilon\zeta_y V V_y- (\zeta_y^2-\zeta_x^2)\}\big|_{\displaystyle(x,0)}\right)\in A^{0,\alpha}_{kh}.   \ee
Since this function is real, it follows from (\ref{htrans}) that it must be a constant, a fact which may be written, upon using also (\ref{zeta_x}),
as
\be \label{cdg}\{(Q-2gV)(V_x^2+V_y^2)-(\zeta_y+\Upsilon VV_y)^2\}\big|_{\displaystyle(x,0)}\equiv C_0.\ee
It now follows from (\ref{megc2}) that $C_0=0$, so that (\ref{gc2}), and therefore (\ref{add}) holds.

$(ii)$
Conversely, suppose that (\ref{add}) holds. Then (\ref{meana}) is obtained by taking averages.
Also (\ref{add}) is exactly the same as (\ref{cdg}) with $C_0=0$.
This implies that (\ref{dhp}) holds, then, in view of (\ref{zd2}), that (\ref{rhhp}) holds, and then, upon using (\ref{zeta_x}), that (\ref{qel}) holds. It is a consequence of (\ref{noz}) that
\[\left(x\mapsto \{V_y(x,0)+iV_x(x,0)\}^{-1}\right)\in A^{0,\alpha}_{kh}.\]
Again using the fact that $A^{0,\alpha}_{kh}$ is an algebra, we deduce from (\ref{qel}) that (\ref{rhh}) holds, which means, as proved earlier, that (\ref{vara}) holds.
\end{proof}

\begin{remark}{\rm
Let us make an observation about the statement (\ref{rhh}), which is at the heart of the proof of Theorem 1. If we denote, for all $x\in \R$,
\begin{eqnarray*}
F(x) &=& V_y(x,0)+iV_x(x,0), \\
G(x) &=& \{(Q-2gV-\Upsilon^2V^2)(V_y-iV_x)-2\Upsilon\zeta_yV\}\big|_{\displaystyle(x,0)},\\
a(x) &=& Q-2gV(x,0)-\Upsilon^2V^2(x,0), \\
b(x) &=& -2\Upsilon V(x,0)\zeta_y(x,0),
\end{eqnarray*}
then $F, G\in A^{0,\alpha}_{kh}$, $a,b$ are real-valued functions, and (\ref{rhh}) takes the form
$a \overline F +b= G$, where $\overline F$ denotes the complex conjugate of $F$. An equation of this form, but with $a,b$ given, and $F,G$ to be determined, is called a \emph{Riemann-Hilbert problem}. In our problem, however, $a,b, F, G$ are coupled, and in our proof of Theorem 1 we made no use of any general result of Riemann-Hilbert theory. }
\end{remark}

\subsection{Variational structure of the new formulation} Given $v \in C^{2,\alpha}_{2\pi}(\R)$,
we separate the average $h$ of $v$ by setting
\begin{equation}\label{vw}
v=w+h\quad\hbox{with}\quad [w]=0\,.
\end{equation}
With this notation we introduce the functional
\begin{align}\label{ev}
\Lambda(w,h)&=\displaystyle\int_{-\pi}^\pi\Big(Q\,v\,-\,g\,v^2-\,\frac{\Upsilon^2}{3}\,v^3\Big)\,\Big(\frac{1}{k}\,+\,\mcc_{kh}(v')\Big)\,dx \nonumber\\&\quad
+\,\displaystyle\int_{-\pi}^\pi\Big(m\,-\,\frac{\Upsilon}{2}\,v^2\Big)\,\Big(\frac{m}{kh}\,-\,\frac{\Upsilon}{2kh}\,[v^2]\,
-\,\Upsilon\,\mcc_{kh}(vv')\Big)\,dx\,,
\end{align}
of which the domain of definition is the space $\mathcal S$ of pairs $(w,h)$
with $h>0$ and $w \in C^{2,\alpha}_{2\pi,\circ}(\R)$, where the subscript $\circ$ indicates mean zero.
Taking variations of $\Lambda$ with respect
to the function $w$ and with respect to the parameter $h$, we will obtain the reformulation of the governing equations
(2.22). While the particular form of the functional $\Lambda$ in (\ref{ev}) may appear at this point unmotivated, the
reason for considering it is that, as we show in full detail in Appendix B.2, this form is obtained by transforming via
conformal mapping the natural energy associated to the flow in the physical plane. In this section we
content ourselves with merely deriving the equation for the critical points of the functional $\Lambda$.

\begin{theorem}\label{t2}
Any critical point of $\Lambda$ in the space $\mathcal S$ satisfies the equation (\ref{vara})
as well as the constraint (\ref{meana}).
\end{theorem}

\begin{proof}  We compute the first variation of $\Lambda$ at $v$ by considering in turn
the variations of $w$ and of $h$.
First,
for $\varphi$ smooth, $2\pi$-periodic and
with $[\varphi]=0$, we compute from (\ref{ev}) the variation
\[\displaystyle\frac{\delta \Lambda}{\delta w}\,(w,h)\,\varphi=\displaystyle\lim_{\varepsilon \to 0}\,\frac{\Lambda(w+\varepsilon\varphi,h)-\Lambda(w,h)}{\varepsilon}\]
with respect to $w$ as
$$\begin{array}{r}
\displaystyle\frac{\delta \Lambda}{\delta w}\,(w,h)\,\varphi =\displaystyle\int_{-\pi}^\pi\Big(Q\,v\,-\,g\,v^2\,-\,\frac{\Upsilon^2}{3}\,v^3\,\Big)\,\mcc_{kh}(\varphi')\,dx +\displaystyle\int_{-\pi}^\pi\Big(Q\,-\,2\,g\,v\,-\,\Upsilon^2\,v^2\Big)\,\varphi\,\Big(\frac{1}{k}+\mcc_{kh}(v')\Big)\,dx\\
 +\displaystyle\int_{-\pi}^\pi\Big(m-\frac{\Upsilon}{2}\,v^2\Big)\,\Big(-\frac{\Upsilon}{kh}[v\varphi]-\Upsilon\mcc_{kh}((v\varphi)')\Big)\,dx
 -\Upsilon\displaystyle\int_{-\pi}^\pi v \varphi\,\Big(\frac{m}{kh}\,-\,\frac{\Upsilon}{2kh}\,[v^2]\,
-\Upsilon\mcc_{kh}(vv')\Big)\,dx\,.
\end{array}$$
						Using the fact
that $f \mapsto \mcc_{kh}(f')$ is self-adjoint, this may be written as
$$\begin{array}{ccc}
\displaystyle\frac{\delta \Lambda}{\delta w}\,(w,h)\,\varphi
&=&\displaystyle\int_{-\pi}^\pi\,\mcc_{kh}\Big(\Big(Q-2gv-\Upsilon^2v^2\Big)\,v'\Big)\varphi \,dx
+ \,\displaystyle\int_{-\pi}^\pi\Big(Q-2gv-\Upsilon^2v^2\Big)\Big(\frac{1}{k}+\mcc_{kh}(v')\Big)\varphi\,dx\\
&&\quad
-\,\displaystyle\frac{\Upsilon}{kh}\Big(m-\frac{\Upsilon}{2}\,[v^2]\Big)\,\displaystyle\int_{-\pi}^\pi v\varphi\,dx\,+\,\Upsilon^2\,\int_{-\pi}^\pi v\,\mcc_{kh}(vv')\varphi\,dx\\
&&\quad -\,\Upsilon\,\displaystyle\,\displaystyle\int_{-\pi}^\pi\, v\Big(\frac{m}{kh}\,-\,\frac{\Upsilon}{2kh}\,[v^2]\,-\,\Upsilon\,\mcc_{kh}(vv')\Big)\varphi\,dx\,.
\end{array}$$
						Therefore
\begin{equation}
\displaystyle\frac{\delta \Lambda}{\delta w}\,(w,h)\,\varphi= \int_{-\pi}^\pi \eta\,\varphi\,dx,\label{varw}
\end{equation}
where
\begin{align}\label{defeta}
\eta&=\displaystyle\mcc_{kh}\Big(\Big(Q\,-\,2gv\,-\,\Upsilon^2\, v^2\Big)\,v'\Big) +
\Big(Q\,-\,2gv\,-\,\Upsilon^2\, v^2\Big)\,\Big( \displaystyle\frac{1}{k} +
\mcc_{kh}(v')\Big)\nonumber\\&\qquad
-\displaystyle\,2\Upsilon\,v \,\Big( \frac{m}{kh} - \frac{\Upsilon}{2kh}\,[v^2] - \Upsilon\,\mcc_{kh}(vv')\Big).
\end{align}
Now, at a critical point $v=w+h$ of $\Lambda(w,h)$, we of course have
$\displaystyle\frac{\delta \Lambda}{\delta w}\,(w,h)\,\varphi=0$
for all smooth $2\pi$-periodic functions $\varphi$ with $[\varphi]=0$, so that
(\ref{varw}) implies that $\eta$ is a constant.

Next, we compute the variation of $\Lambda$  with respect to $h$. For this purpose, we notice that for
any function $f \in C^{2,\alpha}_{2\pi}(\R)$ we have
\begin{equation}\label{varc}
\frac{d}{dh}\,\Big(\mcc_{kh}\,(f')\Big)=-\,k\,f'' \,-\,k\,\mcc^2_{kh}\,(f'')\,.
\end{equation}
Indeed, writing
$$f(x)=[f]+\sum_{n=1}^\infty \Big(a_n\cos(nx)+b_n\sin(nx)\Big),$$
we have
$$\mcc_{kh}\,(f')=\sum_{n=1}^\infty n\coth(nkh)\,\Big(a_n\cos(nx)+b_n\sin(nx)\Big),$$
so that
$$
\displaystyle\frac{d}{dh}\,\Big(\mcc_{kh}\,(f')\Big) =
\displaystyle\sum_{n=1}^\infty kn^2(1\,  -  \,\coth^2(nkh))\Big( a_n\cos(nx)+b_n\sin(nx)\Big)\\
=-\,k\,f'' \,-\,k\,\mcc^2_{kh}\,(f'').$$
						Expressing
$\Lambda$ explicitly in terms of $w$, the functional $\Lambda(w,h)$ equals
$$\displaystyle\,\int_{-\pi}^\pi\Big(\,Q\,w\,+\,Q\,h-\,g\,w^2\,-\,2g\,w\,h\,-\,g\,h^2\,-\,\displaystyle\frac{\Upsilon^2}{3}\,w^3\,-\,\Upsilon^2\,w^2\,h\,-\,\Upsilon^2\,w\,h^2
\,-\,\displaystyle\frac{\Upsilon^2}{3}\,h^3
\Big)\,\Big(\frac{1}{k}+\mcc_{kh}(w')\Big)\,dx$$
$$+\,\displaystyle\,\int_{-\pi}^\pi\Big(m-\frac{\Upsilon}{2}\,w^2\,-\,\Upsilon\,w\,h\,-\,\frac{\Upsilon}{2}\,h^2\Big)\,
\Big(\frac{m}{kh}\,-\,\frac{\Upsilon}{2kh}\,[w^2]\,-\,\frac{\Upsilon\,h}{2k}\,
-\,\Upsilon\,\mcc_{kh}(ww')\,-\,\Upsilon\,h\,\mcc_{kh}(w')\Big)\,dx\,$$
By
 (\ref{varc}) we now compute
\begin{align}\label{ew2}
\displaystyle\frac{\delta \Lambda}{\delta h}\,(w,h)&=\displaystyle\,\int_{-\pi}^\pi(Q-2gv-\Upsilon^2v^2)\,\Big(\frac{1}{k}+\mcc_{kh}(v')\Big)\,dx \nonumber\\
&\quad\,+\,\displaystyle\,\int_{-\pi}^\pi\Big(Q\,v\,-\,g\,v^2\,-\,\frac{\Upsilon^2}{3}\,v^3\,\Big)\,\Big(-k\,v''\,-\,k\,\mcc^2_{kh}(v'')\Big)\,dx \nonumber\\
&\quad -\,\displaystyle{\Upsilon}\,\int_{-\pi}^\pi v\,\,\Big(\frac{m}{kh}\,-\,\frac{\Upsilon}{2kh}\,[v^2]\,-\,\Upsilon\,\mcc_{kh}(vv')\Big)\,dx \\
&\quad +\,\displaystyle\,\int_{-\pi}^\pi \Big(m-\frac{\Upsilon}{2}\,v^2\Big)\,\Big\{-\frac{m}{kh^2}\,+\,\frac{\Upsilon}{2kh^2}\,[v^2]\,-\,\frac{\Upsilon}{k}-\,\Upsilon\,\mcc_{kh}(v') \nonumber\\
&\hskip 7cm \,+\,\Upsilon k\,(vv')'\,+\Upsilon k\,\mcc_{kh}^2\Big((vv')'\Big)\Big\}\,dx\,, \nonumber
\end{align}
since the terms $\pm \Upsilon hk(1+\C_{kh}^2) v''$ have canceled each other.
In the second and fourth integral in (\ref{ew2}) we integrate once by parts the terms involving $v''$ and $(vv')'$,
and for the terms
involving $\mcc_{kh}^2(v'')$ and $\mcc_{kh}^2\big((vv')'\big)$ we use the fact that $f \mapsto \mcc_{kh}(f')$ is
self-adjoint, to express (\ref{ew2}) as
\begin{align}\label{ew3}
\displaystyle\frac{\delta \Lambda}{\delta h}\,(w,h)&=\displaystyle\,\int_{-\pi}^\pi(Q-2gv-\Upsilon^2v^2)\,\Big(\frac{1}{k}+\mcc_{kh}(v')\Big)\,dx \,+\,\displaystyle k\,\int_{-\pi}^\pi (Q-2gv-\Upsilon^2v^2)\,(v')^2\,dx \nonumber\\
&\quad -\,\displaystyle k\,\int_{-\pi}^\pi\mcc_{kh}\Big((Q-2gv-\Upsilon^2v^2)v'\Big)\mcc_{kh}(v')\,\,dx\nonumber\\
&\quad \,
-\,\displaystyle \Upsilon\,\int_{-\pi}^\pi v\,\,\Big(\frac{m}{kh}\,-\,\frac{\Upsilon}{2kh}\,[v^2]\,-\,\Upsilon\,\mcc_{kh}(vv')\Big)\,dx
\\&\qquad
+\,\displaystyle\int_{-\pi}^\pi \Big(m-\frac{\Upsilon}{2}\,v^2\Big)\,\Big\{-\frac{m}{kh^2}+\frac{\Upsilon}{2kh^2}\,[v^2]-\frac{\Upsilon}{k}
-\Upsilon\,\mcc_{kh}(v')
\Big\}\,dx\,\nonumber\\&\qquad
+\,\displaystyle k\Upsilon^2\,\int_{-\pi}^\pi v^2(v')^2\,dx\,-\,\displaystyle k\Upsilon^2\,\int_{-\pi}^\pi \Big( \mcc_{kh}(vv')\Big)^2\,dx\,.\nonumber
\end{align}
			We further use (\ref{defeta}) to substitute the third integral in (\ref{ew3}) by
\begin{align*}
 &-k\int_{-\pi}^\pi \eta \,\mcc_{kh}(v')\,dx+ \displaystyle k\,\int_{-\pi}^\pi (Q-2gv-\Upsilon^2v^2)\,\Big(\frac{1}{k}+\mcc_{kh}(v')\Big)\,\mcc_{kh}(v')\,dx\,
 \\
 &\qquad -2\,k\Upsilon\,\displaystyle\int_{-\pi}^\pi v\,\Big\{\frac{m}{kh}\,-\,\frac{\Upsilon}{2kh}\,[v^2]\,- \,\Upsilon\,\mcc_{kh}(vv')\Big\}
 \,\mcc_{kh}(v')\,dx.
\end{align*}
Consequently  (\ref{ew3}) takes the form
\begin{eqnarray*}
\displaystyle\frac{\delta \Lambda}{\delta h}\,(w,h) &=& -k\,\displaystyle\int_{-\pi}^\pi \eta \,\mcc_{kh}(v')\,dx
+\displaystyle k\,\int_{-\pi}^\pi(Q-2gv-\Upsilon^2v^2)\,\Big\{(v')^2\,+\,\Big(\frac{1}{k}+\mcc_{kh}(v')\Big)^2\Big\}\,dx \\[0.1cm]
&&
-\,2 k\Upsilon\,\displaystyle\int_{-\pi}^\pi v\,\Big\{\frac{m}{kh}\,-\,\frac{\Upsilon}{2kh}\,[v^2]\,- \,\Upsilon\,\mcc_{kh}(vv')\Big\}
 \,\mcc_{kh}(v')\,dx\,\\[0.1cm]
&&  -\,\displaystyle \Upsilon\,\int_{-\pi}^\pi v\,\,\Big(\frac{m}{kh}\,-\,\frac{\Upsilon}{2kh}\,[v^2]\,-\,\Upsilon\,\mcc_{kh}(vv')\Big)\,dx\\[0.1cm]
&&
+\,\displaystyle\,\int_{-\pi}^\pi \Big(m-\frac{\Upsilon}{2}\,v^2\Big)\,\Big\{-\frac{m}{kh^2}\,+\,\frac{\Upsilon}{2kh^2}\,[v^2]\,-\,\frac{\Upsilon}{k}\Big\}\,dx
\\[0.2cm]
&&
+\,\displaystyle\frac{\Upsilon^2}{2}\,\int_{-\pi}^\pi v^2\,\mcc_{kh}(v')\,dx
+\,\displaystyle k\Upsilon^2\,\int_{-\pi}^\pi v^2(v')^2\,dx\,-\,\displaystyle k\Upsilon^2\,\int_{-\pi}^\pi \Big( \mcc_{kh}(vv')\Big)^2\,dx.
\end{eqnarray*}
One can easily check that the fifth integral term in the above expression is precisely
$$-\,k\,\int_{-\pi}^\pi \Big\{ \frac{m}{kh}\,-\,\frac{\Upsilon}{2kh}\,[v^2]\Big\}^2\,dx\,-\,\Upsilon\,\int_{-\pi}^\pi v\,\Big(\frac{m}{kh}\,-\,\frac{\Upsilon}{2kh}\,[v^2]\Big)\,dx$$
while the self-adjointness of $f \mapsto \mcc_{kh}(f')$ yields
$$\frac{\Upsilon^2}{2}\,\int_{-\pi}^\pi v^2\,\mcc_{kh}(v')\,dx={\Upsilon}^2\,\int_{-\pi}^\pi v\,\mcc_{kh}(vv')\,dx\,.$$
Since $[\mcc_{kh}(vv')]=0$, we get
\begin{eqnarray*}
\displaystyle\frac{\delta \Lambda}{\delta h}\,(w,h) &=& -k\int_{-\pi}^\pi \eta \,\mcc_{kh}(v')\,dx\,
+\displaystyle k\,\int_{-\pi}^\pi(Q-2gv-\Upsilon^2v^2)\,\Big\{(v')^2\,+\,\Big(\frac{1}{k}+\mcc_{kh}(v')\Big)^2\Big\}\,dx \\
&&\quad \,-\,\displaystyle k\,\int_{-\pi}^\pi\Big\{ \frac{m}{kh}\,-\,\frac{\Upsilon}{2kh}\,[v^2]\,-\,\Upsilon\,\mcc_{kh}(vv')\Big\}^2\,dx\,
+\,\displaystyle k\Upsilon^2\,\int_{-\pi}^\pi v^2(v')^2\,dx\\
&&\quad
-\,2k\Upsilon\,\displaystyle\int_{-\pi}^\pi v\,\Big\{\frac{m}{kh}\,-\,\frac{\Upsilon}{2kh}\,[v^2]\,- \,\Upsilon\,\mcc_{kh}(vv')\Big\}
 \,\Big(\frac{1}{k}\,+\,\mcc_{kh}(v')\Big)\,dx\,.
\end{eqnarray*}
It follows that
\begin{align}\label{vh}
&\displaystyle\frac{\delta \Lambda}{\delta h}\,(w,h)=-k\int_{-\pi}^\pi \eta \,\mcc_{kh}(v')\,dx
 + k\displaystyle\int_{-\pi}^\pi(Q-2gv)\,\Big\{(v')^2\,+\,\Big(\frac{1}{k}+\mcc_{kh}(v')\Big)^2\Big\}\,dx\\
&\qquad-\,k\, \displaystyle\int_{-\pi}^\pi\Big\{ \frac{m}{kh}\,-\,\frac{\Upsilon}{2kh}\,[v^2]\,-\,\Upsilon\,\mcc_{kh}(vv')\,
+\,\Upsilon\,v\,\Big(\frac{1}{k}+\mcc_{kh}(v')\Big)\Big\}^2\,dx\,,\nonumber
\end{align}
where $\eta$ is given by (\ref{defeta}).

At a critical point $v=w+h$ of $\Lambda(w,h)$ we also have
$\displaystyle\frac{\delta \Lambda}{\delta h}\,(w,h)=0$.
We have already shown that $\eta$ is a constant.
The exact value of the constant is computed by taking averages in (\ref{defeta}). Using
the identity (\ref{aux}), we obtain precisely (\ref{vara}).
Due to
the fact that $[\mcc_{kh}(v')]=0$, the condition $\displaystyle\frac{\delta \Lambda}{\delta h}\,(w,h)=0$
is easily seen to be equivalent to (\ref{meana}). This completes the proof of Theorem \ref{t2}.
\end{proof}

\begin{remark}{\rm
Let us now discuss the special case $\Upsilon=0$ (irrotational flow). For $k=1$ and $w=v-h$
 (\ref{vara}) reduces to
\begin{equation}\label{var0}
\mu\,\mcc_{h}(w')=w\,+\,w\,\mcc_{h}(w') +\mcc_{h}(ww')\,-\, [w\,\mcc_{h}(w')]
\end{equation}
with $\mu=\displaystyle\frac{Q-2gh}{g}$.
Note that for $\Upsilon=0$ the constant $m$ disappears from (\ref{vara}).
In this case, the constraint (\ref{meana}) merely specifies the value of $m$. Setting
$$\beta=[w\,\mcc_h w'],\quad \tilde{v}=w-\beta,$$
we transform (\ref{var0}) into
\begin{equation}\label{var00}
\tilde{\mu} \,\mcc_h(\tilde{v}')=\tilde{v}+\tilde{v}\,\mcc_h(\tilde{v}') + \mcc_h(\tilde{v}\tilde{v}')
\end{equation}
with $\tilde{\mu}=\mu -2\beta$.
We contrast this formula with the formula for
irrotational steady waves in water of infinite depth,  mentioned in the introduction, namely,
\begin{equation}\label{varid}
\tilde{\mu} \,\mcc(\tilde{v}')=\tilde{v}+\tilde{v}\,\mcc(\tilde{v}') + \mcc(\tilde{v}\tilde{v}'),
\end{equation}
where $\mcc$ is the standard Hilbert transform and $\tilde{\mu}>0$ is a constant
(see \cite[Equation (1.8)]{BDT1}).
Note the direct analogy between (\ref{var00}) and (\ref{varid}).
}
\end{remark}

\section{Existence theory}
\nequation

A powerful approach for establishing the existence of travelling water waves relies on bifurcation theory, tailored for the
study in-the-large of parameter-dependent families of solutions. The existence of waves of {\it small} amplitude is
addressed by means of local bifurcation theory: identifying suitable parameters that, by crossing through certain
thresholds, lead to sudden changes of the corresponding flat-surface flows into genuine waves. Global bifurcation
theory uses topological methods to show that these families of perturbations of simple solutions belong to connected sets
of solutions of {\it global} extent. Since global bifurcation is not a perturbative approach, exploiting instead the topological structure
of the solution set, this global continuum provides wave patterns that are not small disturbances
of flows with a flat-free surface.

In this section we study the existence of solutions $(m,Q,v)$ of (\ref{sys}) in the space $\R\times \R\times C^{2, \alpha}_{2\pi, e}(\bdr)$ and such that $[v]=h$, where  $$C^{2, \alpha}_{2\pi, e}(\bdr) =\{ f\in C^{2, \alpha}_{2\pi}(\bdr):
\ f(x)=f(-x)\text{ for all }x\in\bdr\},$$
for some constant $\alpha \in (0,1)$. Our construction provides parameterized families of solutions, which when the parameter is $0$ satisfy $v\equiv h$.
The requirement that $v$ is an even function
reflects the fact that the corresponding wave profile is symmetric about the crest located at $x=0$. Symmetric
travelling periodic waves are ubiquitous in nature (see e.g. the photographs in \cite{Cb}). Moreover, in the
absence of stagnation points in the flow and for surface waves that are represented by graphs of functions, one can show
that a wave profile that is monotonic between crests and troughs has to be symmetric cf. \cite{CE0, CEW}. In view of Theorem \ref{equiv}, any solution of (\ref{sys}) satisfies also (\ref{m0})-(\ref{m1}); whether or not it will give rise to solutions of the free-boundary problem (\ref{fs})-(\ref{g}) depends on whether or not it also satisfies conditions (\ref{pos})-(\ref{m3}). The existence of solutions of (\ref{sys}) will be provided in Subsection 3.1, while in Subsection 3.2 we study to what extent they satisfy (\ref{pos})-(\ref{m3}), as well as proving further qualitative properties, such as strict monotonicity between any crest and trough.

A construction of solutions of (\ref{add}) by means of local bifurcation theory has been carried out in \cite{CV}. However, the applicability of \emph{global bifurcation theory} requires certain compactness of the nonlinear operators in question, which do not seem to be available for the equation (\ref{add}). Thus, the use of  the system (\ref{sys}) seems essential for the global theory that we develop in what follows.

Note first that (\ref{sys}) has a family of trivial solutions, for which $v\equiv h$, while $Q$ and $m$ are related by
\be Q=2gh + \Big( \frac{m}{h}+\frac{\Upsilon
h}{2}\Big)^2, \label{zar}\ee
where $m\in\R$ is arbitrary. This family represents a curve
\[{\mathcal K}_{triv}:=\left\{\left(m, 2gh + \Big( \frac{m}{h}+\frac{\Upsilon
h}{2}\Big)^2, h\right): m\in\R\right\}\]
in the space $\R\times \R\times C^{2, \alpha}_{2\pi, e}(\bdr)$.
 These solutions
correspond to laminar flows in the fluid domain bounded below by the
rigid bed $\mcb$ and above by the free surface $Y=h$, with stream
function
\[\psi(X,Y)=\frac{\Upsilon}{2}Y^2+\left(\frac{m}{h}-\frac{\Upsilon
h}{2}\right)Y-m, \qquad X\in\bdr, 0\leq Y\leq h,\]
 velocity field
\begin{equation}\label{lb8}
(\psi_Y,-\psi_X)=\Big(\Upsilon Y+\frac{m}{h}-\frac{\Upsilon h}{2}
,0\Big), \qquad X\in\bdr, 0\leq Y\leq h,
\end{equation}
and period $L=2\pi/k$.

\begin{theorem}[Global Bifurcation]   \label{glbp}
Let $h,\,k>0$ and $\Upsilon \in \bdr$ be given. For each $n\in {\mathbb N}$, let
\begin{equation}\label{lb7}
m_{n,\pm}^* =-\frac{\Upsilon h^2}{2}+\,\frac{\Upsilon h\tanh(nkh)}{2nk}\pm
h\,\displaystyle\sqrt{\frac{
\gamma^2\tanh^2(nkh)}{4n^2k^2}+g\,\frac{\tanh(nkh)}{nk}}
\end{equation}
and
\be \label{qlk7} Q_{n,\pm}^*= 2gh + \Big( \frac{m_{n,\pm}^*}{h}+\frac{\Upsilon
h}{2}\Big)^2.\ee
First,
for any $m\in\R$ with $m\not\in\{m_{n,\pm}^*: n\in{\mathbb N}\}$, there exists a neighbourhood in $\R\times \R\times C^{2, \alpha}_{2\pi, e}(\bdr)$ of the point $(m, Q, h)$ on ${\mathcal K}_{triv}$, where $Q$ is related to $m$ by (\ref{zar}), in which the only solutions of (\ref{sys}) are those on  ${\mathcal K}_{triv}$.
Secondly,
consider the points $m_{n,\pm}^*$.
For each integer $n\ge 1$ and each choice of sign in $\pm$, there exists in the space $\R\times \R\times C^{2, \alpha}_{2\pi, e}(\bdr)$ a continuous curve
\be \mathcal{K}_{n,\pm}=\{(m(s), Q(s), v_s):s\in \R\}\ee
of solutions of (\ref{sys}) such that the following properties hold.
\begin{itemize}
\item[(i)]
$(m(0), Q(0), v_0)=(m_{n,\pm}^*, Q_{n,\pm}^*, h)$, where $m_{n,\pm}^*$ and $Q_{n,\pm}^*$ are given by (\ref{lb7}) and (\ref{qlk7});
\item[(ii)] $v_s(x) = h+s\cos(nx) + o(s)$ in $C^{2, \alpha}_{2\pi, e}(\bdr), 0<\ |s| < \varepsilon$, for some $\varepsilon>0$ sufficiently small;
\item[(iii)]
$\{(m,Q,v) \in {\mathcal W}_{n,\pm}:\ v \not\equiv h, \text{(\ref{sys}) holds}\}= \{(m(s), Q(s), v_s) :\ 0< |s| < \varepsilon \}$, for some neighbourhood ${\mathcal W}_{n,\pm}$ of $(m_{n,\pm}^*, Q_{n,\pm}^*, h)$ in $\R\times \R\times C^{2, \alpha}_{2\pi, e}(\bdr)$ and $\varepsilon>0$ sufficiently small;
\item[(iv)]
$Q(s)-2gv_s(x)>0$ for all $s, x\in \R$;
\item[(v)]
$\mathcal{K}_{n,\pm}$ has a real-analytic reparametrization locally around each of its points;
\item[(vi)]
One of the following alternatives occurs:
\begin{itemize}
\item[($\alpha$)] either
\be  \min\left\{\frac{1}{1+||(m(s), Q(s), v_s)||_{\R\times\R\times C^{2, \alpha}_{2\pi, e}}}, \min_{x\in\R}\{Q(s)-2gv_s(x)\}\right\}\to 0 \quad\text{as }s\to\pm\infty;\label{coonc}\ee
    \item[($\beta$)] or there exists $T>0$ such that $ (m(s+T), Q(s+T), v_{s+T})=(m(s), Q(s), v_s)$ for all $s\in\R$.
\end{itemize}
\end{itemize}
Moreover, for each integer $n\ge 1$ and each choice of sign in $\pm$, such a curve of solutions of (\ref{sys}) with the properties (i)-(vi) is unique (up to reparametrization).
\end{theorem}
In this theorem, (i) states where the curve begins, (ii) and (iii) describe the local behavior of the curve,
and (iv)-(vi) describe the global behavior.  The alternative ($\alpha$) means that either the curve
is unbounded in the function space or it approaches a wave of greatest height,
while the alternative ($\beta$) is that the curve forms a loop.

 Our main tool in the proof of Theorem \ref{glbp} is the following version of the global bifurcation theorem
 for real-analytic operators due to Dancer \cite{D} and improved by Buffoni and Toland
 \cite[Theorem 9.1.1]{BT}. That result is however
slightly inaccurate as stated there, and here we provide a corrected version,
slightly modified to better suit our purposes.  For a linear operator $\mcl$ between two Banach spaces,
let us denote by ${\mathcal N}(\mcl)$ its null space and by
${\mathcal R}(\mcl)$ its range.

\begin{theorem}   [Analytic Bifurcation Theory]                 \label{abstract}
Let $\mcx$ and $\mcy$
be Banach spaces, ${\mathcal O}$ be an open subset of $X$ and $F: {\mathcal O}\to \mcy$ be a real-analytic function. Suppose that
\begin{itemize}
\item[$(H_1)$] $(\lambda,0)\in {\mathcal O}$ and  $F(\lambda,0) = 0$ for all $\lambda \in \R$;
\item[$(H_2)$] for some $\lambda^*\in \R$, ${\mathcal N}(\pa_u F(\lambda^\ast,0))$ and $\mcy/{\mathcal
R}(\pa_u F(\lambda^\ast,0))$ are one-dimensional, with the nullspace
generated by $u^\ast$, and the transversality condition $\pa^2_{\lambda,u}
F(\lambda^\ast,0)\,(1,u^\ast) \not \in {\mathcal R}(\pa_u
F(\lambda^\ast,0))$ holds;
\item[$(H_3)$] $\pa_u F(\lambda, u)$ is a Fredholm operator of index zero whenever $F(\lambda, u)=0$ with  $(\lambda, u)\in {\mathcal O}$;
\item[$(H_4)$] for some sequence $({\mathcal Q}_j)_{j\in {\mathbb N}}$ of bounded closed subsets of ${\mathcal O}$ with  ${\mathcal O}= \cup_{j\in{\mathbb N}} {\mathcal Q}_j$, the set $\{(\lambda, u)\in {\mathcal O}: F(\lambda, u)=0\}\cap {\mathcal Q}_j$ is compact for each $j\in {\mathbb N}$.
\end{itemize}
Then there exists in ${\mathcal O}$ a continuous curve
${\mathcal K}=\{(\lambda(s),u(s)) :  s\in\R\}$ of solutions to $F(\lambda, u)=0$
such that:
\begin{itemize}
\item[$(C_1)$] $(\lambda(0), u(0))=(\lambda^*,0)$;
\item[$(C_2)$] $u(s) = su^\ast + o(s)$ in $\mcx,\ |s| < \varepsilon$ as $s\to 0$;
\item[$(C_3)$]
$\{(\lambda,u) \in {\mathcal W}:\ u \neq 0,\ F(\lambda,u)=0\}= \{(\lambda(s),u(s)) :\ 0< |s| < \varepsilon \}$, for some neighbourhood ${\mathcal W}$ of $(\lambda^*,0)$ and $\varepsilon>0$ sufficiently small;
\item[$(C_4)$] ${\mathcal K }$ has a real-analytic parametrization locally around each of its points;
\item[$(C_5)$] One of the following alternatives occurs:
\begin{itemize}
\item[($\alpha$)] for every $j\in\N$, there exists $s_j>0$ such that $(\lambda(s), u(s))\not\in {\mathcal Q}_j$ for all $s\in\R$ with $|s|>s_j$;
    \item[($\beta$)] there exists $T>0$ such that $(\lambda(s+T), u(s+T))=(\lambda(s), u(s))$ for all $s\in\R$.
\end{itemize}
\end{itemize}
Moreover, such a curve of solutions to $F(\lambda, u)=0$ having the properties $(C_1)-(C_5)$ is unique (up to reparametrization).
\end{theorem}

\begin{figure} \centering  \includegraphics[width=9cm]{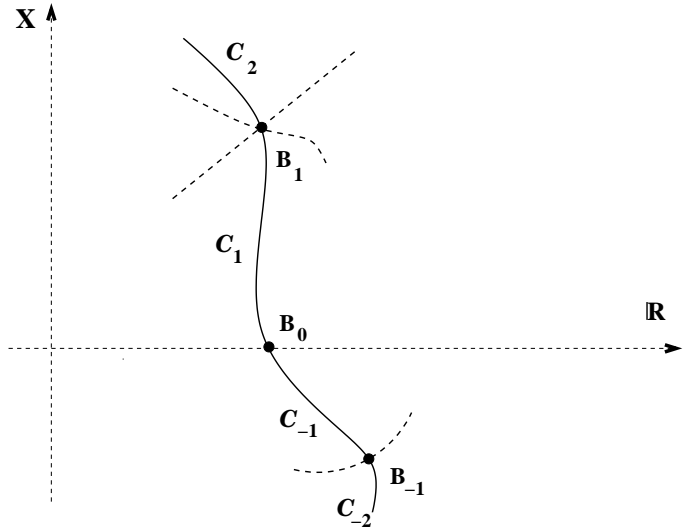}
\caption{\small The global bifurcation curve ${\mathcal K} \subset {\mathbb R} \times X$ consists of distinguished
real-analytic open arcs ${\mathcal C}_j$ that end in the branch points $B_{j-1}$ and $B_j$ if $j>0$, or in the branch points $B_j$ and
$B_{j+1}$ if $j<0$, with $B_0=(\lambda^\ast,0)$ and
${\mathcal C}_1 \cup \{B_0\} \cup {\mathcal C}_{-1}$ being the local bifurcation curve ${\mathcal K}_{loc}$. A point on ${\mathcal C}_j$
corresponds to a non-singular solution (near which the implicit function theorem applies), while
each $B_i$ arises as the unique intersection point of the closures of a finite even number of open one-dimensional
real-analytic varieties. A distinguished arc ${\mathcal C}_j$ can be uniquely continued across $B_j$ by choosing an
outgoing branch ${\mathcal C}_{j + 1}$ if $j>0$, or ${\mathcal C}_{j- 1}$ if $j<0$, so that each curve ${\mathcal C}_j \cup \{B_j\} \cup {\mathcal C}_{j + 1}$ if $j>0$, or ${\mathcal C}_j \cup \{B_j\} \cup {\mathcal C}_{j -1}$ if $j<0$, admits
a local uniformizing real-analytic parametrization near $B_j$.}
\end{figure}

 \begin{remark} {\rm The local version of Theorem \ref{abstract}, in which assumptions $(H_1)-(H_2)$ imply the existence of a  real-analytic local bifurcating curve ${\mathcal K}_{loc}=\{(\lambda(s),u(s)) :  s\in(-\varepsilon,\varepsilon)\}$ of solutions to $F(\lambda, u)=0$ with the properties $(C_1)-(C_3)$ is the real-analytic version of the standard Crandall-Rabinowitz local bifurcation theorem \cite{CR}. The curve ${\mathcal K}$ exhausts all possibilities
of adding real-analytic arcs to the local bifurcation curve ${\mathcal K}_{loc}$ in such a way that ${\mathcal K}$
has a real-analytic parametrization around each of its points (see Figure 2),
but is not necessarily a maximal connected subset of the solution set. }
 \end{remark}

\begin{remark} {\rm
We now discuss how Theorem \ref{abstract} relates to \cite[Theorem 9.1.1]{BT}.
In \cite {BT}, assumption $(H_4)$ is replaced by the slightly stronger assumption that
all bounded and closed subsets (in $\R\times X$) of $\{(\lambda, u)\in {\mathcal O}: F(\lambda, u)=0\}$ are compact.
In \cite{BT} it is proven that there does not exist any sequence $s_k\to\infty$ such that
the sequence $(\lambda(s_k), u(s_k))_{k\geq 1}$ is \emph{both} bounded \emph{and} bounded away
from the boundary of ${\mathcal O}$.
However, \cite{BT} incorrectly claims the strictly stronger statement that
\begin{itemize}
\item[$(C_5)$]$(\alpha')$ either $||(\lambda(s),u(s))||_{\R\times X}\to \infty$ as $s\to\infty$, or $\textnormal{dist}((\lambda(s),u(s)), \partial {\mathcal O})\to 0$ as $s\to\infty$.
\end{itemize}
The correct conclusion for \cite[Theorem 9.1.1]{BT} may be restated in the following way:

 As $s\to\infty,\ (\lambda(s), u(s))$ eventually leaves every bounded closed subset $B$ of ${\mathcal O}$.

\noindent Indeed, it suffices to notice that
the intersection of $B$ with $\{(\lambda, u)\in {\mathcal O}: F(\lambda, u)=0\}$ is compact.
Thus our slightly weaker assumption (H4) suffices to reach the correct conclusion by means of exactly
the same arguments as in the proof of \cite[Theorem 9.1.1]{BT}.
} \end{remark}

For the purpose of applying  Theorem \ref{abstract} to problem (\ref{sys}), it is necessary to make some simple changes of variables. Since one must necessarily have  $[v]=h$, it is natural to work with the function
\be w=v-h,\label{cjn}\ee
for which $[w]=0$. Then $w$ satisfies
\begin{subequations}\label{scb}
\begin{eqnarray}\label{varanp}
&& \ \ 2(Q-2gh)\,\mcc_{kh}(w')-\,2g \Big( \mcc_{kh}(ww')+w\,\mcc_{kh}(w')\Big)  - \Upsilon^2\,\Big( \mcc_{kh}(w^2w') +w^2\,\mcc_{kh}(w') - 2w\,\mcc_{kh}(ww')\Big)
 \\[0.1cm]
&&\qquad\qquad + \displaystyle\frac{\Upsilon^2}{kh}\,w[w^2] - \displaystyle\frac{\Upsilon^2}{k}\,w^2
+ \frac{\Upsilon^2}{k}\,[w^2]\,+2g\,[w\,\mcc_{kh}(w')] - \frac{2g}{k}\,w - \frac{2\Upsilon}{k}\Big(\frac{m}{h}+\frac{\Upsilon h}{2}\Big)\,w=0\,, \nonumber
\end{eqnarray}
and
\begin{align}\label{meannp}
&\left[\left\{\frac{1}{k}\left(\frac{m}{h}+\frac{\Upsilon h}{2}\right) -\Upsilon \left(
\frac{[w^2]}{2kh}-\frac{w}{k} +\mcc_{kh}(ww')-w\mcc_{kh}(w')\right)\right\}^2\right]=\\
&\qquad\qquad=\left[(Q-2gh-2gw)\,\left\{w'^2 +
\left(\frac{1}{k}+\mcc_{kh}(w')\right)^2\right\}\right].\non
\end{align}
\end{subequations}
We observe that, although $Q$ and $m$ are related by (\ref{zar}) for trivial solutions,
this need not be the case in general.  This
observation suggests the introduction of a new parameter
\be\mu:= Q-2gh - \Big( \frac{m}{h}+\frac{\Upsilon
h}{2}\Big)^2.\label{mu}\ee
Note also that, for the laminar flows given by (\ref{lb8}), the horizontal velocity at the free
surface is $\frac{m}{h}+\frac{\Upsilon h}{2}$. Since this expression occurs naturally in (\ref{scb}), while also in fluid
dynamics literature it is customary to identify the laminar flows at which nonlinear small-amplitude waves bifurcate
using the speed of their particles at the free surface (rather than the value of their flux), we introduce
another parameter
\be\lambda=\frac{m}{h}+\frac{\Upsilon h}{2}\label{lamb}\ee

The change of parameters $(m,Q)\mapsto (\lambda, \mu)$ given by (\ref{mu})-(\ref{lamb}) is a bijection from $\R^2$ onto itself.
In terms of the new parameters, the system (\ref{scb}) may be written as
\be F(\lambda, (\mu, w))=0,\label{eqso}\ee where $F:\R\times X\to Y$, with
 \be\label{fcs}
\mcx=\bdr\times C^{2,\alpha}_{2\pi,\circ, e}(\bdr),\qquad
\mcy= C^{1,\alpha}_{2\pi, \circ,e}(\bdr)\times \R,
\ee
where the subscripts indicate period $2\pi$,  zero average and evenness, respectively,
and where $F=(F_1,F_2)$ is given by
\begin{subequations}\label{fels}
\begin{eqnarray}\label{varan}
F_1(\lambda, (\mu, w)) &=& 2(\mu+\lb^2)\,\mcc_{kh}(w')-\,2g \Big\{ \mcc_{kh}(ww')+w\,\mcc_{kh}(w')\Big\} \nonumber \\[0.1cm]
&&\quad- \Upsilon^2\,\Big\{ \mcc_{kh}(w^2w') +w^2\,\mcc_{kh}(w') - 2w\,\mcc_{kh}(ww')\Big\} \\[0.1cm]
&&\quad + \displaystyle\frac{\Upsilon^2}{kh}\,w[w^2]  - \displaystyle\frac{\Upsilon^2}{k}\,w^2
+ \frac{\Upsilon^2}{k}\,[w^2]\,+2g\,[w\,\mcc_{kh}(w')] - \frac{2g}{k}\,w - \frac{2\lambda\Upsilon}{k}\,w\,, \nonumber
\end{eqnarray}
and
\begin{align}\label{meann}
 F_2(\lambda, (\mu, w))=\Upsilon^2\,\Big[ \Big\{ \mcc_{kh}(ww') -w\,\mcc_{kh}(w') - \displaystyle\frac{1}{k}\,w + \frac{[w^2]}{2kh}\Big\}^2\Big]
+ \frac{2(2g+\lambda\Upsilon)}{k}\,\big[w\,\mcc_{kh}(w')\big] + 2g\,[w(w')^2] \nonumber\\[0.1cm]
\quad + \,2g\,\big[w\,\big(\mcc_{kh}(w')\big)^2\big] - \displaystyle\frac{\lambda\Upsilon}{k^2h}\,[w^2] - (\mu+\lambda^2)\,
\Big(\big[\big(\mcc_{kh}(w')\big)^2\big] +\big[(w')^2\big]\Big) - \frac{\mu}{k^2}.
\end{align}
\end{subequations}

\begin{proof}[Proof of Theorem \ref{glbp}] We are going to apply Theorem \ref{abstract} in the
setting (\ref{eqso})-(\ref{fels}), and then we shall transfer in a rather straightforward  manner the results
obtained for (\ref{fels}) to the corresponding results for (\ref{sys}). It is obvious that the mapping $F$
is \emph{real-analytic} on $\R\times X$. Let
\be
{\mathcal O}:=\{(\lambda, (\mu, w))\in \R\times X: \mu+\la^2-2gw(x)>0\quad\text{for all }x\in\R\},\label{mco}\ee
which is an open set in $\R\times X$. We now check the validity of the assumptions $(H_1)-(H_4)$.

It is obvious that $(H_1)$ holds.   As for $(H_2)$, we easily compute
\begin{equation}\label{lb5}
\pa_{(\mu,w)}\,F(\lambda,(0,0))\,(\nu, \varphi)
=\left(2\,\lb^2 \,\mcc_{kh}(\varphi') - \frac{2g}{k}\,\varphi - \frac{2\lambda\Upsilon}{k}\,\varphi, -\frac{\nu}{k^2}\right),\qquad
(\nu, \varphi) \in \mcx.
\end{equation}
Expanding the function $\varphi$, which is even, has period $2\pi$, and has zero average,
in the series
 $\varphi(x) = \displaystyle\sum_{n=1}^\infty a_n \cos nx$, we obtain the representation
\begin{equation}\label{lineariz}
\pa_{(\mu,w)}\,F(\lambda,(0,0))\,(\nu, \varphi)
=\left(\sum_{n=1}^\infty a_n  \left\{2\lb^2\, n\coth(nkh)
 - \frac{2g}{k}  - \frac{2\lb\Upsilon}{k} \right\}   \cos (nx), -\frac{\nu}{k^2}\right),\qquad
(\nu, \varphi) \in \mcx.
\end{equation}
It follows from (\ref{lineariz})
that the bounded linear operator
$\pa_{(\mu,w)}\,F(\lambda,(0,0)):\mcx\to \mcy$ is invertible
whenever $\lambda$  is \emph{not} a solution of \be\lambda^2
nk\coth(nkh)= g+\Upsilon\lambda,\label{bif}\ee for any integer $n\ge1$.
Hence by the Implicit Function Theorem these points are not bifurcation points.
The solutions of (\ref{bif}) are, for any integer $n\ge1$, given by
\begin{equation}
\lambda_{n,\pm}^*=\frac{\Upsilon\tanh(nkh)}{2nk} \pm
\displaystyle\sqrt{\frac{
\Upsilon^2\tanh^2(nkh)}{4n^2k^2}+g\,\frac{\tanh(nkh)}{nk}}.\label{lb6}
\end{equation}
Observe that all of these values are distinct and none of them vanishes.
We claim that $(H_2)$ holds for every
$\la^*\in\{\lambda_{n,\pm}^*:n\in\N\}.$ Indeed, consider any such $\lambda^*$. It follows easily from (\ref{lineariz}) that $
{\mathcal N}(\pa_{(\mu,w)} F(\lambda^\ast,(0,0)))$ is
one-dimensional and generated by $(0, w^*)\in \mcx$, where
$w^*(x)=\cos (nx)$ for all $x\in\bdr$, while ${\mathcal
R}(\pa_{(\mu,w)} F(\lambda^\ast,(0,0)))$ is the closed subspace of
$\mcy$ formed by the elements $(f,c)\in \mcy$ where $c\in\R$ is arbitrary and $f$ satisfies
$$\int_{-\pi}^{\pi} f(x)\,\cos(nx)\,dx=0.$$
 		From (\ref{lb5}), we now compute
$$\pa^2_{\lambda,(\mu,w)}\,F(\lambda^\ast,(0,0)) \,(1,(0,w^\ast))=\left(
4\lambda^\ast n\coth(nkh)-\frac{2\Upsilon}{k}\right)\,w^\ast \not \in {\mathcal
R}(\pa_{(\mu,w)} F(\lambda^\ast,(0,0)))$$ since, using (\ref{bif}),
we have
$$4\lambda^\ast n\coth(nkh)-\frac{2\Upsilon}{k}=\,2\lambda^\ast\,\Big(n\coth(nkh)+\frac{g}{k(\lambda^\ast)^2}\Big)
\neq 0.$$ This shows that $({H_2})$ holds for every $\la^*\in\{\lambda_{n,\pm}^*:n\in\N\}$.

To check the validity of ($H_3$)-($H_4$), we rewrite $F_1$ in (\ref{varan}) in the following form
\begin{align}
 F_1(\lambda, (\mu, w))&=2(\mu+\lb^2-2gw)\,\mcc_{kh}(w')-\,2g \Big\{ \mcc_{kh}(ww')-w\,\mcc_{kh}(w')\Big\} \nonumber
 \\
 &\qquad\qquad- \Upsilon^2\,\Big\{ \mcc_{kh}(w^2w') +w^2\,\mcc_{kh}(w') - 2w\,\mcc_{kh}(ww')\Big\}\nonumber
 \\
 &\qquad\qquad + \displaystyle\frac{\Upsilon^2}{kh}\,w[w^2]  - \displaystyle\frac{\Upsilon^2}{k}\,w^2
+ \frac{\Upsilon^2}{k}\,[w^2]\,+2g\,[w\,\mcc_{kh}(w')] - \frac{2g}{k}\,w - \frac{2\lambda\Upsilon}{k}\,w\,\label{seag}
\\
&=2(\mu+\lb^2-2gw)\,\mcc_{kh}(w')+J(\lambda,w),\non
\end{align}
where we have slightly rearranged the first four terms, and left the others unchanged. Here
\[J(\lambda,w)= J_1(w)+J_2(w)+J_3(\lambda,w),\]
with
\begin{align*}J_1(w)&=-\,2g \Big\{ \mcc_{kh}(ww')-w\,\mcc_{kh}(w')\Big\},\\
J_2(w)&=- \Upsilon^2\,\Big\{ \mcc_{kh}(w^2w') +w^2\,\mcc_{kh}(w') - 2w\,\mcc_{kh}(ww')\Big\},\end{align*}
and $J_3(\lambda,w)$ gathers all the remaining terms in (\ref{seag}).
Since $J_2$ may also be rewritten as
\[
J_2(w)=- \Upsilon^2\,\Big\{ \mcc_{kh}(w(ww'))- w\,\mcc_{kh}(ww')
+w(w\,\mcc_{kh}(w') -\mcc_{kh}(ww'))\Big\},\]
it is a consequence of Lemma \ref{lm} that the continuous nonlinear mappings
$J_1, J_2: C^{2,\alpha}_{2\pi,\circ}(\bdr)\to C^{1,\alpha}_{2\pi}(\bdr)$ map bounded sets of
$C^{2,\alpha}_{2\pi,\circ}(\bdr)$ into bounded sets of  $C^{2,\alpha/2}_{2\pi}(\bdr)$,
and thus into relatively compact subsets of $C^{1,\alpha}_{2\pi}(\bdr)$.
We also rewrite (\ref{meann}) as
\be F_1(\lambda, (\mu, w))=-\frac{\mu}{k^2}+K(\lambda, (\mu, w)).\ee
It then follows that the nonlinear mapping from $X$ into $Y$
\be (\lambda, (\mu, w))\mapsto (J(\lambda, w), K(\lambda, (\mu,w))\ee
maps bounded sets of
$X$ into bounded sets of  $C^{2,\alpha/2}_{2\pi}\times \R$, and thus into relatively compact subsets of $Y$, and is therefore a nonlinear compact operator.
It then follows (see \cite[Lemma 3.1.12]{BT}) that any of its partial derivatives is a linear compact operator.
Note that, for any $(\lambda, \nu,w)\in {\mathcal O}$, one may write
\begin{eqnarray*}
\partial_{(\mu,w)}F_1(\lambda, (\mu, w))(\nu,\varphi)
&=& 2(\mu+\lb^2-2gw)\,\mcc_{kh}(\varphi')-4g\,\mcc_{kh}(w')\varphi \\
&&\qquad\qquad +\,2\,\mcc_{kh}(w')\nu+\partial_{(\mu,w)}J(\lambda, (\mu, w))(\nu,\varphi),\\
\partial_{(\mu,w)}F_2(\lambda, (\mu, w))(\nu,\varphi)&=&-\frac{\nu}{k^2}
+\partial_{(\mu,w)}K(\lambda, (\mu, w))(\nu,\varphi).
\end{eqnarray*}
Since the condition $\mu+\lb^2-2gw(x)>0$, guaranteed by the definition of ${\mathcal O}$,
ensures that the bounded linear operator from $X$ to $Y$ given by
\[(\nu,\varphi)\mapsto \Big(2(\mu+\lb^2-2gw)\,\mcc_{kh}(\varphi'),-\frac{\nu}{k^2}\Big) \]
is invertible, it follows that $\partial_{(\mu,w)}F:X\to Y$ is the sum of an invertible linear operator
and a compact linear operator.   It  is therefore a Fredholm operator of index zero
(see \cite[Theorem 2.7.6]{BT}). Thus assumption $(H_3)$ is indeed satisfied.

We now verify assumption $(H_4)$ for the sequence $({\mathcal Q}_j)_{j\geq 1}$ given by
\begin{equation}\label{quj}
\begin{array}{l}
 {\mathcal Q}_j=\Big\{(\lambda, (\mu,w))\in {\mathcal O}: ||(\lambda, (\mu, w))||\leq j\,,\\[0.2cm]
\qquad \qquad \mu+\la^2-2gw(x)\geq \displaystyle\frac{1}{j}\quad\text{for all }x\in\R\Big\}\,.
\end{array}
\end{equation}
It is obvious that each ${\mathcal Q}_j$ is bounded and closed and that $\cup_{j\in J}{\mathcal Q}_j={\mathcal O}$.
Let $j\in\N$ be arbitrary and consider $(\lambda, (\mu,w))\in {\mathcal Q}_j$ such that $F(\lambda, (\mu,w))=0$.
Then we have in particular that
\[2(\mu+\lb^2-2gw)\,\mcc_{kh}(w')+J(\lambda, w)=0,\]
in the notation defined earlier in the proof. Since $\mu+\lb^2-2gw\geq 1/j>0$, one may invert the linear operator $w\mapsto \mcc_{kh}(w')$ in this equation to get
\[ w=-(\mcc_{kh}\partial_x)^{-1} \left(\frac{1}{\mu+\lb^2-2gw}J(\lambda, w)\right).\]
Combining the bounds ensured by the definition of ${\mathcal Q}_j$ with the commutator estimates satisfied by $J_1$ and $J_2$ that we have used above, we obtain  an uniform upper bound for $w$ in the space $C^{3, \alpha/2}_{2\pi}(\bdr)$. This implies that $\{(\lambda, (\mu, w))\in {\mathcal O}: F(\lambda, (\mu, w))=0\}\cap {\mathcal Q}_j$ is bounded in $\R\times \R\times C^{3, \alpha/2}_{2\pi}(\bdr)$, and therefore compact in $\R\times X$. Thus assumption $(H_4)$ is satisfied by the sequence of sets given by (\ref{quj}).

We have thus checked that the assumptions $(H_1)$--$(H_4)$ are satisfied in the setting (\ref{eqso})-(\ref{fels}). In view of the definition of the set ${\mathcal O}$ in (\ref{mco}), the relation between $(\lambda, \mu, w)$ and $(m, Q, v)$ expressed by
(\ref{cjn}), (\ref{mu}) and (\ref{lamb}), and the definition of the sets ${\mathcal Q}_j$ in (\ref{quj}), the conclusion of Theorem \ref{abstract} can be rephrased in a straightforward way to yield the result claimed by Theorem \ref{glbp}.

\end{proof}

\section{Nodal analysis}
\nequation

The wave profiles corresponding to points lying on the solution curves $\K_{n,\pm}$ present certain
qualitative features. Firstly, they are symmetric since the function space in Theorem 3 is $C_{2\pi, e}^{2,\alpha}({\mathbb R})$.
Furthermore, along any wave except the bifurcating laminar flow, the elevation of the surface wave is strictly
decreasing between the wave crest and a successive wave trough. For waves that are close to the bifurcating
laminar flow, this is a direct consequence of local bifurcation, while away from the bifurcating flow with
a flat free surface this will be proved in what follows by means of a continuation argument.

\begin{lemma}  [Periodicity and Symmetry]                                            \label{mult}
Let $h,\,k>0$ and $\Upsilon \in \bdr$ be given.
For each integer $n\ge1$ and both choices of sign  $\pm$, denote by
\be \mathcal{K}_{n,\pm}=\{(m(s), Q(s), v_s):s\in \R\}\ee
  the continuous curve of solutions of (\ref{sys}) in the space $\R\times \R\times C^{2, \alpha}_{2\pi, e}(\bdr)$ given by Theorem \ref{glbp}. Then the following additional properties hold:
 \begin{itemize}
  \item[(i)] $v_s$ is periodic of period $2\pi/n$, for each $s\in\R$ ;
 \item[(ii)]  $m(-s)=m(s)$, $Q(-s)=Q(s)$, and $v_{-s}(x)= v_s(x+\pi)$ for all $x\in\R$, for each $s\in\R$.
\end{itemize}
 \end{lemma}

 \begin{proof} (i) Fix some integer $n\geq 2$, arbitrary. Consider again equations (\ref{eqso})--(\ref{fels}), but this time in the setting of even functions of period $2\pi/n$, that is, with the spaces $X$ and $Y$ being replaced by $\widetilde X$ and $\widetilde Y$ given by \be\label{fcs0}
\widetilde\mcx=\bdr\times C^{2,\alpha}_{{2\pi/n},\circ, e}(\bdr),\qquad
\widetilde\mcy= C^{1,\alpha}_{{2\pi/n}, \circ,e}(\bdr)\times \R,
\ee
where the notation has the obvious meaning. Then it is immediate to check that Theorem \ref{abstract} is applicable in the new setting, and global real-analytic bifurcation takes place exactly from the local bifurcation points $(\lambda_{np,\pm}^*)_{p\geq 1}$ given by (\ref{lb6}). This leads to the existence of global bifurcation curves $\widetilde{\mathcal K}_{p,\pm}$ for each $p\geq 1$ of solutions to (\ref{sys}) with the properties as in Theorem \ref{glbp}. Since $\widetilde X\subset X$ and $\widetilde Y\subset Y$, it follows that $\widetilde{\mathcal K}_{1,\pm}$ are curves of solutions of (\ref{sys}) in the space $X$ too, with properties analogous to (i)-(v) as satisfied by  ${\mathcal K}_{n,\pm}$. The uniqueness claim in Theorem \ref{glbp} then ensures that $\widetilde{\mathcal K}_{1,\pm}={\mathcal K}_{n,\pm}$, up to reparametrization. It therefore follows that ${\mathcal K}_{n,\pm}\subset \widetilde X$, which is the required result.

(ii) Fix some integer $n\geq 1$, arbitrary. It is easy to check that
\[\R\ni s \mapsto (m(-s), Q(-s), v_{-s}(\cdot+\pi))\]
is a curve of solutions of (\ref{sys}) satisfying the properties (i)-(v) in Theorem \ref{glbp}. The required result is a consequence of the uniqueness claim in Theorem \ref{glbp}.
 \end{proof}

For simplicity, we will concentrate in what follows on discussing qualitative properties of the solutions on the curves ${\mathcal K}_{1,\pm}$. Similar properties (with obvious modifications) may be proven for solutions on the curves ${\mathcal K}_{n,\pm}$ for any integer $n\geq 2$ if one works in the space of functions of period $2\pi/n$ (this choice is justified by Lemma \ref{mult}(i)).
 Also denoting,  for both choices of sign  $\pm$ and for $n=1$,
\be
{\mathcal K}_{1,\pm}= {\mathcal K}_{\pm}^{<}\cup \{(m_{1,\pm}^*, Q_{1,\pm}^*, h)\}\cup {\mathcal K}_{\pm}^{>},\ee
where
\be
{\mathcal K}_{\pm}^{<}:=\{(m(s), Q(s), v_s): s\in (-\infty, 0)\}\quad\text{and}\quad{\mathcal K}_{\pm}^{>}:=\{(m(s), Q(s), v_s): s\in (0,\infty)\},\label{kpm}\ee
it suffices,  because of Lemma \ref{mult}(ii), to study the properties of the solutions on ${\mathcal K}_{\pm}^{>}$.

For any function $v\in C^{2,\alpha}_{2\pi, e}$ with $[v]=h$, we consider in $\overline\mcr_{kh}$ the functions $U$ and $V$ as defined in Section 2.
Then the evenness of $v$ implies that
\be
x\mapsto V(x,y)\text{ is an even function, for each }y\in [-kh, 0]\label{veven}\ee
and that the arbitrary additive constant in the definition of $U$ may be chosen so that
\be
x\mapsto U(x,y)\text{ is an odd function, for each }y\in [-kh, 0].\label{uodd}\ee
				This ensures that (\ref{U}) holds.
It follows in particular that \be\text{$U(0,y)=0$ for all $y\in [-kh, 0]$},\label{sd1}\ee
and as a consequence of (\ref{v}) that
\be
\text{$U(m \pi, y)= \frac{m\pi}{k}$ for all $y\in [-kh, 0]$ and $m\in \mathbb Z$}.\label{sd2}\ee

\begin{lemma}  [Injectivity]        \label{injcond}
If
\be
V_x(x,0)\neq 0\quad\text {for all }x\in (0,\pi),          \label{sd5}\ee
then the injectivity condition (\ref{inj}) is valid if and only if
\be
0<U(x,0)< \frac{\pi}{k}\quad\text{for all }x\in (0,\pi).          \label{sd4}\ee
\end{lemma}
\begin{proof}
Suppose first that (\ref{inj}) holds.
If $U(x_0,0)=0$ for some $x_0\in (0,\pi)$, then $(U(x_0,0), V(x_0,0))  =  (0,V(x_0,0))  =  (U(-x_0,0), V(-x_0,0))$,
which contradicts (\ref{inj}).
If $U(x_0,0)=\pi/k$ for some $x_0\in (0,\pi)$, then,
because of (\ref{veven}),(\ref{uodd}) and \eqref{v},
$(U(x_0,0), V(x_0,0))  =  (-U(x_0,0)+\frac{2\pi}{k}, V(x_0,0))
=  (-U(x_0-2\pi,0), V(-x_0,0))  =  (U(2\pi-x_0,0), V(2\pi-x_0,0))$,
which also contradicts  (\ref{inj}).
Hence $0 \ne U(x,0) \ne \pi/k$ for all $x\in(0,\pi)$.
Furthermore, from \eqref{U}, we have
\[
\int_0^\pi U(x,0)\,dx=\frac{1}{k}\frac{\pi^2}{2}, \]
from which it follows that \eqref{sd4} is valid.

For the converse, suppose that $(U(x_1,0), V(x_1,0)) = (U(x_2,0), V(x_2,0))$
for some pair of real numbers  $x_1\ne x_2$.
Since $V(\cdot,0)$ is even and has period $2\pi$, we have
$$ x_2 = - x_1 + 2\pi m \quad \text{ for some integer } m, $$
due to the assumption \eqref{sd5}.
Then the oddness of $U(\cdot,0)$ together with \eqref{v} give us
$U(x_1,0) = U(x_2,0) = -U(x_1-2\pi m,0) = -U(x_1,0) + 2\pi m/k$.
Thus $U(x_1,0) = m\pi/k$.
Writing $x_1 = n\pi + x_0$ for some $n\in\mathbb Z$ and $x+0\in [0,\pi)$,
we have
$$ \frac{m\pi}{k}  = U(n\pi+x_0,0) = U(x_0,0) + \frac{n\pi}{k}.$$
By \eqref{sd4}, $m=n$ and $U(x_0,0)=0$, so that $x_0=0$.
Thus $x_1=m\pi$ and $x_2=-x_1+2\pi m = x_1$, a contradiction.
\end{proof}
We shall see in Lemma \ref{nplo} below that  (\ref{sd5}) and (\ref{sd4}) are
satisfied by the solutions on the bifurcation curve ${\mathcal K}_{1,\pm}$
that are close enough to the trivial solution $v\equiv h$.
Of course, in order for $v$ to give rise to a water wave it is also necessary that
\be
V(x,0)>0\quad\text{for all }x\in\R.\label{sd6}\ee
 				The discussion above leads us to consider the following seven properties of a function $v$:
 \be
 v(x)>0 \text{ for all } x\in \R,\label{as03}\ee
 \be v\not\equiv h, \label{nontr}\ee
 \be v'(x)<0\text{ for all }x\in \left(0,\pi\right),\label{nmbv}\ee
 \be v''(0)<0,\,\, v''\left(\pi\right)>0,\label{nodalloc}\ee
 \be 0< \frac{x}{k}+ \mcc_{kh}(v-h)(x)< \frac{\pi}{k}\text{ for all }x\in \left(0,\pi\right),\label{ga1}\ee
 \be \frac{1}{k}+ \mcc_{kh}(v')(0)>0,\quad \frac{1}{k}+ \mcc_{kh}(v')\left(\pi\right)>0,\label{ga2}\ee
 \begin{equation}\label{aza0}
 \pm \left(\frac{m}{kh} - \frac{\Upsilon}{2kh}[v^2]-\Upsilon\mcc_{kh}(vv') +\Upsilon v\,\Big(\frac{1}{k}+\mcc_{kh}(v')\Big)\right)> 0
\text{ for all } x\in \R.
\end{equation}
Equation \eqref{add} implies that \eqref{aza0} is equivalent to $Q-2gv\ne0$ and
$(v')^2 + (\frac1k + \mcc_{kh} v')^2 \ne 0$.
Because of Lemma \ref{injcond} and the discussion in Section 2,
any solution in ${\mathcal K}_{\pm}^{>}$
that satisfies \eqref{as03}-\eqref{aza0} corresponds to a water wave.

We will study to what extent these properties are satisfied along ${\mathcal K}_{\pm}^{>}$.
To that aim, it is convenient to define the sets
\be {\mathcal V}_{\pm}:=\{(m, Q,v)\in \R\times \R\times C^{2, \alpha}_{2\pi, e}(\bdr): (\ref{as03})-(\ref{aza0}) \text{ hold}\}\ee
the choice of sign in ${\mathcal V}_{\pm}$ being the same as that in (\ref{aza0}).
Note that  ${\mathcal V}_{\pm}$ are {\it open sets} in $\R\times \R\times C^{2, \alpha}_{2\pi, e}(\bdr)$.
The next lemma deals with solutions on ${\mathcal K}_{\pm}^{>}$ which are close to the trivial one.

\begin{lemma}[Local Properties]          \label{nplo}
For either choice of sign in $\pm$, let ${\mathcal K}_{\pm}^{>}=\{(m(s), Q(s), v_s): s\in (0,\infty)\}$ be
the curve of solutions of (\ref{sys}) given by (\ref{kpm}). Then there exists $\ep>0$ sufficiently small such that
 \be
 \{(m(s), Q(s), v_s): s\in (0,\ep)\}\subset {\mathcal V}_{\pm}\quad\text{for all }s\in (0,\ep).\ee
  \end{lemma}

\begin{proof}

Since the set of $(m, v)$ defined by the conditions (\ref{as03}), (\ref{ga1}), (\ref{ga2}), and (\ref{aza0}) is open $ \R\times C^{2, \alpha}_{2\pi, e}(\bdr)$, and the mapping $s\mapsto (m(s), v_s)$ is continuous from $\R$ into $\R\times C^{2, \alpha}_{2\pi, e}(\bdr)$,
the fact that those conditions are satisfied at $s=0$ implies that they are satisfied on $(-\varepsilon_1, \varepsilon_1)$, for some $\varepsilon_1>0$ sufficiently small.
Similarly, since the set
\[{\mathcal Y}:=\{u\in C^{2, \alpha}_{2\pi, e}(\bdr): u\not\equiv 0, u'(x)<0\text{ for all }x\in (0,\pi), u''(0)<0, u''(\pi)>0\}\]
 is open in $C^{2, \alpha}_{2\pi, e}(\bdr)$ and the mapping
\[s\mapsto u_s:=
\left\{\begin{array}{l}
\displaystyle    \left(v_s -h\right)/s  \quad\,\,\,\text{ for }s\neq 0,\\
(x\mapsto \cos x)\quad\text{ for }s= 0,
\end{array}\right.
\]
is continuous from $\R$ into $C^{2, \alpha}_{2\pi, e}(\bdr)$, the fact that $u_0\in {\mathcal Y}$ implies that $u_s\in {\mathcal Y}$ for all $s\in (-\varepsilon_2, \varepsilon_2)$, for some $\varepsilon_2>0$ sufficiently small; it follows that $v_s$ satisfies (\ref{nontr})--(\ref{nodalloc}) for all $s\in (0,\varepsilon_2)$. Setting $\varepsilon:=\min\{\varepsilon_1, \varepsilon_2\}$ yields the required result.

\end{proof}

Fix some $\ep>0$ given by Lemma 3, and let us define ${\mathcal K}_{\pm, \textnormal{loc}}^{>}:=\{(m(s), Q(s), v_s):s\in (0,\ep)\}$. By reparametrizing ${\mathcal K}_{\pm, \textnormal{loc}}^{<}$ such that Proposition \ref{mult}(ii) holds, let us define also ${\mathcal K}_{\pm, \textnormal{loc}}^{<}:=\{(m(s), Q(s), v_s):s\in (-\ep, 0)\}$.
The main result of this subsection is the following theorem.

\begin{theorem}[Global Continuation]                   \label{nodalthm}
For either choice of sign $\pm$, let ${\mathcal K}_{\pm}^{>}=\{(m(s), Q(s), v_s): s\in (0,\infty)\}$ be the curve of solution of (\ref{sys}) given by (\ref{kpm}). Then one of the following alternatives occurs:
\begin{itemize}
\item[$(A_1)$] ${\mathcal K}_{\pm}^{>}\subset {\mathcal V}_{\pm}$, in which case alternative $(\alpha)$ in Theorem \ref{glbp} occurs;
\item[$(A_2)$] there exists $s^* \in (0,\infty)$ such that $\{(m(s), Q(s), v_s): s\in (0,s^*)\}\subset {\mathcal V}_{\pm}$, while
$(m(s^*), Q(s^*), v_{s^*})$ satisfies (\ref{as03}), (\ref{nontr}), (\ref{nmbv}), (\ref{nodalloc}),
(\ref{ga2}) and (\ref{aza0}), while instead of (\ref{ga1}) it satisfies
\begin{subequations}\label{ga}
\begin{align} 0< \frac{x}{k}+ \mcc_{kh}(v-h)(x)&\leq \frac{\pi}{k}\text{ for all }x\in \left(0,\pi\right),\label{ga3}\\
 \frac{x_0}{k}+ \mcc_{kh}(v-h)(x_0)&=\frac{\pi}{k}\quad\text{for some }x_0\in (0,\pi).\label{ga4}\end{align}
 \end{subequations}
\end{itemize}
\end{theorem}

\begin{remark} {\rm

(i) Roughly speaking, alternative $(A_1)$ means that all the solutions on ${\mathcal K}_{\pm}^{>}$ correspond
to physical water waves that are symmetric and whose vertical coordinate strictly decreases between each
of its consecutive global maxima and minima, which are unique per minimal period.
However, since we make no claim that
the horizontal coordinate is strictly monotone, such waves could have overhanging profiles. The existence of
overhanging waves is strongly suggested by numerical simulations and remains an important open problem.
More generally, the behaviour of the
solutions on ${\mathcal K}_{\pm}^{>}$ as $s\to \infty$ remains another important open problem.

(ii) Alternative $(A_2)$ means that solutions on  ${\mathcal K}_{\pm}^{>}$ that correspond
to physical water waves with qualitative properties as described above do exist until a limiting configuration
with a profile that self-intersects on the line strictly above the trough is reached at $s=s^*$. Indeed, recall that $u(x)=U(x,0)$.  By (\ref{U}), (\ref{ga4}) and (\ref{sd2}),
$u(x_0) = x_0/k + \C_{hk}(v-h)(x_0) = \pi/k = u(\pi)$.
This implies that the physical point $(u(x_0),v(x_0)) \in \S$ lies directly
above the trough (which is at $(u(\pi),v(\pi))$).
By (\ref{v}) and (\ref{uodd}), $u(2\pi-x_0) = u(-x_0) + 2\pi/k = 2\pi/k - u(x_0) =
u(x_0)$, while  (\ref{v}) and (\ref{veven}) yield $v(2\pi-x_0) = v(-x_0) = v(x_0)$.
Thus the curve $\S$ intersects itself at the physical point $(u(x_0),v(x_0))$,
which lies directly above the trough.}
\end{remark}

\begin{proof}[Proof of Theorem 5]
 Fix a choice of sign in $\pm$. In what follows we choose the $+$ sign merely for definiteness.
 All the arguments below can be straightforwardly adapted to the choice of the $-$ sign.
 For convenience, define $I=\{s\in (0,\infty): (m(s), Q(s), v_s)\in {\mathcal V}_+\}$.
 Since ${\mathcal V}_+$ is an open set in  $\R\times \R\times C^{2, \alpha}_{2\pi, e}(\bdr)$, it follows that $I$ is an open subinterval of $(0,\infty)$.
 In case $I$ equals all of $(0,\infty)$, the definition of $\mathcal V_+$ implies
 ${\mathcal K}_{+}^{>}\subset {\mathcal V}_{+}$,
 so that $(\beta)$ in Theorem \ref{glbp} is excluded and therefore $(\alpha)$ is valid.

 Hence we are left with the case that the open interval $I$ is not the whole of $(0,\infty)$.
With $\ep>0$ given by Lemma \ref{nplo}, we have $(0,\ep)\subset I$.
Let $s^*$ be the upper end-point of the largest interval that contains $(0, \ep)$ and  is contained in $I$.
Then $(0,s^*)\subset I$ and $s^*\notin I$.
In what follows we shall investigate the properties of the solution $(m(s^*), Q(s^*), v_{s^*})$.

We first claim that necessarily $v_{s^*}\not\equiv h$.
Suppose on the contrary that $v_{s^*}\equiv h$.
Then the point $(m(s^*), Q(s^*), h)$ belongs to ${\mathcal K}_{triv}$ and is a limit of a sequence of nontrivial solutions.
It follows from Theorem \ref{glbp} that necessarily there exists an integer $n\geq 1$ and a choice of sign in $\pm$ such that $m(s^*)=m_{n,\pm}^*$ and $Q(s^*)=Q_{n,\pm}^*$, where $m_{n,\pm}^*$ and $Q_{n,\pm}^*$ are given by (\ref{lb7}) and (\ref{qlk7}).
At this point note that it is a consequence of Theorem \ref{glbp}(iv) and the periodicity in
Lemma \ref{mult}(i) that,  for any $n\geq 2$,
all nontrivial solutions in a neighbourhood of $(m_{n,\pm}^*, Q_{n,\pm}^*, h)$
are periodic of period $2\pi/n$.
By \eqref{nmbv}, $v'_{s^*} \le 0$ for $0\le x\le \pi$.  Thus $n=1$.
It follows that either
$(m(s^*), Q(s^*)=(m_{1,+}^*, Q_{1,+}^*)$ or $(m(s^*), Q(s^*)=(m_{1,-}^*, Q_{1,-}^*)$.
However the possibility that $(m(s^*), Q(s^*)=(m_{1,-}^*, Q_{1,-}^*)$ is ruled out by the fact
proved in Lemma \ref{nplo} that all nontrivial solutions in a neighbourhood of $(m_{1,-}^*, Q_{1,-}^*, h)$
satisfy (\ref{aza0}) with the minus sign, combined with the fact that all solutions on  ${\mathcal K}_{+}^{>}$
satisfy (\ref{aza0}) with the plus sign.
Thus the only remaining possibility is that $(m(s^*), Q(s^*)=(m_{1,+}^*, Q_{1,+}^*)$.

Now all the nontrivial solutions in a neighbourhood of $(m_{1,+}^*, Q_{1,+}^*, h)$ belong to
${\mathcal K}_{+}^{>} \cup {\mathcal K}_{+}^{<}$.
By the evenness and periodicity,  the solutions on
${\mathcal K}_{+,\textnormal{loc}}^{<}$  satisfy the opposite of (\ref{nmbv}),
so that ${\mathcal K}_{+}^{<}$ is excluded.
Hence there exists $\delta_1, \delta_2>0$ sufficiently small such that
$\{(m(s), Q(s), v_s):s\in (0,\delta_1)\}= \{(m(s), Q(s), v_s): s\in (s^*-\delta_2, s^*)\}$.
However, this possibility is ruled out by $(C_4)$ in Theorem \ref{abstract}, that is, by
the construction of the real-analytic global bifurcation curve in \cite{BT}; see Figure 3.
We have thus proven that $v_{s^*}\equiv h$ is not possible, which proves the claim.

\begin{figure} \centering  \includegraphics[width=10cm]{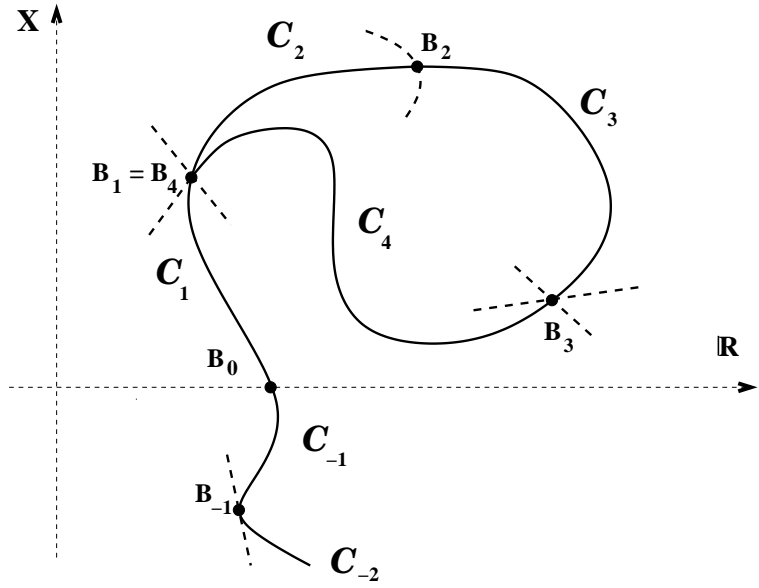}
\caption{\small The
possibility of a loop is eliminated by a nodal pattern analysis.}
\end{figure}

 For notational simplicity, throughout the remainder of the proof we shall denote $(m(s^*), Q(s^*), v_{s^*})$
 by $(m,Q, v)$.   This is a limit of solutions satisfying \eqref{as03} to \eqref{aza0}.
 The definition of $I$ implies that
\be v(x)\geq 0 \text{ for all }x\in\R, \text {and is even of period } 2\pi,    \label{as5}\ee
\be v'(x)\leq 0\text{ for all }x\in [0,\pi], \text{ so that } v'\ge 0 \text { in } [\pi,2\pi],  \label{as1}\ee
\be 0\leq \frac{x}{k}+ \mcc_{kh}(v-h)(x)\leq \frac{\pi}{k}\ \text{ for all }x\in [0,\pi],        \label{as2}\ee
\begin{equation}\label{aza}
 +\left(\frac{m}{kh} - \frac{\Upsilon}{2kh}[v^2]-\Upsilon\mcc_{kh}(vv') +\Upsilon v\,\Big(\frac{1}{k}+\mcc_{kh}(v')\Big)\right)\geq 0
\text{ for all } x\in \R.\end{equation}
By \eqref{U},  (\ref{as2}) may be rewritten as
\be 0= U(0,0)\leq U(x,0)\leq U(\pi,0)=\pi/k \quad\text{for all }x\in [0,\pi]\label{as22}.\ee
Also, by \eqref{vzc} with $\zeta$ defined as in Section 2, (\ref{aza}) may be rewritten as
\be
\zeta_y(x,0)+\Upsilon V(x,0) V_y(x,0)\
\geq 0 \quad\text{for all }x\in \R.
\label{as6}\ee
By Theorem \ref{equiv}, $v$ also satisfies (\ref{add}), which is the same as (\ref{gc2}).

We are now going to show that $(m,Q,v)$ satisfies the remaining properties claimed in the statement of
alternative $(A_2)$ of Theorem \ref{nodalthm}.
The proof will be based on sharp forms of maximum principles in the infinite strip $\mcr_{kh}$, as well as in the rectangular domain
\be
\mcr:=\{(x,y): 0<x<\pi,\  -kh<y<0\}, \ee
whose boundary $\partial \mcr$ is the union of the four line segments
\begin{align*}&\partial\mcr_t=\{(x,0):0\leq x\leq \pi\},\qquad \partial\mcr_b=\{(x,-kh):0\leq x\leq \pi\},\\
&\partial\mcr_l=\{(0,y):-kh\leq y\leq 0\},\qquad\partial\mcr_r=\{(0,y):-kh\leq y\leq 0\},\end{align*}
see Figure 4.

\begin{figure}\label{rect}    \begin{center}$
\includegraphics[width=2.25in]{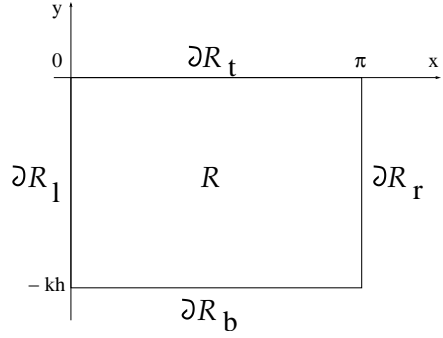}$   \end{center}
\caption{The rectangular domain ${\mathcal R}$, the conformal image of which
is the fluid domain between the wave crest and a successive wave trough.}  \end{figure}

Suppose on the contrary that (\ref{as03}) fails. Then (\ref{as5}) and (\ref{as1}) imply that $v(\pi)=0$. Since then the harmonic function $V$ in $\mcr_{kh}$ has a global minimum in $\overline{\mcr_{kh}}$  at $(\pi, 0)$, by the Hopf boundary-point lemma, $V_y (\pi,0)<0$.
On the other hand, it is a consequence of (\ref{as22}) and the Cauchy-Riemann equations that
$V_y(\pi,0)= U_x(\pi, 0)\geq 0$, which is a contradiction. It follows that (\ref{as03}) does hold.

 We now prove the strict inequality (\ref{aza0}) with the $+$ sign; that is,
\begin{equation}\label{aza00}
\zeta_y + \Upsilon VV_y \Big|_{y=0}  =
+\left(\frac{m}{kh} - \frac{\Upsilon}{2kh}[v^2]-\Upsilon\mcc_{kh}(vv')
+\Upsilon v\,\Big(\frac{1}{k}+\mcc_{kh}(v')\Big)\right)> 0
\text{ for all } x\in \R.
\end{equation}
Recalling from Theorem \ref{glbp}(i) that
\be Q-2gv(x)>0 \text{ for all }x\in\R, \label{qnm}\ee
and taking  (\ref{gc2}) into account, the validity of (\ref{aza00}) is equivalent to
$$(\zeta_y\,+\,\Upsilon VV_y)^2=(Q-2gV)(V_x^2+V_y^2) > 0 \text { for } y=0 $$
or equivalently
\be
v'(x)^2 +\left(\frac{1}{k}+\mcc_{kh}(v')(x)\right)^2>0\text{ for all } x\in \R,\label{as4}\ee
or also equivalently
\be V_x^2(x,0)+V_y^2(x,0)>0\quad\text{for all }x\in\R.\label{gds}\ee

We will prove \eqref{gds} by contradiction.  Suppose that there exists a point $x_0\in \R$  such that
$V_x^2(x_0, 0)+V_y^2(x_0,0)=0$.
It follows from (\ref{gc2}) that $\zeta_y(x_0,0)+\Upsilon V(x_0,0) V_y(x_0,0)=0$.
Combined with (\ref{as6}), this ensures that
$\zeta_y + \gamma VV_y = O((x-x_0)^2)$  as $x\to x_0$ at $y=0$.
By (\ref{gc2}) it implies that $V_x^2+V_y^2 = O((x-x_0)^4)$ as $x\to x_0$ at $y=0$.
Examining the second derivative of the expression in the right-hand side of (\ref{gc2}), we see that
\be
V_{xx}(x_0,0)=V_{xy}(x_0,0)=0.\label{as7}\ee
Because of the evenness and periodicity of $v$, we may assume with no loss of generality that $0\le x_0 \le \pi$.

If $x_0=0$, it follows from \eqref{as1} that the harmonic function $V$ in $\mcr_{kh}$ has at $(0, 0)$
a global maximum in $\overline{\mcr_{kh}}$.
By  the Hopf boundary-point lemma, $V_y (0,0)>0$.   Hence
\be \frac{1}{k}+\mcc_{kh}(v')(0)   =V_y(0,0)     > 0, \ee
a contradiction.  So $x_0\ne0$.
Now suppose $0<x_0<\pi$.
 The harmonic function $V_x$ has its global maximum in $\overline{\mcr}$ at $(x_0,0)$
 because $V_x=0$ on three sides (by oddness and the bottom boundary condition)
 while $V_x\le0$ on the top $\partial\mcr_t$ by \eqref{as1}.
 By the Hopf boundary-point lemma,  $V_{xy}(x_0,0)>0$, which contradicts (\ref{as7}).

It remains to consider the case $x_0=\pi$.
Since $x\mapsto V(x,y)$ is even about $\pi$, for each $y\in [-kh,0]$, it follows that
\be
V_{xxx}(\pi,0)= V_{xyy}(\pi, 0)=0.\label{as8}\ee
Now the harmonic function $V_x$ in the rectangle $\mcr=\{(x,y): 0<x<\pi,\  -kh<y<0\}$
has its global maximum in $\overline \mcr$  at $(\pi, 0)$.
It follows from the Serrin corner-point lemma, by taking into account (\ref{as7}) and (\ref{as8}),
that $V_{xxy}(\pi,0)<0$.
				By the Cauchy-Riemann equations, $U_{xxx}(\pi,0)<0$.
Similarly, we also deduce  that $U_x(\pi, 0)=U_{xx}(\pi, 0)=0$,
since $V_y(\pi,0)=0$ by assumption and $V_{xy}(\pi, 0)=0$ by (\ref{as7}).
These properties of the derivatives of the function $U(\cdot,0)$ at $(\pi,0)$
imply that
$U(x,0) - U(\pi,0) = \frac16 U_{xxx}(\pi,0) (x-\pi)^3  +  O(x-\pi)^4  >  0$  as $x\nearrow \pi$.
This contradicts (\ref{as22}).

Thus $x_0$ does not exist and \eqref{gds} is true.    Since $V_x(\pi, 0)=0$, we must have $V_y(\pi,0)\neq 0$.
It is a consequence of (\ref{as22}) and the Cauchy-Riemann equations that $V_y(\pi,0)=U_x(\pi,0)\geq 0$.
It follows that
\be \frac{1}{k}+\mcc_{kh}(v')(\pi)  =  V_y(\pi,0)  > 0. \ee
We have therefore proved not only (\ref{aza0}), but also (\ref{ga2}) and (\ref{as4}).
It remains to prove \eqref{nmbv}, \eqref{nodalloc} and \eqref{ga}.

Recall at this point that the solution $(m, Q, v)$ that we are discussing
(more precisely, $(m(s^*), Q(s^*), v_{s^*})$) is the limit in $\R\times\R\times C^{2, \alpha}_{2\pi, e}(\bdr)$
of a sequence $(m^j, Q^j, v^j):= (m(s_j), Q(s_j), v_{s_j})$ where $s_j\in I$ for all $j\in\N$
and $s_j\nearrow s^*$ as $j\to\infty$.
For each $j\in\N$, let $V^j$ and $U^j$ be the associated harmonic functions in ${\mathcal R}_{kh}$.
The definition of $I$ means that, for each $j\in\N$, we have that $(m^j, Q^j, v^j)\in {\mathcal V}_{+}$,
which implies by Lemma \ref{injcond} that
the mapping  $x\mapsto (U^j(x,0), V^j(x,0))$ is injective on $\R$.
As noted in \cite{V2},\cite{CV}, a suitable application of the Darboux--Picard Theorem
\cite[Corollary 9.16, p.\ 3]{Burk} yields that $U^j+iV^j$ is a conformal mapping between
the strip $\mcr_{kh}$ and the domain $\Omega$ whose boundary consists of the curve
\[
\mcs^j:=\{(U^j(x,0), V^j(x,0)):x\in\R\}\]
and the real axis $\mcb$, that extends continuously from $\overline\mcr_{kh}$ to $\overline\Omega$.

As a consequence,
\[
(V^j_x)^2+ (V^j_y)^2>0 \quad\text{in }\mcr_{kh}\cup\{(x, -kh):x\in\R\}. \]
On $\{(x, -kh):x\in\R\}$ we have that $V^j=V_x^j=0$ and $V_y^j\ne0$.
Using the maximum principle and the fact that $U+iV$ is not constant, we may
pass to the limit as $j\to\infty$ to obtain
\[V_x^2+ V_y^2>0 \quad\text{in }\mcr_{kh}\cup\{(x, -kh):x\in\R\}. \]
By taking (\ref{gds}) into account, we therefore deduce that
\be V_x^2+ V_y^2>0 \quad\text{in }\overline{\mcr_{kh}}.\label{nondef}\ee

At this point we state the following lemma, whose proof is somewhat technical.
In order not to disrupt the flow of the argument, we choose to take its validity for granted for the moment,
and defer proving it until after the proof of Theorem \ref{nodalthm}.

\begin{lemma}\label{lrr} Let $(m,Q,v)\in\R\times\R\times C^3_{2\pi, e}(\R)$, with $[v]=h$ and $v\not\equiv h$, be a solution of (\ref{sys}) which satisfies (\ref{qnm}),
(\ref{nondef}), and (\ref{as1}). Then $v$ also satisfies (\ref{nmbv}) and (\ref{nodalloc}).
\end{lemma}

It has been noted in the proof of Theorem 3 that any solution $v$ of class $C^{2,\alpha}$ is necessarily of
class $C^3$, so that Lemma 4 is applicable. As a consequence of Lemma \ref{lrr},
our solution $(m, Q,v)$, which is in fact $(m(s^*), Q(s^*), v_{s^*})$, satisfies (\ref{nmbv}) and (\ref{nodalloc}).
It finally  remains  to prove (\ref{ga}).
Recall that (\ref{as2}) is valid. Suppose first for a contradiction that (\ref{ga3}) fails, so that there exists $x_0\in (0,\pi)$ such that
\[
0=\frac{x_0}{k}+ \mcc_{kh}(v-h)(x_0) = U(x_0,0).\]
Then the  harmonic function $U$ in $\mcr$ has at $(x_0,0)$ a global minimum in $\overline\mcr$.
Hence it follows from the Hopf  lemma that $U_y(x_0, 0)<0$.
By the Cauchy-Riemann equations, $V_x(x_0,0)>0$, which contradicts (\ref{as1}).
This proves that (\ref{ga3}) is valid.
On the other hand, (\ref{ga4}) must necessarily be true, since otherwise, taking into account everything we have proved so far, it would follow that $(m(s^*), Q(s^*), v_{s*})\in {\mathcal V}_{+}$, a contradiction to the definition of $s^*$. This completes the proof of Theorem \ref{nodalthm}, modulo Lemma \ref{lrr}.
\end{proof}

\begin{proof}[Proof of Lemma \ref{lrr}] By Theorem \ref{equiv}, equation (\ref{add}) is also satisfied.
Let $\zeta$ be the solution of (\ref{gc}) and $V$  the solution of \eqref{VV}.
Then    (\ref{gc2}) also holds due to the discussion after \eqref{add}.
The assumptions (\ref{qnm}) and (\ref{nondef}) ensure that
\be
\zeta_y= -\Upsilon VV_y\pm (Q-2gV)^{1/2}(V_x^2+V_y^2)^{1/2}
\quad\text{at }(x,0)\text{ for all }x\in\R,\label{knx}\ee
the choice of sign in the $\pm$ above being the same for all $x\in\R$.
The required result will be obtained by applying maximum principle-type arguments to the
function  $f:\overline\mcr_{kh}\to \R$  given by
\be f=\frac{V_x\zeta_y- V_y\zeta_x}{V_x^2+V_y^2}.\label{deff}\ee
(The choice of $f$ is motivated by the formula (\ref{vel}) for the horizontal velocity of the fluid
in a steady water wave. However, here we are not assuming (\ref{inj}), so that our solution need
not correspond to a water wave.)
It is easy to see that $f$ is a harmonic function in $\mcr_{kh}$
because it is the imaginary part of $g/h$, where $g=-(\zeta_y+i\zeta_x)$ and $h=V_y+iV_x$.
					Note that
$f=0$ on $\partial\mcr_l\cup \partial\mcr_b\cup \partial\mcr_r$, while
\be
f= \frac{V_x(\zeta_y+\Upsilon VV_y)}{V_x^2+V_y^2}
=\pm \frac{V_x(Q-2gV)^{1/2}}{(V_x^2+V_y^2)^{1/2}}\quad\text{on }\partial\mcr_t,\label{newf}\ee
as a consequence of (\ref{zeta_x}) and then (\ref{knx}).
Therefore, in view of the assumptions (\ref{qnm}) and (\ref{nondef}),
$f$ has a constant sign on $\partial\mcr_t$ (depending on the choice of sign in $\pm$), and
$f\not\equiv 0$ on $\partial \mcr_t$.
By the maximum principle it follows that
$f$ has a strict sign in $\mcr$.

We will first prove that $v'(x)<0$ for all $x\in (0,\pi)$.
Suppose on the contrary that $v'(x_0)=0$ for some $x_0\in (0,\pi)$.
Since $x_0$ is a global maximum for $v$ on $(0,\pi)$, it follows that $v''(x_0)=0$.
Also, it follows from (\ref{newf}) that $f(x_0,0)=0$, and thus the harmonic function $f$ in $\mcr$ has at $(x_0,0)$ a global extremum in $\mcr\cup\partial\mcr$.
It follows from the Hopf boundary-point lemma that $f_y(x_0,0)\neq 0$.
We shall now prove, by direct calculation, that $f_y(x_0,0)=0$, thus obtaining a contradiction.
To simplify the following calculations, it is convenient to write $f=p/q$, where $p=V_x\zeta_y-V_y\zeta_x$, $q=V_x^2+V_y^2$.
Then, at every point in $\overline{\mcr_{kh}}$ we have
\be
f_y=\frac{p_yq-pq_y}{q^2}.\label{sale}\ee
This implies in particular that, at the point $(x_0,0)$, at which $p=0$, $V_x=0$, and $V_{xx}=0$, we have
\be
f_y= \frac{p_y}{q}=\frac{V_{xy}\zeta_y- V_y\zeta_{xy}}{V_y^2}.\label{fy}\ee
We now differentiate equation (\ref{knx})  with respect to $x$ to obtain
\be
\zeta_{xy}= -\Upsilon (V_xV_y+VV_{xy})
\pm \left(\frac{-gV_x(V_x^2+V_y^2)^{1/2}}{(Q-2gV)^{1/2}}
+  \frac{(Q-2gV)^{1/2}(V_xV_{xx}+V_yV_{xy})}{(V_x^2+V_y^2)^{1/2}}\right)  \label{knxd}  \ee
at $(x,0)$  for all $x\in\R$.
It follows from (\ref{knx}) and (\ref{knxd}) that, at the point $(x_0,0)$,
at which $V_x=0$ and $V_{xx}=0$, we have
\begin{align}
\zeta_y&= -\Upsilon VV_y\pm (Q-2gV)^{1/2}|V_y|, \label{ale}\\\
\zeta_{xy}&= -\Upsilon VV_{xy}\pm \frac{(Q-2gV)^{1/2}V_yV_{xy}}{|V_y|},\end{align}
from which one can see that $V_{xy}\zeta_y=V_y\zeta_{xy}$ and therefore, by (\ref{fy}), that $f_y=0$, thus obtaining a contradiction. This proves that indeed $v'(x)<0$ for all $x\in (0,\pi)$.

We will now prove that $v''(0)<0$ and $v''(\pi)>0$. Since $v'(0)=v'(\pi)=0$ and $v'\leq 0$ on $[0,\pi]$,
it follows that $v''(0)\leq 0$ and $v''(\pi)\geq 0$, so we just need to prove that the two inequalities are strict.
Suppose on the contrary that $v''(x_0)=0$ for some $x_0\in \{0,\pi\}$.
Then $V_x=0$ and $V_{xx}=0$ at $(x_0, 0)$.
Since the harmonic function $f$ in $\mcr$ has at $(x_0,0)$ a global extremum in $\mcr\cup\partial \mcr$,
the Serrin corner-point lemma \cite{Cb}
ensures that not all first and second order derivatives of $f$ can vanish at $(x_0,0)$.
We now show however by direct calculation that under the present assumptions
all first and second order derivatives of $f$ do vanish at $(x_0,0)$, which will constitute a contradiction.
Indeed, because $f=0$ on $\partial \mcr_l\cup \partial\mcr_r$, it follows that $f_y=0$ and $f_{yy}=0$
		 and also  that $f_{xx}=0$  on $\partial \mcr_l\cup \partial\mcr_r$ since $f$ is harmonic.
Thus we can see from (\ref{newf}) that each term in the expression for the derivative
with respect to $x$ of $f$ at $(x_0,0)$ contains as a factor either  $V_x$ or $V_{xx}$,
so that $f_x(x_0,0)=0$. It thus remains to examine the mixed derivative $f_{xy}(x_0,0)$.
To accomplish this task, we calculate from (\ref{sale}) that at every point in $\overline\mcr_{kh}$ we have
\be
f_{xy}=f_{yx}= \frac{(p_{xy}q+p_yq_x-p_xq_y-pq_{xy})q^2-(p_yq-pq_y)2qq_x}{q^4}.\ee
This implies in particular that at the point $(x_0,0)$, at which
$\zeta_x=V_x=\zeta_{xy}=V_{xy}=V_{xx}=V_{yy}=p=p_x=p_y=0$,
we have
\be
f_{xy}=\frac{p_{xy}}{q}.\label{qws}\ee
It is easy to see, when calculating $p_{xy}$ by differentiation in the formula $p=V_x\zeta_y-V_y\zeta_x$,
that at the point $(x_0,0)$  six out the eight terms are zero and we are left with
\be
p_{xy}= V_{xxy}\zeta_y-V_y\zeta_{xxy}.\label{pln}\ee
We now calculate $\zeta_{xxy}(x_0,0)$ by differentiating with respect to $x$ in (\ref{knxd})
and, taking into account that at $(x_0,0)$ we have  $V_x=V_{xx}=V_{xy}=0$, we obtain
\be
\zeta_{xxy}= -\Upsilon VV_{xxy}\pm \frac{(Q-2gV)^{1/2}V_yV_{xxy}}{|V_y|}.\ee
Because (\ref{ale}) is also valid at $(x_0,0)$, we see that $V_{xxy}\zeta_y=V_y\zeta_{xxy}$.
Therefore by (\ref{qws}) and (\ref{pln}) we have $f_{xy}(x_0,0)=0$.
Thus all first and second derivatives of $f$ vanish at $(x_0,0)$, a contradiction. As explained earlier, this implies that indeed $v''(0)<0$ and $v''(\pi)>0$.

\end{proof}

\section{The flow beneath a wave of small amplitude}

In this section we point out some features of the small amplitude waves whose existence was established
in Section 3. As before, we only discuss the bifurcation from the lowest eigenvalue
so that throughout this section $n=1$.    For simplicity, we refrain
from using the subscript ``1" when referring to $\lambda^\ast_{1,\pm}$ or to $\K_{1,\pm}$.

By restricting the local bifurcation curve to a sufficiently small neighborhood of $(\lb_\pm^\ast,0) \in \bdr \times \mcx$,
the condition
$v'\neq 0$ in $(0,\pi)$
will hold everywhere on it, except for the bifurcating laminar flow (with a flat surface profile)
corresponding to $s=0$.
For $s \neq 0$ the wave profiles $\mcs$  obtained thereby are  graphs that are strictly monotone
between crests and troughs and are symmetric about the crest line.
For $s>0$ the crest is located at $x=0$, while for $s<0$  the trough is at $x=0$.
As for the velocity field in the fluid, for a given fluid domain $\Omega$ beneath the surface $\mcs$ and
above the flat bed $Y=0$, the boundary-value problem (\ref{g1})-(\ref{g3})-(\ref{g4}) with $L$-periodicity
in the $X$-variable has a unique solution.   The symmetry of $\mcs$ implies that
the stream function $\psi$ is even in the $X$-variable.

Recall that by a {\it critical point} we mean a point where $\psi_Y=0$, while a critical line is
a curve of critical points, and by a {\it stagnation point} we mean a point where $\psi_X=\psi_Y=0$.
We will state the results for the local bifurcating curve $\K_{-,loc}$, corresponding to the choice of the minus sign
in (\ref{lb6}), the case of $\K_{+,loc}$ being obtained by interchanging $\pm$ in $\K_\pm$ and in
the sign of $\Upsilon$.

\begin{theorem}[The local curve $\K_{-,loc}$]     \label{local props}

 (i) A laminar flow on $\K_-$ that is a bifurcation point admits critical points if and only if $\Upsilon<0$ and
\begin{equation}\label{nss}
\frac{\tanh(kh)}{kh} \le \frac{\Upsilon^2h}{g+\Upsilon^2 h}\,.
\end{equation}
If they exists, all critical points are located on a unique horizontal line beneath the flat free surface.

(ii) If either $\Upsilon > 0$, or $\Upsilon<0$ is such that (5.50) fails, then the flows on the local bifurcating curve $\K_{-,loc}$ that
are sufficiently close to the bifurcating laminar flow present no critical points and the
streamlines foliate the fluid domain (see Figure 5).

(iii) If $\Upsilon<0$ and if the laminar flow at the bifurcation point has a critical line {\it strictly above}
the flat bed, then the nearby waves on $\K_{-,loc}$ have a cat's eye structure with critical points
(see Figure 6).

(iv)  If $\Upsilon<0$ and if the flat bed is itself a critical line of the laminar flow at the bifurcation point,
then the nearby waves on $\K_{-,loc}$ have an isolated region of flow reversal near the bed, delimited by a critical
layer (see Figure 7).
\end{theorem}

\begin{proof}
$(i)$
The stream function $\psi(X,Y)$ of the  {\it laminar flow} corresponding to the bifurcation point
$(\lb_-^\ast;\,\mu=0,w=0)$  solves (\ref{g}) with $\psi_X=0$ and
$m=m_-^*  =  \lb_-^\ast h - {\Upsilon h^2}/{2}$
throughout $\Omega=\{(X,Y):\ 0 \le Y \le h\}$. We have explicitly
\begin{equation}\label{blf}
\psi(X,Y)=\frac{\Upsilon}{2}\,Y^2 + (\lb_-^\ast -\Upsilon h)\,Y - \lb_-^\ast h
+ \frac{\Upsilon h^2}{2},\qquad 0 \le Y \le h\,,
\end{equation}
the corresponding velocity field being
\begin{equation}\label{vfl}
(\psi_Y,\,-\psi_X)=(\Upsilon Y + \lb_-^\ast - \Upsilon h,\,0),\qquad 0 \le Y \le h\,.
\end{equation}				
Since $\lb_-^\ast < 0$ by (\ref{lb6}), there are no critical points in the irrotational case $\Upsilon =0$.
On the other hand, for $\Upsilon \neq 0$, it is easily seen that
 critical points exist in the bifurcating laminar flow corresponding to $\lb_-^\ast$ if and only if
\begin{equation}\label{nss0}
h \ge \frac{\lb_-^\ast}{\Upsilon} \geq0\,.
\end{equation}
Whenever (\ref{nss0}) is satisfied, critical points lie on the critical line
\begin{equation}\label{cl}
Y=h - \frac{\lb_-^\ast}{\Upsilon}
\end{equation}
that is located {\it beneath} the flat free surface $Y=h$, since $\lb_-^\ast < 0$. Straightforward
manipulations starting from (\ref{lb6}) show that (\ref{nss0}) is satisfied if and only if
$\Upsilon<0$ and \eqref{nss} holds.

If $\Upsilon<0$, we remark that
since the function $s \mapsto \tanh(s)/s$ is a strictly decreasing bijection from $(0,\infty)$ onto $(0,1)$,
critical layers in the bifurcating laminar flow occur if and only if $k$ exceeds a certain critical value; since $L=\frac{2\pi}{k}$, this means that critical layers occur whenever the wavelength is sufficiently small.

$(ii)$ If $\Upsilon > 0$, then (5.52) shows that the maximum of $\psi_Y$ throughout
the laminar flow at the bifurcation point is attained at the flat free surface $Y=h$, where $\psi_Y(h)=\lb_-^\ast <0$,
so that $\psi_Y<0$ throughout the flow. If $\Upsilon<0$ is such that (5.50) fails, then (5.52) with the choice
$\lb_-^\ast$ shows that the maximum of $\psi_Y$ throughout
the laminar flow at the bifurcation point is attained on the flat bed $Y=0$, with the inequality opposite to (5.50)
ensuring that $\psi_Y<0$ there. Consequently, we have again that $\psi_Y<0$ throughout
the laminar flow at the bifurcation point. Using (\ref{vel}), we infer that this inequality will
persists along the curve $\K_-$ of non-trivial solutions, provided that these are sufficiently close to the
bifurcating laminar flow. All these waves will therefore not present critical points in the flow, with
the streamlines  in this case providing a foliation of the fluid domain. Typical streamline patterns for
the laminar flows with bifurcation parameter $\lb_-^\ast$ are depicted in Figure 5.

\begin{figure}   \begin{center}$
\begin{array}{cc}
\includegraphics[width=3in]{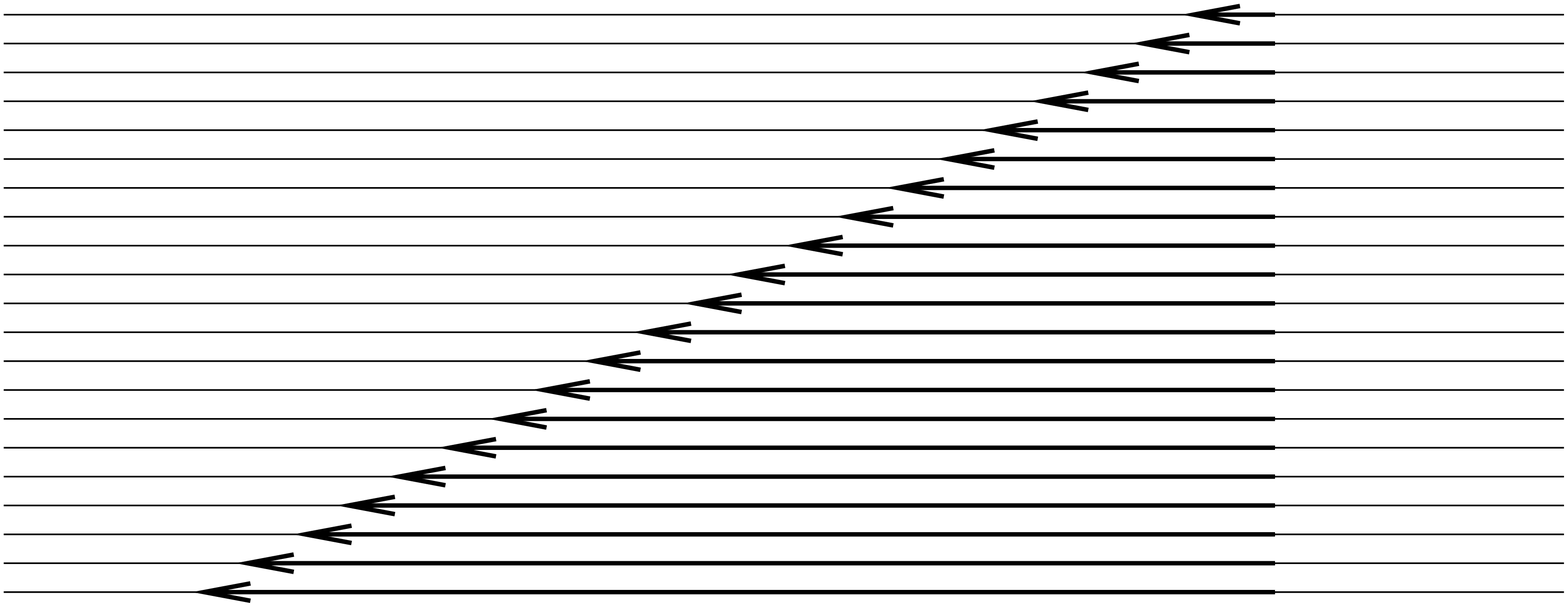} & \includegraphics[width=3in]{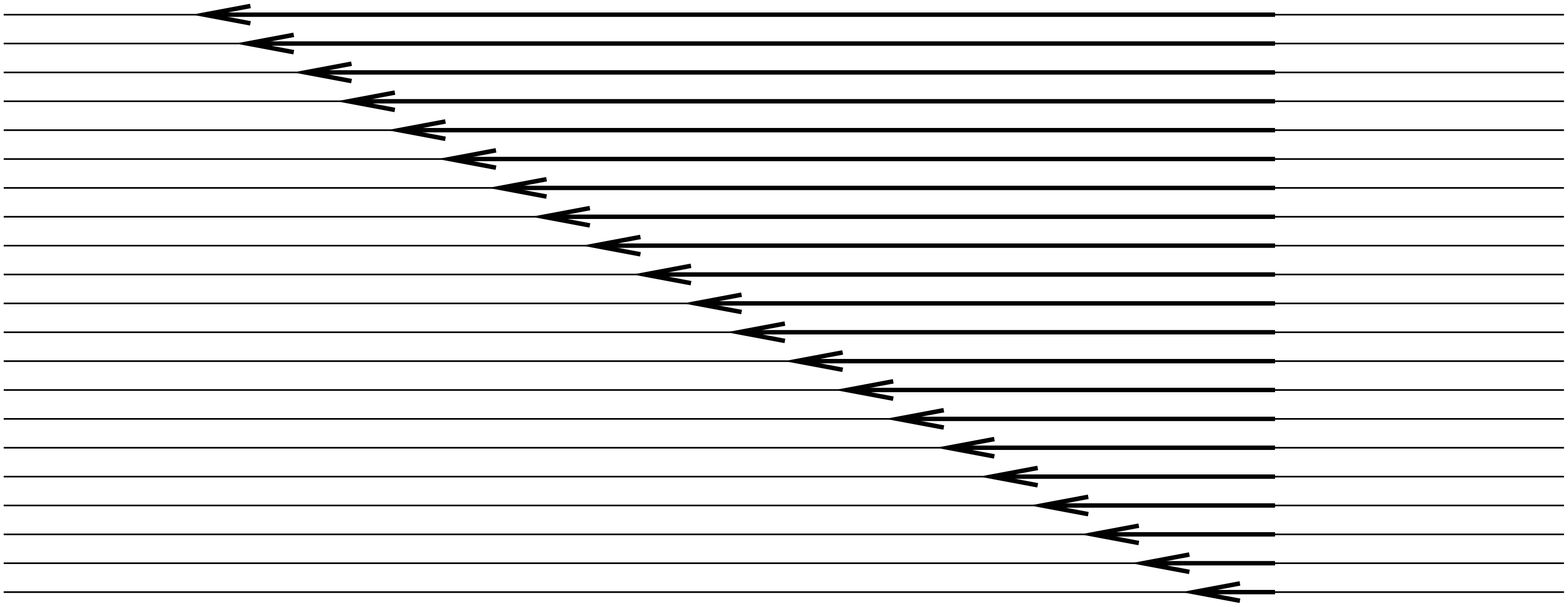}
\end{array}$
\end{center}
 \includegraphics[scale=0.33]{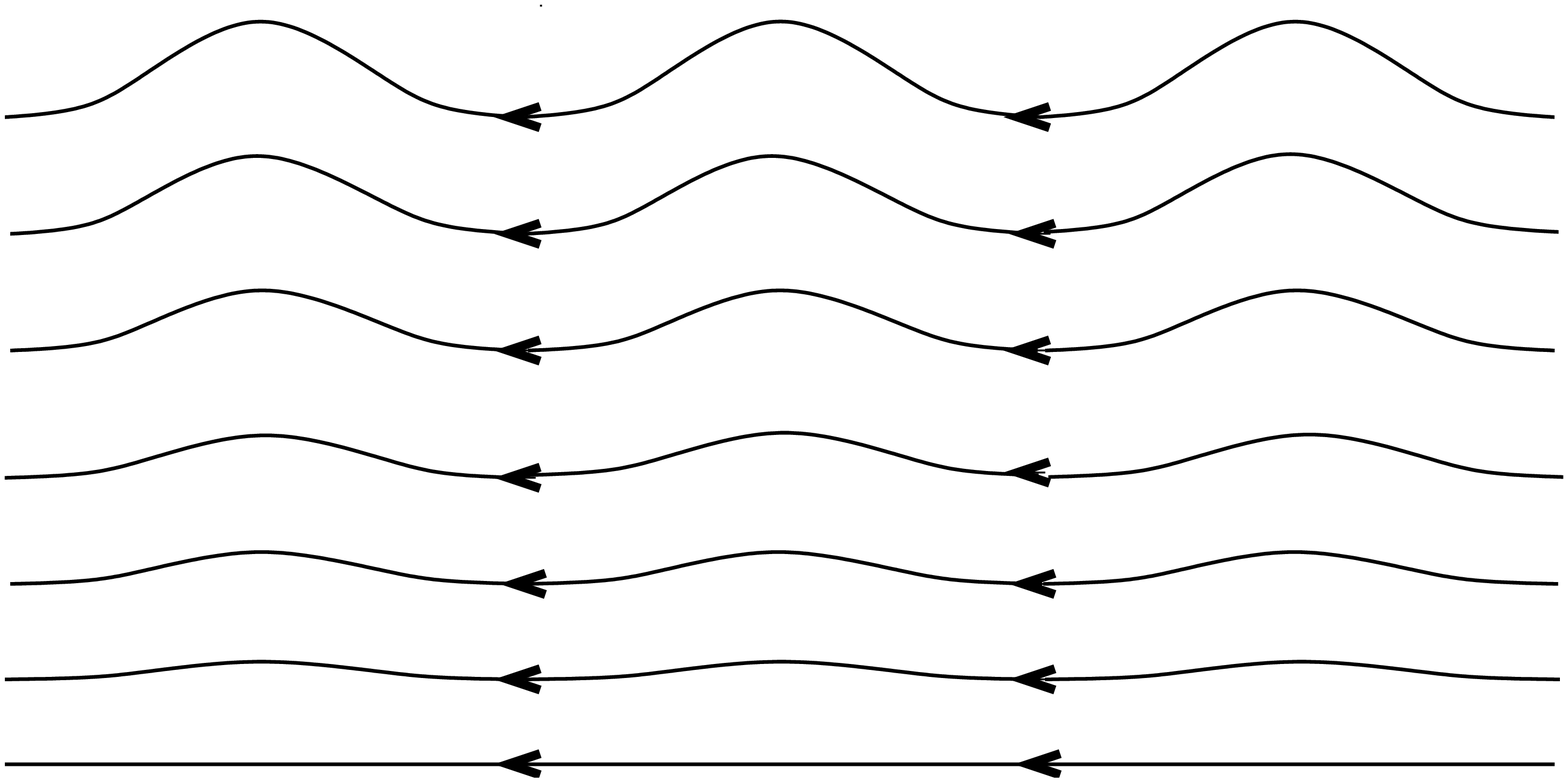}
\caption{Streamlines of a flow of constant vorticity $\Upsilon$, in the absence of critical points: On the top left laminar flows for $\Upsilon > 0$, on the top right laminar flows
for $\Upsilon<0$, and below nearby waves of small amplitude, lying in function space on the local bifurcation curve. In both settings
the horizontal fluid velocity is negative throughout the flow.}  \end{figure}

On the other hand, for a bifurcating laminar flow with critical points in the case  $\Upsilon<0$,
the streamline pattern for a
wave corresponding to a non-trivial solution on $\K_{-,loc}$ is in
marked contrast to that familiar from flows without critical points.
Indeed, the previous considerations
show that the bifurcation parameter is $\lb_-^\ast<0$.
Moreover, from (\ref{vfl}) we get that the horizontal fluid velocity $\psi_Y$ of this laminar flow takes on
the value $\psi_Y(h)=\lb_-^\ast<0$ at the surface $Y=h$, with $\psi_Y(0)=\lb_-^\ast - \Upsilon h \ge 0$
on the bed $Y=0$, the latter inequality being a restatement of (\ref{nss0}).

\begin{figure}

\includegraphics[width=3.6in]{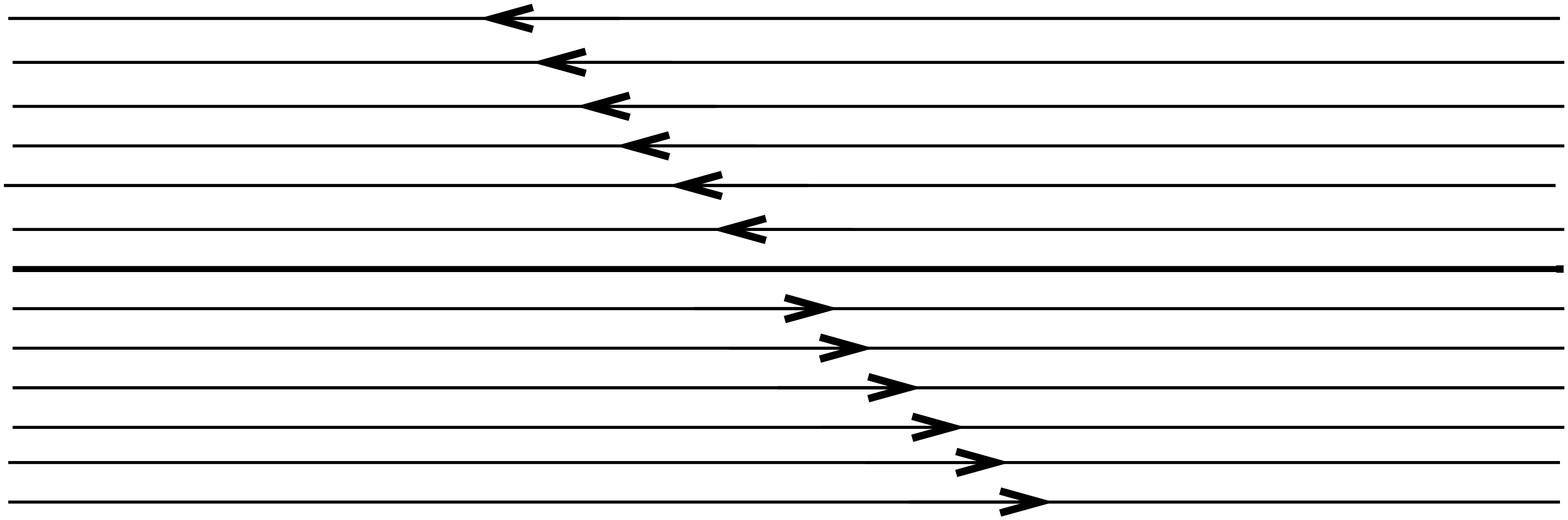}

\bigskip

\includegraphics[scale=0.36]{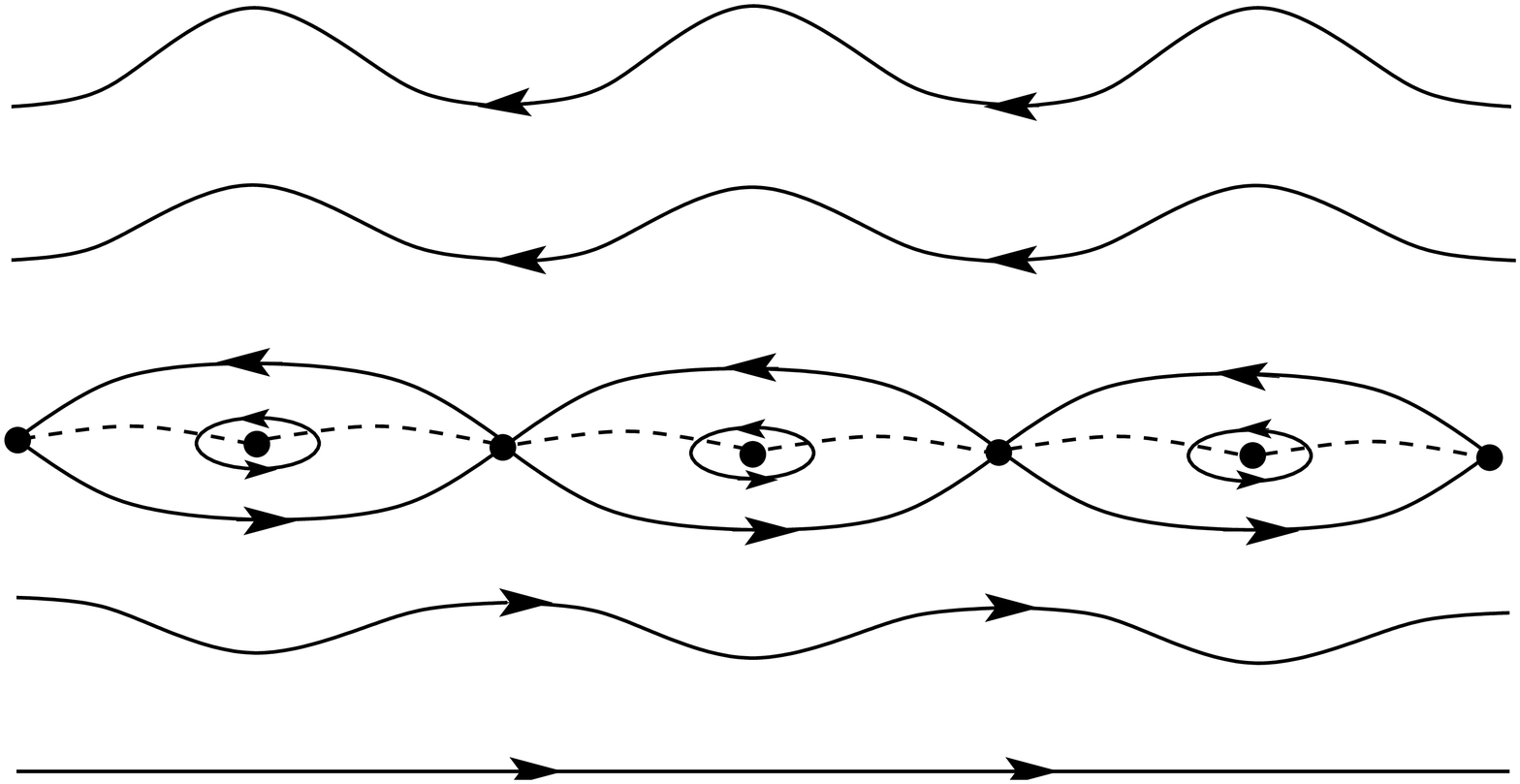}
\caption{Flows with constant vorticity $\Upsilon<0$, with critical points above the flat bed: On top, for a laminar flow the horizontal critical layer consists of stagnation points
that mark the flow reversal, while below, for a nearby wave of small amplitude (lying in function space on the local bifurcation curve)
only the points of the critical layer that are beneath the crest or the trough are stagnation points. In the moving
frame the stagnation point beneath the crest is surrounded by closed streamlines (Kelvin's `cat's eye' flow pattern). The critical
layer, depicted by the dashed curve, delimits an upper region of negative horizontal fluid velocity, and beneath it
the flow direction is reversed.}   \end{figure}

In the specific scenario $(iii)$
the stream function $\psi$ of the laminar bifurcating flow is such that $\psi_{YY}<0$ throughout the fluid, with
$\psi_Y<0$ near the surface and $\psi_Y>0$ near the bed. All these inequalities will persist for the stream functions associated
to the flows along the curve $\K_-$ of non-trivial solutions, provided that these flows are sufficiently close to the
bifurcating laminar flow. Let $Y=\eta(X)$ be the free surface of such a flow, with the wave crest located at $X=0$ and
the wave troughs at $X=\pm L/2$. At every fixed $X$, the corresponding real-analytic stream function $\psi$ is such
that the function $Y \mapsto \psi(X,Y)$ is strictly concave on $[0,\eta(X)]$, with $\psi_Y(X,0)>0$ and $\psi_Y(X,\eta(X))<0$.
Consequently, there is a unique maximum point $Y_0(X) \in (0,\,\eta(X))$ of this function, with $\psi_Y(X,Y_0(X))=0$, and $\psi_Y(X,Y) >0$
for $0 \le Y < Y_0(X)$, while $\psi_Y(X,Y)<0$ for $Y \in (Y_0(X),\,\eta(X)]$. Since $\psi_{YY}<0$ throughout the flow, the implicit function
theorem \cite{BT} ensures that the critical layer $\{(X,Y) \in \Omega:\ \psi_Y(X,Y)=0\}$ coincides with
the graph of the real-analytic map $X \mapsto Y_0(X)$.

On the other hand, by (\ref{g3}) we have  $\psi(X,\eta(X))=0$ for all $X \in \bdr$. Differentiating this relation with respect to $X$ and
taking into account the fact that $\psi_Y(X,\eta(X))<0$ while $\eta'(X)<0$ for $X \in (0,L/2)$ in view of (\ref{nmbv}),
we deduce that $\psi_X(X,\eta(X))<0$
for all $X \in (0,L/2)$. The boundary condition (\ref{g4}) and the fact that $\psi$ is even and $L$-periodic in the $X$-variable
yield $\psi_X(0,Y)=0$ for $Y \in [0,\eta(0)]$, $\psi_X(X,0)=0$ for $X \in [0,L/2]$, and $\psi_X(L/2,Y)=0$ for
$Y \in [0,\eta(L/2)]$. Applying the strong maximum principle to the harmonic function $\psi_X$ in the domain
$$\Omega_+=\{(X,Y):\ 0 < X < L/2,\ 0 < Y < \eta(X)\},$$
we conclude that $\psi_X(X,Y)<0$ in $\Omega_+$. The fact that $\psi$ is
even in the $X$-variable yields $\psi_X(X,Y)>0$ in
$$\Omega_-=\{(X,Y):\ -L/2 < X < 0,\ 0 < Y < \eta(X)\}.$$
These considerations show that
the only stagnation points of this flow are the points on the critical layer that lie beneath the wave troughs and crests.

To describe qualitatively
the streamline pattern of the flow, notice first that $\partial_X\,\psi(X,Y_0(X))=\psi_X(X,Y_0(X))$ is
positive in $\Omega_-$
and negative in $\Omega_+$. Consequently the quantity $M(X)=\psi(X,Y_0(X))$, representing
the maximum of the function
$Y \mapsto \psi(X,Y)$ on $[0,\eta(X)]$, attained uniquely on the critical layer, is strictly increasing
as $X$ runs from $-L/2$ to $0$
and strictly decreasing as $X$ runs from $0$ to $L/2$. Set $M_-=M(-L/2)=M(L/2)$ and $M_+=M(0)$.
We know that $\psi_{Y}<0$ in the region
\[\Omega^+=\{(X,Y):\ -L/2 < X < L/2,\ Y_0(X) < Y < \eta(X)\}\]
above the critical layer, while $\psi_{Y}>0$ in the region
\[\Omega^-=\{(X,Y):\ -L/2 < X < L/2,\ 0 < Y < Y_0(X)\}\]
beneath the critical layer. Since $\max\,\{0,-m\}<M_-<M(X)=M(-X)<M_+$ for
every $X \in (-L/2,0)$, the implicit function theorem yields that in $\Omega_-^+:=\Omega_-\cap \Omega^+$ the level set $[\psi=M_-]$
consists of the
graph of a smooth strictly increasing curve $C_+$, joining the points $(-L/2,Y_0(-L/2))$ and $(0,Y_+)$,
where $Y_+ \in (Y_0(0),\,\eta(0))$ is the
unique solution to $\psi(0,Y)=M_-$ in this interval.
Similarly, in $\Omega_-^-:=\Omega_-\cap \Omega^-$ the level set $[\psi=M_-]$  consists of the graph of
a smooth strictly decreasing curve $C_-$ that joins $(-L/2,Y_0(-L/2))$ to $(0,Y_-)$,
where $Y_0 \in (0,Y_0(0))$ is the unique
solution to $\psi(0,Y)=M_-$ in this interval. The level set $[\psi=M_-]$ in $\Omega_+$ is the mirror image of the union of these
two curves. In the closure of $\Omega_- \cup \Omega_+$, the level set $[\psi=M_+]$ consists of the stagnation point $(0,Y_0(0))$,
while each level set $[\psi=\gamma]$ with $\gamma \in (M_-,M_+)$ is a closed curve encircling $(0,Y_0(0))$. For $\gamma \in (0,M_-)$,
the implicit function theorem yields that the level set $[\psi=\gamma]$ in $\Omega_-^+$ is the graph of a strictly increasing
function, located between $C_+$ and the free surface, $[\psi=\gamma]$ in $\Omega_-^-$ is the graph of a strictly decreasing
function, located between $C_-$ and the flat bed, while $[\psi=\gamma]$ in $\Omega_+$ is the mirror image of the union of
 these two curves. The cat's eye type flow pattern depicted in Figure 6 emerges.

\begin{figure}
\includegraphics[width=2.7in]{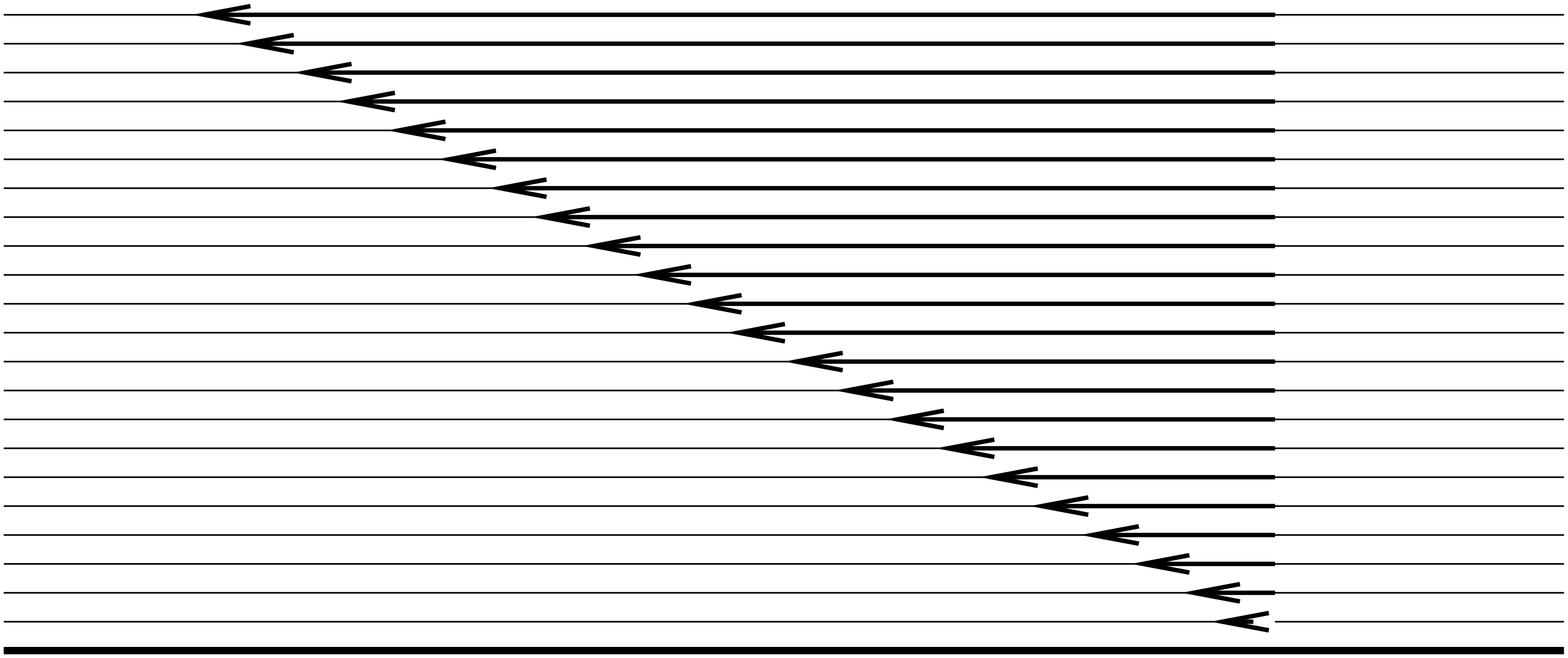}

\bigskip

\includegraphics[scale=0.33]{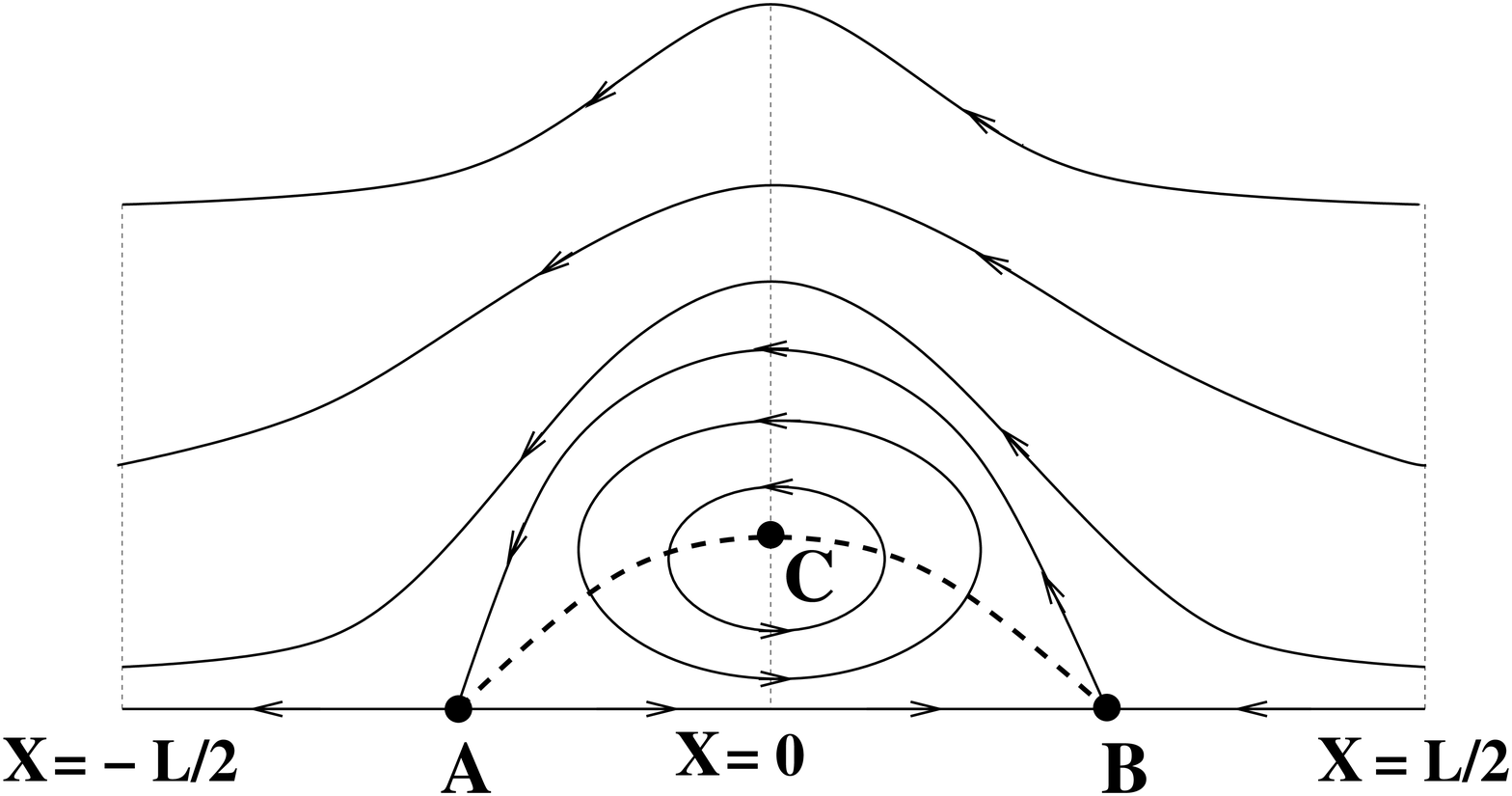}

\caption{Flows with constant vorticity $\Upsilon <0$, admitting in the moving frame stagnation points on the flat bed:
On top, a laminar flow for which all points on the bed are stagnation points, and below, a nearby wave of small amplitude (lying in
function space on the local bifurcation curve) presents in a periodicity window only three
stagnation points (identified as $A,\,B,\,C$ in the figure), connected by the critical layer curve
(depicted by the bold dashed curve) that encloses a near-bed region of flow-reversal. The stagnation point located beneath the
wave crest is surrounded by closed streamlines.}   \end{figure}

$(iv)$ We now consider the case $\Upsilon<0$ in which the critical points of the bifurcating laminar flow (with a flat free surface) are confined to the flat bed,
 that is, $\lb_-^\ast - \Upsilon h=0$. The proof of Theorem 3 shows that along the solution curve $\K_-$ the non-trivial solutions
sufficiently close to the bifurcating laminar flow correspond to $v_s(x)=h+s\,\cos(x)+o(s)$ in $C^{2,\alpha}_{2\pi, e}(\R)$ and $m(s)=m_-^\ast+o(s)$ for $s>0$ small
enough, where $m_-^\ast=\Upsilon h^2/2$ due to (3.18), (3.3) and $\lambda_-^\ast=\Upsilon h$. From (A.3) we deduce that
$$u_s(x)=\frac{x}{k} + s\,\coth(kh)\,\sin(x)+o(s)\quad\text{in $C^{2,\alpha}_{2\pi}(\R)$}.$$
The corresponding solution of (2.9) is given by
\begin{equation}\label{nm1}
V^s(x,y)=\frac{y+kh}{k} + s\,\cos(x)\,\frac{\sinh(y+kh)}{\sinh(kh)} + o(s)\qquad  \text{in $C^{2,\alpha}(\overline{{\mathcal R}_{kh}}$)},
\end{equation}
so that, in view of (2.10), the harmonic conjugate of $-V^s$ on ${\mathcal R}_{kh}$ is
$$U^s(x,y)=\frac{x}{k} + s\,\sin(x)\,\frac{\cosh(y+kh)}{\sinh(kh)} + o(s)\qquad \text{in $C^{2,\alpha}(\overline{{\mathcal R}_{kh}}$)}.$$
Furthermore, we have that the solution to (2.13) is given by
\begin{equation}\label{nm2}
\zeta^s(x,y)=-\,\Upsilon\, sh\,\cos(x)\,\frac{\sinh(y+kh)}{\sinh(kh)} + o(s)\qquad \text{in $C^{2,\alpha}(\overline{{\mathcal R}_{kh}}$)}.
\end{equation}
For $s>0$ small enough, the expansion
$$(u_s'(x),\,v_s'(x))=\left( \frac{1}{k} + s\,\coth(kh)\,\cos(x)\,+o(s),\,-\,s\,\sin(x) + o(s)\right)$$
in $C^{1,\alpha}_{2\pi}(\R)$ ensures that the free surface is the graph of a function $Y=\eta_s(X)$ with $\eta_s'(X) <0$ for $X \in (0,L/2)$. Throughout the bifurcating laminar flow
we have that $\psi_{YY} =\Upsilon<0$, while $\psi_Y(h)=\lb_-^\ast<0$, so that for nearby waves we will have $\psi_{YY}^s<0$ throughout the
flow and $\psi_Y^s<0$ near the free surface $Y=\eta_s(X)$. Recalling that $\eta_s'(X)<0$ for $X \in (0,L/2)$, by differentiating the relation $\psi^s(X,\,\eta_s(X))=0$,
valid due to (2.2b), we deduce that $\psi_X^s(X,\,\eta_s(X))<0$ for $X \in (0,L/2)$. On the other hand, $\psi_X^s(X,0)=0$ for all $X \in \R$ due
to (2.2c), while the symmetry properties ensure $\psi_X^s(0,Y)=0$ for $0 \le Y \le \eta_s(0)$ and $\psi_X^s(L/2,Y)=0$ for
$0 \le Y \le \eta_s(L/2)$. The function $\psi_X^s$ being harmonic in the domain $\{(X,Y):\ 0<X<L/2,\,0<Y<\eta_s(X)\}$ by (2.2a), the
maximum principle permits us to deduce that $\psi_X^s(X,Y)<0$ throughout the domain, while Hopf's maximum principle yields
\begin{equation}\label{nm3}
\psi_{XY}^s(X,0) <0\,,\qquad X \in (0,L/2)\,.
\end{equation}
To elucidate the behaviour of $\psi_Y^s$ in the closure of the domain $\{(X,Y):\ -L/2<X<L/2,\,0<Y<\eta_s(X)\}$, note that (2.21) yields
\begin{equation}\label{nm4}
\psi_Y^s(U^s(x,-kh),\,0)=-\,\frac{\Upsilon\,skh\,\cos(x)}{\sinh(kh)}+o(s)\qquad\text{in $C^{1,\alpha}_{2\pi}(\R)$,}
\end{equation}
since $V^s(x,-kh)=0$ by (2.9c), while by (\ref{nm1}) and (\ref{nm2}),\[ V_Y^s(x,-kh)=\frac{1}{k} + s\,\frac{\cos(x)}{\sinh(kh)} + o(s)\qquad\text{in $C^{1,\alpha}_{2\pi}(\R)$,} \] and
\[\zeta_Y^s(x,-kh)=-\Upsilon\,sh\,\frac{\cos(x)}{\sinh(kh)} + o(s)\qquad\text{ in $C^{1,\alpha}_{2\pi}(\R)$}.\] For $x=0$ and $x=\pm \pi$ in (\ref{nm4}) we
get $\psi_Y^s(0,0)>0$ and $\psi_Y^s(\pm L/2,0)<0$, respectively. Taking into account (\ref{nm3}) and the fact that $X \mapsto \psi^s(X,0)$ is
even, we deduce the existence of some $X_0 \in (0,L/2)$ with $\psi_Y^s(\pm X_0,0)=0$ and $\psi_Y^s(X,0)>0$ for
$X \in (-X_0,X_0)$, $\psi_Y^s(X,0)<0$ for $X \in [-L/2,-X_0) \cup (X_0,L/2]$. Denote the points $(-X_0,0)$ and $(X_0,0)$ by $A$ and $B$, respectively.
Along the vertical segment $\{(X,Y):\ 0 \le Y \le \eta_s(X)\}$ we know that $\psi_{YY}^s(X,Y)<0$, with $\psi_Y^s(X,\eta_s(X))<0$. Thus
$\psi_Y^s(X,Y)<0$ for all $X \in  [-L/2,-X_0] \cup [X_0,L/2]$ and $Y \in (0,\eta_s(X)]$. On the other hand, for every $X \in (-X_0,X_0)$
there exists a unique $Y_0^s(X) \in (0,\,\eta_s(X))$ such that $\psi_Y^s(X,\,Y_0^s(X))=0$, with $\psi_Y^s(X,Y) <0$ for $Y \in (Y_0^s(X),\,\eta_s(X)]$ and
$\psi_Y^s(X,Y) >0$ for $Y \in [0,\,Y_0^s(X))$. Denote the stagnation point $(0,\,Y_0^s(0))$ by $C$. The curve $X \mapsto Y_0^s(X)$ with $X \in [-X_0,\,X_0]$ is
the critical layer and $A$, $B$, $C$ are the stagnation points of the flow in the closure of the domain $\{(X,Y):-L/2<X<L/2,\,0<Y<\eta_s(X)\}$. Recall from (2.2) that
$\psi^s(X,\,\eta_s(X))=0$ while $\psi^s(X,0)=-m(s)>0$. For a fixed $X \in [-L/2,\,-X_0] \cup [X_0,\,L/2]$, the function $Y \mapsto \psi^s(X,Y)$ is strictly decreasing
on $[0,\,\eta_s(X)]$, while for $X \in (-X_0,X_0)$, the function $Y \mapsto \psi^s(X,Y)$ attains its maximum $M^s(X)>-m(s)$ on $[0,\,\eta_s(X)]$ at $y=Y_0^s(X)$, being
strictly monotone on either side of $Y_0^s(X)$. These considerations suffice to infer the full qualitative flow pattern -- see Figure 7. In particular, the stagnation
point $C$ is surrounded by closed streamlines.\end{proof}

\section{Further results and conjectures}

In a sequel to this paper, we will prove the following results:
\begin{itemize}
\item If ${\mathcal K}$ becomes unbounded in $\R\times\R\times C^{2,\alpha}_{2\pi}({\mathbb R})$, then $v'$ is
unbounded in $L^2_{2\pi}({\mathbb R})$;
\item For all waves in ${\mathcal K}$, the free surface ${\mathcal S}$ is a real-analytic curve.
\end{itemize}
Furthermore, we have a conjecture about the furthest boundary of the
global curve ${\mathcal K}$ that is more specific than Theorem 5.  Our conjecture is
that at this boundary we reach
\begin{itemize}
\item either a wave with a stagnation point and a corner of $120^\circ$
at its crest whose surface may be overhanging or a graph (see Figure 8)
\item or a wave that has no stagnation point but its surface is overhanging with
self-intersections on the trough line (see Figure 9).
\end{itemize}
This conjecture is supported by some analysis and by the numerical
simulations in [21,49].

It would also be interesting to extend our analysis in Section 5 of
the nature of a flow beneath the wave profile from the case of
small-amplitude waves to larger waves.

\begin{figure}\label{lw1}   \begin{center}
$\begin{array}{cc}
\includegraphics[width=3.3in]{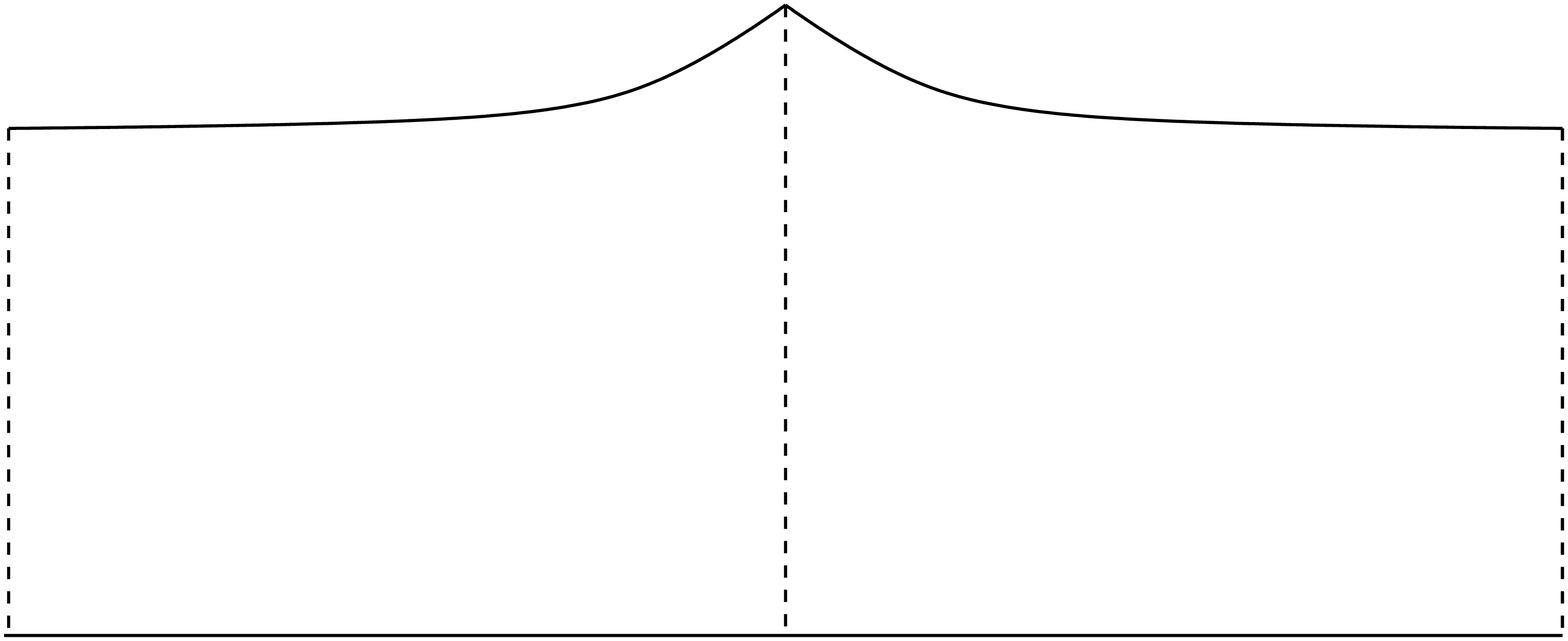} &     \includegraphics[scale=0.19]{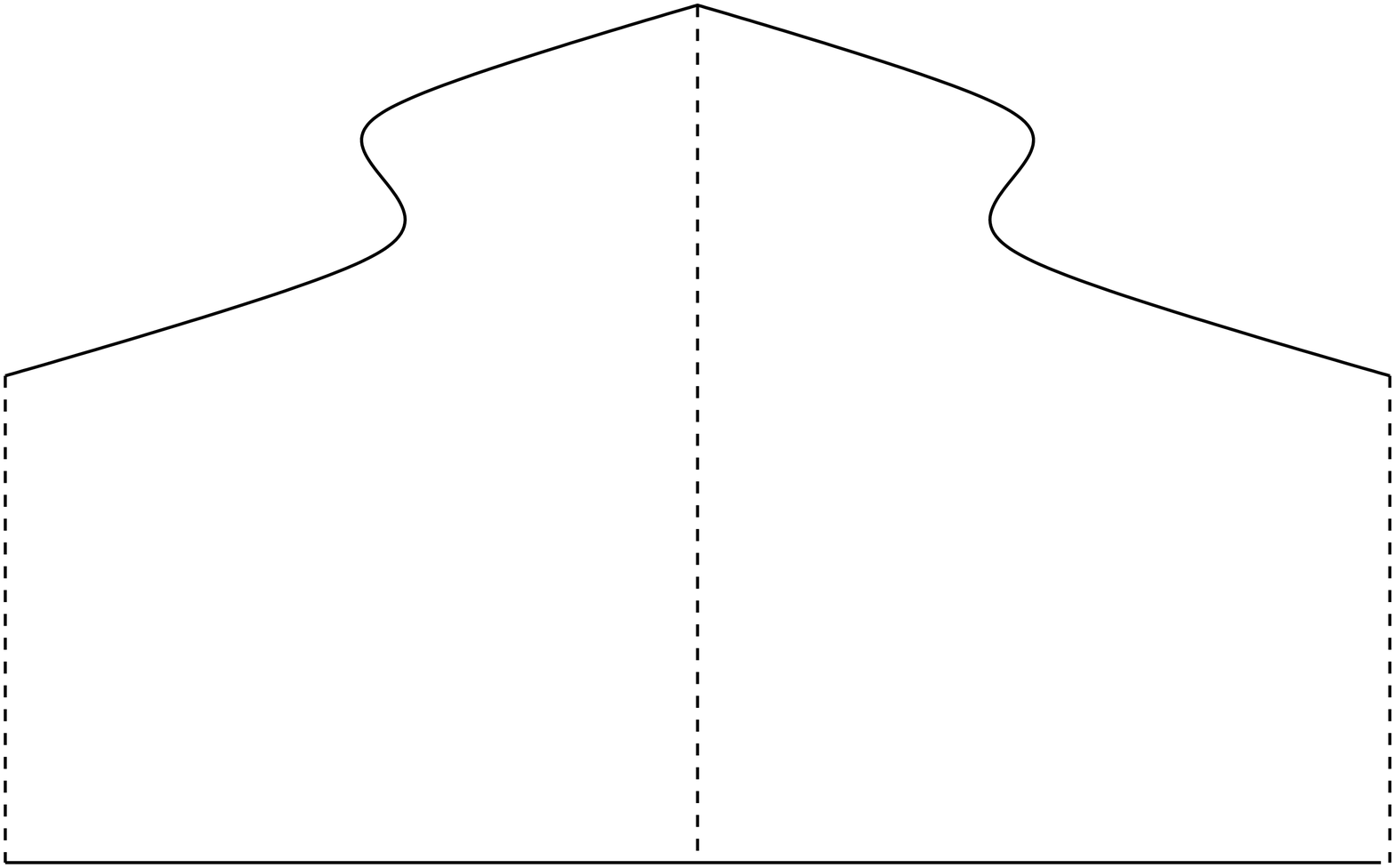}   \end{array}$
\end{center}
\caption{Waves with stagnation points and corners of $120^\circ$ at their crests: overhanging profiles (on
the right) and profiles that are graphs (on the left).}     \end{figure}

\begin{figure}\label{lw3}   \centering
\includegraphics[scale=0.2]{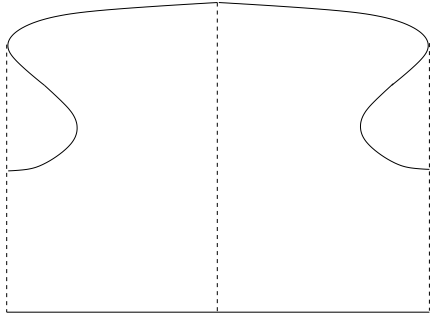}
\caption{Overhanging wave with self-intersections on the trough line.}   \end{figure}

\begin{appendix}
\numberwithin{equation}{section}

\section{The periodic Hilbert transform $\C_d$ on a strip}

We discuss the conjugation and Dirichlet-Neumann operators acting on periodic functions on a strip.
The Dirichlet-Neumann operator
$\mcg_{d}$ for the strip $\mcr_{d}$ is defined by
\begin{equation}\label{G}
\big(\mcg_d(w)\big)(x)=W_y(x,0),\qquad w \in C^{p,\alpha}_{2\pi}(\R)\quad\hbox{with}\quad p \ge 1 \quad\hbox{an integer},
\end{equation}
where $W \in C_{2\pi}^{p,\alpha}(\overline{\mcr_d})$ is the unique
solution to the boundary-value problem
\begin{equation}\label{BC}
\left\{\begin{array}{l}
\Delta W =0\quad\hbox{in}\quad \mcr_d,\\[0.1cm]
W(x,0)=w(x),\qquad x \in \R,\\[0.1cm]
W(x,-d)=0,\qquad x \in \R\,.
\end{array}\right.
\end{equation}
$\mcg_{d}$ is a bounded linear operator from $C^{p,\alpha}_{2\pi}(\R)$ to $C^{p-1,\alpha}_{2\pi}(\R)$, given by
$$\big(\mcg_d(w)\big)(x)= \frac{[w]}{d}\,+\,\sum_{n=1}^\infty na_n \coth (nd)\cos(nx)+\sum_{n=1}^\infty nb_n \coth (nd)\sin(nx)$$
where
$w \in C^{p,\alpha}_{2\pi}(\R)$ has the Fourier series expansion
$$w(x)=[w]\,+\,\sum_{n=1}^\infty a_n \cos (nx)+\sum_{n=1}^\infty b_n \sin (nx).$$
						The conjugation operator $\mcc_d$ is defined as
\begin{equation}\label{hfs}
\big(\mcc_d(w)\big)(x)=\sum_{n=1}^\infty a_n \coth
 (nd)\sin(nx)\,-\,\sum_{n=1}^\infty b_n\coth (nd)\cos (nx)\,,
 \end{equation}
 for $2\pi$-periodic functions $w \in C^{p,\alpha}_{2\pi,\circ}(\R)$ of zero mean, $[w]=0$, having the Fourier series expansion
$$w(x)=\sum_{n=1}^\infty a_n \cos (nx)\,+\,\sum_{n=1}^\infty b_n \sin (nx)\,.$$
For any integer $p \ge 0$ and any $\alpha \in (0,1)$, the operator $\mcc_d$ is a bounded invertible operator from the class $C^{p,\alpha}_{2\pi,\circ}(\R)$ into itself.
The two operators are related by means of the identity
\begin{equation}\label{CG}
\mcg_d(w)=\frac{[w]}{d}\,+\,\Big(\mcc_d(w\,-\,[w])\Big)'=\frac{[w]}{d}\,+\,\mcc_d(w'),\qquad w \in C_{2\pi}^{p,\alpha}(\R),
\end{equation}
that holds for all integers $p \ge 1$.    Given $W$ as in \eqref{BC},
let $Z \in C_{2\pi}^{p,\alpha}(\overline{\mcr_d})$ be the harmonic function in $\mcr_d$,
uniquely determined up to a constant, such that $Z+iW$ is holomorphic in $\mcr_d$. Then
$Z(x,y)=\displaystyle\frac{[w]}{d}\,x+Z_0(x,y)$ throughout $\mcr_d$, for a harmonic function $(x,y) \mapsto Z_0(x,y)$ that is
$2\pi$-periodic in the $x$-variable (see \cite{CV}). The function $Z_0$ being unique up to an additive constant, we normalize it by
requiring that $x \mapsto Z_0(x,0)$ has zero mean over one period. Then $x \mapsto Z_0(x,y)$ has zero mean for
every $y \in [-d,0]$. The restriction of this particular harmonic conjugate of $W$ to $y=0$ is given by
\begin{equation}\label{H}
Z(x,0)=\frac{[w]}{d} \,x+\mcc_d(w-[w]),\qquad x \in \bdr\,.
\end{equation}

Let $L^2_{2\pi,\circ}(\bdr)$ be the space of $2\pi$-periodic locally square
integrable functions of one real variable, with zero mean over one period.
The operator $\mcc_d$ can be extended by
complex-linearity to complex-valued functions in
$L^2_{2\pi,\circ}(\bdr)$, being characterized by its action on the
trigonometric system $\{e^{int}\}_{n\in\mathbb{Z}\setminus\{0\}}$ as
\begin{equation}\label{c}
 \mcc_d(e^{int})=-{i}\,{\coth(nd)}\,e^{int},\qquad n \in \bdz
\setminus \{0\}.
\end{equation}
It is a skew-adjoint operator.
Let $\mcc$ denote the standard periodic Hilbert transform \cite{BT,To}, defined by
\begin{equation}\label{c1}
\mcc(e^{int})=-{i}\,\textnormal{sgn} (n)\,e^{int},\qquad n \in \bdz
\setminus \{0\}\,.
\end{equation}
It is well-known that  $\mcc$ has a pointwise almost everywhere representation
as a singular integral
\begin{equation}\label{ht}
\big(\mcc (w)\big)(t)=\frac{1}{2\pi}\,PV\,\int_{-\pi}^\pi
{\cot\,\left(\frac{t-s}{2}\right)}{w(s)}\,ds,
\end{equation}
where $PV$ denotes a principal value integral \cite{St}, which is instrumental in the investigation of the
structural properties of the operator $\mcc$. Writing
\be\label{cc}
\mcc_d=\mcc+\mathcal{K}_d,
\ee
we see that the operator $\mathcal{K}_d$ corresponds to the Fourier multiplier
operator on $L^2_{2\pi,\circ}(\bdr)$ given by
\begin{equation}\label{kap}
 \Big\{w = \sum_{n \in
\bdz\setminus\{0\}} c_n\,e^{int}\Big\} \mapsto \Big\{\sum_{n \in
\bdz\setminus \{0\}}
-i\,\textnormal{sgn}(n)\,\lambda_n\,c_n\,e^{int}\Big\}\,,
\end{equation}
with $\lambda_n=\displaystyle\frac{2}{e^{2|n|d}-1}$ for $|n| \ge 1$. Since
$\displaystyle\sum_{n\in\bdz\setminus\{0\}} |n|^{2p}\lambda_n^2<\infty$ for every integer $p \ge 0$, the
function $\kappa_d\in L^2_{2\pi,\circ}(\bdr)$ given by
\begin{equation}\label{kapad}
\kappa_d(t)=\sum_{n\in\bdz\setminus\{0\}}-i\,\textnormal{sgn}(n)\,\lambda_n
\,e^{int} =\sum_{n=1}^{\infty}2\lambda_n\sin(nt),
\quad{t\in\bdr},
\end{equation}
is of class $C^\infty$ (see \cite{DM}). From (\ref{kap}) we infer that $\mck_d(w)$ is the
convolution of $w$ with this smooth function $\kappa_d$, that is,
\begin{equation}\label{c2}
\big(\mck_d (w)\big)(t)=\frac{1}{2\pi}\int_{-\pi}^\pi \kappa_d
(t-s)\,w(s)\,ds,\qquad t\in\bdr.
\end{equation}
This representation is useful in establishing the following commutator estimate (see \cite{CV} for the proof).

\begin{lemma}\label{lm}
  If $f \in C^{j,\alpha}_{2\pi,\circ}(\bdr)$ and $g\in C^{j-1,\alpha}_{2\pi,\circ}(\bdr)$ with $j\in\N$ and $\alpha \in
(0,1)$, then  $\Big(f\,\mcc_d(g)-\mcc_d(fg)\Big) \in
C^{j,\delta}_{2\pi}(\bdr)$ for all $\delta \in (0,\alpha)$ (with the inequality $\delta<\alpha$ being sharp), and there exists a constant $C=C(j,\alpha,\delta)$ such that
\[||f\,\mcc_d(g)-\mcc_d(fg)||_{j,\delta}\leq C||f||_{j,\alpha}||g||_{j-1, \alpha}.\]
\end{lemma}

In order to complete the description of the periodic Hilbert transform in a strip, we now derive the analogue
of the fundamental formula (\ref{ht}).

\begin{lemma}
For any
$w \in C_{2\pi,\circ}^{1,\alpha}(\bdr)$,  we have
\begin{equation}\label{HE0}
\big(\mcc_d(w)\big)(x)=  \displaystyle\frac{1}{2d} \,PV\,\int_\bdr \Big\{ \tanh\Big(\frac{\pi s}{2d}\Big) +
\coth\Big(\frac{\pi(x-s)}{2d}\Big)\Big\}\,w(s)\,ds
\end{equation}
						and also
\begin{equation}\label{HE3}
\big(\mcc_d(w)\big)(x)= \frac{1}{2\pi d} \,\int_{-\pi}^\pi \Big(\pi g_d(s)+2s\Big)\,w(s)\,ds
+  \frac{1}{2d} \,PV\,\int_{-\pi}^\pi \beta_d(x-s)\,\,w(s)\,ds
\end{equation}
for some $2\pi$-periodic functions $g_d$ and $\beta_d$,
where $g_d$ is continuous on $\bdr$ and $\beta_d$ is continuous on $\bdr \setminus 2\pi{\mathbb Z}$,
while $s \mapsto \beta(s)-\displaystyle\coth\Big(\frac{\pi (s-2\pi k)}{2d}\Big)$
is continuous at $s=2\pi k$ with $k \in {\mathbb Z}$.
\end{lemma}
						We do not pursue
the quest of providing an explicit formula for $g_d$ and $\beta_d$ since the intricacy of
the hypergeometric point evaluations available for $g_d(0)$ and $\beta_d(0)$  is indicative of the level of complexity involved \cite{Dieck}.
To illustrate the usefulness of the representation (\ref{HE3}), note that term by term differentiation of their series
representations from above yield that $g$ and $\beta$ are decreasing, respectively increasing, on $(-\pi,\pi)$. This information
is not immediately obtained from (\ref{kapad}).

\begin{figure} \centering    \includegraphics[width=12cm]{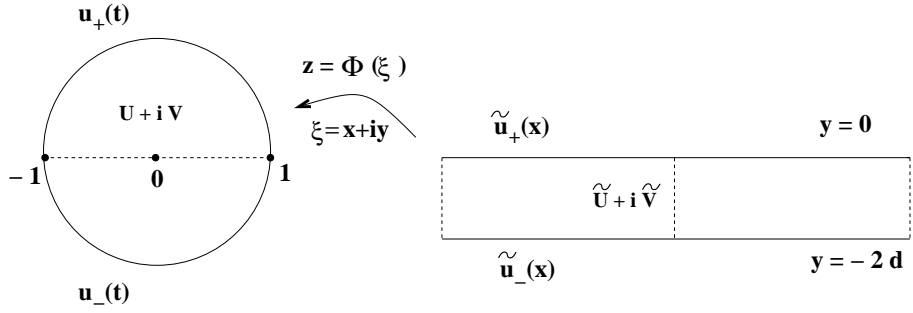}
\caption{\small The correspondence between boundary values of harmonic functions on a strip and on the unit disk.}  \end{figure}

\begin{proof}
Given $d>0$, the function $\phi(\xi)=\displaystyle\frac{e^{\frac{\pi}{2d}\,(\xi+id)}-1}{e^{\frac{\pi}{2d}\,(\xi+id)}+1}=\tanh\Big(\frac{\pi}{4d}\,(\xi+id)\Big)$
with $\xi=x+iy$ maps the horizontal strip $\mcr_{2d}$ conformally onto the unit disc ${\mathbb D}=\{z \in \bdc: \ |z| <1\}$;
see Figure 10. Let
$(\tilde{U}+i\tilde{V})(\xi)$ be a {\it bounded} holomorphic function in $\mcr_{2d}$, admitting a $C^{1,\alpha}_{loc}(\overline{\mcr_{2d}})$ extension, and denote by
$\tilde{u}_\pm + i\,\tilde{v}_\pm$ the boundary values on the horizontal lines $y=0$ and $y=-2 d$, respectively. For $z=\phi(\xi)$ with $\xi \in \mcr_{2d}$,
let $(U+iV)(z)=(\tilde{U}+i\tilde{V})(\xi)$ be the corresponding holomorphic function in ${\mathbb D}$, whose boundary
values at $z=e^{it}$ on the upper/lower semicircles that make up the boundary of ${\mathbb D}$ we denote by
$u_\pm(t) + i v_\pm (t)$.

Now the bounded holomorphic function $U+iV$ in ${\mathbb D}$ can be represented by Poisson's formula
\begin{equation}\label{po}
U(z)+iV(z) =  \displaystyle\frac{1}{2\pi} \int_0^\pi \frac{e^{it}+z}{e^{it}-z}\,u_+(t)\,dt +
\displaystyle\frac{1}{2\pi} \int_{-\pi}^0 \frac{e^{it}+z}{e^{it}-z}\,u_-(t)\,dt,\qquad z \in {\mathbb D}\,.
\end{equation}
For $t \in (0,\pi)$, if $e^{it}=\phi(s)$ with $s \in \bdr$, then
$$e^{it}=\tanh\Big(\frac{\pi}{4d}\,(s+id)\Big)=\frac{\sinh\big(\frac{\pi}{2d}s\big) +i}{\cosh\big(\frac{\pi}{2d}s\big)}$$
yields $\cos(t)=\tanh\big(\frac{\pi}{2d}\,s\big)$ and $\sin(t)=\displaystyle\frac{1}{\cosh\big(\frac{\pi}{2d}\,s\big)}$, so that
$t=2\,\arctan\big(e^{-\frac{\pi}{2d}\,s}\big)$. On the other hand, for $t \in (-\pi,0)$, if $e^{it}=\phi(s-2id)$ with $s \in \bdr$, then
$$e^{it}=\tanh\Big(\frac{\pi}{4d}\,(s-id)\Big)=\frac{\sinh\big(\frac{\pi}{2d}s\big) -i}{\cosh\big(\frac{\pi}{2d}s\big)}$$
yields $\cos(t)=\tanh\big(\frac{\pi}{2d}\,s\big)$ and $\sin(t)=-\,\displaystyle\frac{1}{\cosh\big(\frac{\pi}{2d}\,s\big)}$, so that
$t=-\,2\,\arctan\big(e^{-\frac{\pi}{2d}\,s}\big)$. Thus the change of variables $t=2\,\arctan\big(e^{-\frac{\pi}{2d}\,s}\big)$ for $t \in (0,\pi)$
and $t=-\,2\,\arctan\big(e^{-\frac{\pi}{2d}\,s}\big)$
for $t \in (-\pi,0)$ transforms (\ref{po}) into
\begin{eqnarray}\label{po2}
U(z)+iV(z) &=& \displaystyle\frac{1}{4d} \int_\bdr \frac{\tanh\big(\frac{\pi}{4d}\,(s+id)\big)
+\tanh\big(\frac{\pi}{4d}\,(x+iy+id)\big)}{\tanh\big(\frac{\pi}{4d}\,(s+id)\big)-\tanh\big(\frac{\pi}{4d}\,(x+iy+id)\big)}\,
\frac{1}{\cosh\big(\frac{\pi}{4d}\,s\big)}\,\tilde{u}_+(s)\,ds \nonumber\\[0.1cm]
&&+ \displaystyle\frac{1}{4d} \int_\bdr \frac{\tanh\big(\frac{\pi}{4d}\,(s-id)\big)
+\tanh\big(\frac{\pi}{4d}\,(x+iy+id)\big)}{\tanh\big(\frac{\pi}{4d}\,(s-id)\big)-\tanh\big(\frac{\pi}{4d}\,(x+iy+id)\big)}\,
\frac{1}{\cosh\big(\frac{\pi}{4d}\,s\big)}\,\tilde{u}_-(s)\,ds\\
&=& \displaystyle\frac{1}{4d} \int_\bdr \frac{\sinh\big(\frac{\pi}{4d}\,(s+x+2id+iy)\big)}{\sinh\big(\frac{\pi}{4d}\,(s-x-iy)\big)}\,
\frac{1}{\cosh\big(\frac{\pi}{2d}\,s\big)}\,\tilde{u}_+(s)\,ds \nonumber \\[0.1cm]
&& +  \displaystyle\frac{1}{4d} \int_\bdr \frac{\sinh\big(\frac{\pi}{4d}\,(s+x+iy)\big)}{\sinh\big(\frac{\pi}{4d}\,(s-x-2id-iy)\big)}\,
\frac{1}{\cosh\big(\frac{\pi}{2d}\,s\big)}\,\tilde{u}_-(s)\,ds \nonumber
\end{eqnarray}
for $z=\phi(x+iy) \in {\mathbb D}$ with $(x,y) \in \mcr_{2d}$. Note that
$$\left\{\begin{array}{l}
\sinh(a+ib)=\sinh(a)\,\cos(b)\,+\,i\,\cosh(a)\,\sin(b),\quad \cosh(a+ib)=\cosh(a)\,\cos(b)\,+\,i\,\sinh(a)\,\sin(b)\,,\\[0.1cm]
\cosh(2a)=2\,\cosh^2(a)-1=2\,\sinh^2(a)+1,\qquad \cos(2a)=2\,\cos^2(a)-1=1-2\,\sin^2(a)\,,\\[0.1cm]
2\,\sinh(\xi_1)\,\sinh(\xi_2)=\cosh(\xi_1+\xi_2)\,-\,\cosh(\xi_1-\xi_2)\,,
\end{array}\right.$$
for $a,\,b \in \bdr$ and $\xi_1,\,\xi_2 \in \bdc$, respectively. Using these identities, we get
\begin{eqnarray}\label{simp1}
&&\displaystyle\frac{\sinh\big(\frac{\pi}{4d}\,(s+x+2id+iy)\big)}{\sinh\big(\frac{\pi}{4d}\,(s-x-iy)\big)}=
\frac{\sinh\big(\frac{\pi}{4d}\,(s+x+2id+iy)\big)\,\sinh\big(\frac{\pi}{4d}\,(s-x+iy)\big)}{\sinh^2\big(\frac{\pi}{4d}\,(s-x)\big)\cos^2\big(\frac{\pi}{4d}\,y\big)
+\cosh^2\big(\frac{\pi}{4d}\,(s-x)\big)\sin^2\big(\frac{\pi}{4d}\,y\big)} \nonumber \\[0.1cm]
&&=\displaystyle\frac{2\,\Big\{\cosh\big(\frac{\pi}{2d}\,(s+id+iy)\big)- \cosh\big(\frac{\pi}{2d}\,(x+id)\big)\Big\}}{\Big\{\cosh\big(\frac{\pi}{2d}\,(s-x)\big)-1\Big\}\Big\{\cos\big(\frac{\pi}{2d}\,y\big)+1\Big\} +
\Big\{\cosh\big(\frac{\pi}{2d}\,(s-x)\big)+1\Big\}\Big\{1-\cos\big(\frac{\pi}{2d}\,y\big)\Big\}} \nonumber \\[0.1cm]
&&=\displaystyle\frac{\cosh\big(\frac{\pi}{2d}\,(s+id+iy)\big)- \cosh\big(\frac{\pi}{2d}\,(x+id)\big)}{\cosh\big(\frac{\pi}{2d}\,(s-x)\big)-\cos\big(\frac{\pi}{2d}\,y\big)}
=\displaystyle\frac{\cosh\big(\frac{\pi}{2d}\,s\big)\,\cos\big(\frac{\pi}{2d}\,(y+d)\big) -
\cosh\big(\frac{\pi}{2d}\,x\big)\,\cos\big(\frac{\pi}{2d}\,d\big)}{\cosh\big(\frac{\pi}{2d}\,(s-x)\big)-\cos\big(\frac{\pi}{2d}\,y\big)} \\[0.1cm]
&&\qquad +\, i\,\displaystyle\frac{\sinh\big(\frac{\pi}{2d}\,s\big)\,\sin\big(\frac{\pi}{2d}\,(y+d)\big) - \sinh\big(\frac{\pi}{2d}\,x\big)\,\sin\big(\frac{\pi}{2d}\,d\big)}{\cosh\big(\frac{\pi}{2d}\,(s-x)\big)-\cos\big(\frac{\pi}{2d}\,y\big)}\nonumber\\[0.1cm]
&&=-\,\displaystyle\frac{\cosh\big(\frac{\pi}{2d}\,s\big)\,\sin\big(\frac{\pi}{2d}\,y\big) - i \,\Big\{\sinh\big(\frac{\pi}{2d}\,s\big)\,\cos\big(\frac{\pi}{2d}\,y\big) -
\sinh\big(\frac{\pi}{2d}\,x\big)\Big\}}{\cosh\big(\frac{\pi}{2d}\,(x-s)\big)-\cos\big(\frac{\pi}{2d}\,y\big)}\,.\nonumber
\end{eqnarray}
Similarly, we get
\begin{equation}\label{simp2}
\frac{\sinh\big(\frac{\pi}{4d}\,(s+x+iy)\big)}{\sinh\big(\frac{\pi}{4d}\,(s-x-2id-iy)\big)}=\,-\,
\displaystyle\frac{\cosh\big(\frac{\pi}{2d}\,s\big)\,\sin\big(\frac{\pi}{2d}\,y\big) - i \,\Big\{\sinh\big(\frac{\pi}{2d}\,s\big)\,\cos\big(\frac{\pi}{2d}\,y\big) +
\sinh\big(\frac{\pi}{2d}\,x\big)\Big\}}{\cosh\big(\frac{\pi}{2d}\,(x-s)\big)+\cos\big(\frac{\pi}{2d}\,y\big)}\,.
\end{equation}
Taking (\ref{simp1}) and (\ref{simp2}) into account, from (\ref{po2}) we deduce that for $z=\phi(x+iy) \in {\mathbb D}$ with
$x+iy \in \mcr_{2d}$, we can express $U(z)+iV(z)$ as
\begin{eqnarray}\label{po3}
&& -\,\displaystyle\frac{1}{4d} \int_\bdr \frac{\sin\big(\frac{\pi}{2d}\,y\big)}{\cosh\big(\frac{\pi}{2d}\,(x-s)\big)-\cos\big(\frac{\pi}{2d}\,y\big)}\,\tilde{u}_+(s)\,ds
+ \displaystyle\frac{i}{4d} \int_\bdr \frac{\tanh\big(\frac{\pi}{2d}\,s\big)\,\cos\big(\frac{\pi}{2d}\,y\big) -
\frac{\sinh(\frac{\pi}{2d}\,x)}{\cosh(\frac{\pi}{2d}\,s)}}{\cosh\big(\frac{\pi}{2d}\,(x-s)\big)-\cos\big(\frac{\pi}{2d}\,y\big)}\,\tilde{u}_+(s)\,ds
\\[0.1cm]
&& -\,\displaystyle\frac{1}{4d} \int_\bdr \frac{\sin\big(\frac{\pi}{2d}\,y\big)}{\cosh\big(\frac{\pi}{2d}\,(x-s)\big)+\cos\big(\frac{\pi}{2d}\,y\big)}\,\tilde{u}_-(s)\,ds
+\, \displaystyle\frac{i}{4d} \int_\bdr \frac{\tanh\big(\frac{\pi}{2d}\,s\big)\,\cos\big(\frac{\pi}{2d}\,y\big) +
\frac{\sinh(\frac{\pi}{2d}\,x)}{\cosh(\frac{\pi}{2d}\,s)}}{\cosh\big(\frac{\pi}{2d}\,(x-s)\big)+\cos\big(\frac{\pi}{2d}\,y\big)}\,\tilde{u}_-(s)\,ds \,. \nonumber
\end{eqnarray}
						As $y \uparrow 0$ the imaginary part of (\ref{po3}) tends to
$$\tilde{v}_+(x) =\displaystyle\frac{1}{4d} \,PV\,\int_\bdr \frac{\tanh\big(\frac{\pi}{2d}\,s\big)\, -
\frac{\sinh(\frac{\pi}{2d}\,x)}{\cosh(\frac{\pi}{2d}\,s)}}{\cosh\big(\frac{\pi}{2d}\,(x-s)\big)-1}\,\tilde{u}_+(s)\,ds
+ \displaystyle\frac{1}{4d} \int_\bdr \frac{\tanh\big(\frac{\pi}{2d}\,s\big)\, +
\frac{\sinh(\frac{\pi}{2d}\,x)}{\cosh(\frac{\pi}{2d}\,s)}}{\cosh\big(\frac{\pi}{2d}\,(x-s)\big)+1}\,\tilde{u}_-(s)\,ds \,,$$
the principal value integral being due to the singularity at $x=s$. Since
$$\tanh\big(\frac{\pi}{2d}\,s\big)\, \mp
\frac{\sinh(\frac{\pi}{2d}\,x)}{\cosh(\frac{\pi}{2d}\,s)} = \mp\,\Big\{ \tanh\big(\frac{\pi}{2d}\,s\big)\,\Big\{\cosh\big(\frac{\pi}{2d}\,(x-s)\big) \mp 1\Big\} \mp
\sinh\big(\frac{\pi}{2d}\,(x-s)\big) \Big\}\,,$$
we deduce that
\begin{eqnarray*}\label{tv}
\tilde{v}_+(x) &=& -\,\displaystyle\frac{1}{4d}  \,PV\,\int_\bdr \Big\{ \tanh\big(\frac{\pi}{2d}\,s\big) +
\frac{\sinh(\frac{\pi}{2d}\,(x-s))}{\cosh\big(\frac{\pi}{2d}\,(x-s)\big)-1}\Big\}\,\tilde{u}_+(s)\,ds \\[0.15cm]
&& +\,\displaystyle\frac{1}{4d} \int_\bdr \Big\{ \tanh\big(\frac{\pi}{2d}\,s\big) + \frac{\sinh(\frac{\pi}{2d}\,(x-s))}{\cosh\big(\frac{\pi}{2d}\,(x-s)\big)+1}\Big\}\,
\tilde{u}_-(s)\,ds \nonumber \\[0.15cm]
&=& -\displaystyle\frac{1}{4d}  \,PV\,\int_\bdr \Big\{ \tanh\big(\frac{\pi}{2d}\,s\big) + \coth\big(\frac{\pi}{4d}\,(x-s)\big)\Big\}\,\tilde{u}_+(s)\,ds
\nonumber \\[0.15cm]
&& +\,\displaystyle\frac{1}{4d} \int_\bdr \Big\{ \tanh\big(\frac{\pi}{2d}\,s\big) + \tanh\big(\frac{\pi}{4d}\,(x-s)\big)\Big\}\,\tilde{u}_-(s)\,ds \,.\nonumber
\end{eqnarray*}
Application of the same argument to the bounded holomorphic function $i(\tilde U +i\tilde V)= -\tilde V +i\tilde U$ leads to the representation
\begin{eqnarray}\label{tvxxx}
\tilde{u}_+(x)
&=& \displaystyle\frac{1}{4d}  \,PV\,\int_\bdr \Big\{ \tanh\big(\frac{\pi}{2d}\,s\big) + \coth\big(\frac{\pi}{4d}\,(x-s)\big)\Big\}\,\tilde{v}_+(s)\,ds
 \\[0.15cm]
&& -\,\displaystyle\frac{1}{4d} \int_\bdr \Big\{ \tanh\big(\frac{\pi}{2d}\,s\big) + \tanh\big(\frac{\pi}{4d}\,(x-s)\big)\Big\}\,\tilde{v}_-(s)\,ds \,.\nonumber
\end{eqnarray}

 Consider now a holomorphic function $\tilde U + i\tilde V$ in ${\mathcal R}_{2d}$, admitting a $C^{1,\alpha}_{loc}(\overline{\mcr_{2d}})$ extension, for which  $\tilde{V}$ is $2\pi$-periodic in the $x$-variable throughout $\overline{\mcr_{2d}}$. In contrast to $\tilde{V}$, the harmonic conjugate $\tilde{U}$ of $-\tilde V$ need not be in general
$2\pi$-periodic in the $x$-variable. Indeed, using the Cauchy-Riemann equations and the $2\pi$-periodicity of $\tilde{V}$ in
the $x$-variable, we see that the function $(x,y) \mapsto \tilde{U}(x+2\pi,y)-\tilde{U}(x,y)$ equals to a constant $K$ throughout $\mcr_{2d}$.
Since the expression
$$\frac{d}{dy}\,\int_{-\pi}^{\pi} \tilde{V}(\tau,y)\,d\tau = \int_{-\pi}^{\pi} \tilde{V}_y(\tau,y)\,d\tau = \int_{-\pi}^{\pi} \tilde{U}_x(\tau,y)\,d\tau \\[0.2cm]
= \tilde{U}(\pi,y)-\tilde{U}(-\pi,y),\qquad -2d < y <0,$$
is independent of $y$, integration on $[-2d,0]$ yields $K=\displaystyle\frac{2\pi}{2d}([\tilde{v}_+]-[\tilde{v}_-])$, where $[\cdot]$ denotes the average of a $2\pi$-periodic function.
Consequently, one may write
\begin{equation}\label{tU}
\tilde{U}(x,y)+i \tilde V(x,y)= \frac{[\tilde{v}_+] - [\tilde{v}_-]}{2d}(x+iy) + i[\tilde v_+]+\tilde{U}_0(x,y)+i \tilde V_0(x,y),\quad
(x,y) \in \mcr_{2d}\,,
\end{equation}
for some holomorphic function $\tilde{U}_0(x,y)+i\tilde V_0(x,y)$ that is $2\pi$-periodic in the $x$-variable, and therefore bounded.

The previous considerations
made to justify (\ref{tv}) are therefore applicable to $\tilde{U}_0(x,y)+i\tilde V_0(x,y)$, and lead to the following representation of the boundary values
$\tilde u_+$ on $y=0$ of $\tilde{U}$ in terms of the boundary values $\tilde v_{\pm}$ of $\tilde {V}$ on $y=0$ and $y=-2d$:
\begin{eqnarray}\label{ghtg}
\tilde{u}_+(x) &=& \frac{[\tilde{v}_+] - [\tilde{v}_-]}{2d}x +\displaystyle\frac{1}{4d} \,PV\,\int_\bdr \Big\{ \tanh\big(\frac{\pi}{2d}\,s\big) + \coth\big(\frac{\pi}{4d}\,(x-s)\big)\Big\}\,(\tilde{v}_+(s)-[\tilde v_+])\,ds \nonumber \\[0.15cm]
&& -\,\displaystyle\frac{1}{4d} \int_\bdr \Big\{ \tanh\big(\frac{\pi}{2d}\,s\big) + \tanh\big(\frac{\pi}{4d}\,(x-s)\big)\Big\}\,(\tilde{v}_-(s)-[\tilde v_-]) \,ds.
\end{eqnarray}

Let us now assume that $\tilde{U}+i\tilde{V}$ is analytic in the strip $\mcr_d$ and has for some $\alpha \in (0,1)$
a $C^{1,\alpha}_{loc}$ extension to $\overline{\mcr_d}$, with $\tilde{V}=0$ on $y=-d$.
The reflection principle \cite{Burk} permits
the analytic continuation of $\tilde{U}+i\tilde{V}$ to $\mcr_{2d}$ by setting
$$(\tilde{U}+i\tilde{V})(x+iy)=\overline{(\tilde{U}+i\tilde{V})(x-2id-iy)},\qquad x \in \bdr,\quad -2d <y <-d\,.$$
In this case $\tilde{v}_-(x)=-\,\tilde{v}_+(x)$, so that (\ref{ghtg}) takes the form
\begin{eqnarray*}
\tilde{u}_+(x) &=& \frac{[\tilde v_+]}{d}x+ \displaystyle\frac{1}{4d} \,PV\,\int_\bdr \Big\{ 2\,\tanh\big(\frac{\pi}{2d}\,s\big)
+ \tanh\big(\frac{\pi}{4d}\,(x-s)\big) +
\coth\big(\frac{\pi}{4d}\,(x-s)\big)\Big\}\,(\tilde{v}_+(s)-[\tilde{v}_+])\,ds \\[0.1cm]
&=& \frac{[\tilde v_+]}{d}x+ \displaystyle\frac{1}{2d} \,PV\,\int_\bdr \Big\{ \tanh\big(\frac{\pi}{2d}\,s\big) +
\coth\big(\frac{\pi}{2d}\,(x-s)\big)\Big\}\,(\tilde{v}_+(s)-[\tilde{v}_+])\,ds\,.
\end{eqnarray*}
Comparing this with (\ref{H}) we obtain, for $w \in C_{2\pi,\circ}^{1,\alpha}(\bdr)$, the explicit representation
\eqref {HE0} for $\mcc_d(w)$.

Now since $w \in C_{2\pi,\circ}^{1,\alpha}(\bdr)$ is $2\pi$-periodic, we can expand (\ref{HE0}) as
\begin{equation}\label{HE1}
\big(\mcc_d(w)\big)(x)=  \frac{1}{2d} \,PV\,\int_{-\pi}^\pi \sum_{k \in {\mathbb Z}}
\Big\{ \tanh\Big(\frac{\pi (s+2\pi k)}{2d}\Big) +  \coth\Big(\frac{\pi(x-s-2\pi k)}{2d}\Big)\Big\}\,w(s)\,ds\,.
\end{equation}

Define the functions $g$ and $\beta$ by
\begin{eqnarray*}
g(s) &=& \displaystyle\sum_{k \in {\mathbb Z}} \Big\{ \tanh\Big(\frac{\pi (s+2\pi k)}{2d}\Big) -\text{sgn}(k)\Big\},\qquad s \in \bdr ,\\[0.1cm]
\beta(s) &=& \displaystyle\sum_{k \in {\mathbb Z}} \Big\{ \coth\Big(\frac{\pi (s-2\pi k)}{2d}\Big) +\text{sgn}(k)\Big\},\qquad s \in \bdr
\setminus 2\pi{\mathbb Z},
\end{eqnarray*}
where $\text{sgn}(k)=1$ for $k > 0$, $\text{sgn}(k)=-1$ for $k<0$ and $\text{sgn}(0)=0$.
The estimates
\begin{eqnarray*}
&& \Big| \displaystyle\tanh\Big(\frac{\pi (s+2\pi k)}{2d}\Big) -\text{sgn}(k)\Big|
=\frac{2}{1+e^{\text{sgn}(k)\,\frac{\pi}{d}\,(s+2\pi k)}}
\le \frac{2}{1+e^{\frac{\pi}{d}\,(2\pi|k|-|s|)}},\qquad k \in {\mathbb Z},\quad s \in \bdr,\\[0.1cm]
&& \Big| \displaystyle\coth\Big(\frac{\pi (s-2\pi k)}{2d}\Big)   +  \text{sgn}(k)\Big|
=\frac{2}{e^{\frac{\pi}{d}\,|2\pi k+s|}-1}
\le \frac{2}{e^{\frac{\pi}{d}\,(2\pi|k|-|s|)}-1},\quad k \in {\mathbb Z},\qquad s \in \bdr \setminus 2\pi{\mathbb Z} ,
\end{eqnarray*}
ensure the continuity of $g$ on $\bdr$ and of $\beta$ on $\bdr \setminus 2\pi{\mathbb Z}$.
						The nature of the singularities
of $\beta$ at points of $2\pi{\mathbb Z}$ is plain. For $s \in \bdr \setminus 2\pi{\mathbb Z}$ we have
$$2=\sum_{k \in {\mathbb Z}} \Big\{ \text{sgn}(k+1) - \text{sgn}(k)\Big\}=g(s+2\pi)-g(s)=- \,\Big(\beta(s+2\pi)-\beta(s)\Big)\,.$$
						Finally defining
 $\beta_d(s)= \beta(s) \,+\,{s}/{\pi}$ and $g_d(s)= g(s) \,-\,{s}/{\pi}$,
both of which are $2\pi$-periodic, we obtain
\begin{equation}\label{HE2}
\sum_{k \in {\mathbb Z}} \Big\{ \tanh\Big(\frac{\pi (s+2\pi k)}{2d}\Big) +
\coth\Big(\frac{\pi(x-s-2\pi k)}{2d}\Big)\Big\}=\frac{2s-x}{\pi} + g_d(s)+\beta_d(x-s)\,.
\end{equation}
Substitution of (\ref{HE2}) into  (\ref{HE1}) yields \eqref{HE3}.
\end{proof}

\section{Some variational considerations}

We present two new variational formulations of
 travelling periodic gravity water waves with constant non-zero vorticity over a flat bed.
 We make no assumptions on the shape of the wave profile, thus allowing profiles that are not graphs.

The much-studied irrotational flows (that is, flows without vorticity)
are models for swell entering a region of still water, in which case there is no current, or else models for
swell entering a region with currents that are uniform with depth.
The earliest variational formulation dates back to Friedrichs \cite{F}.
It was followed by the work of Luke \cite{L}, Zakharov \cite{Z}, Babenko \cite{Ba} and others,
all of which expressed the Lagrangian in terms of a velocity potential.  See \cite{CSS} for further references.

While there is no velocity potential in the presence of vorticity, there is a stream function.
The first formulation,  valid for general vorticity distributions, recasts the water waves as extremals
of a suitably defined Lagrangian functional that is roughly the total energy of the wave.
Lagrangians of a similar type, expressed in terms of a hodograph transform
involving the stream function, were considered in \cite{CSS} under the assumption that
there are no stagnation points in the flow and that the free surface is the graph of a function (no overturning).
A related approach was pursued in the paper \cite{BurT},  which considers a water wave beneath
an elastic sheet obtained by minimization in  a class of rearrangements.
The novelty of our formulation is that it allows rotational waves, overturning free surface profiles,
stagnation points and critical layers in the flow.
Another advantage is that the Lagrangian is expressed directly in terms of the physical
variables instead of depending on a specific choice of coordinates.
A key aspect of our formulation is that the Lagrangian involves a variable domain of integration.
Although we are not aware of an earlier deduction of this formulation in this generality, we refer to it
as ``the standard variational formulation" due to its rather classical form.

The second variational formulation is specific to waves with constant vorticity and it
turns out to be more useful than the first formulation because it involves just a single function of a single variable.
Its essential advantage is that it reduces the governing equations to one pseudo-differential equation
for a function of one variable, namely, the elevation of the free surface when the fluid domain is
regarded as the conformal image
of a strip. Somewhat unexpectedly, this equation for the elevation is coupled to a scalar constraint.
The corresponding Lagrangian is essentially obtained by composing the first Lagrangian, suitably restricted,
with a conformal mapping from a strip. This formulation is new even for irrotational flows of finite depth.
As mentioned in the introduction, it could be regarded
as an analogue of Babenko's formulation \cite{Ba} for the irrotational water-wave problem of infinite depth,
which has turned out to be instrumental in the recent theory of global and subharmonic bifurcation
in the irrotational  infinite-depth case \cite{BDT1, BDT2, BST, ST}.
Our formulation opens up the possibility of similar investigations for waves of finite depth with constant non-zero vorticity,
in which case numerical studies \cite{DP, V} indicate that a much richer picture is expected to emerge.

\subsection{The standard variational formulation}

In this section we construct a functional on a certain function space, critical points of which are solutions
of the travelling water-wave problem (\ref{fs})--(\ref{g}). Since we are dealing with a free-boundary problem,
in which the domain $\Omega$ is unknown, the function space, to be denoted by $\mca$, will consists of
pairs $(\Omega,\psi)$, where $\psi$ is a function on $\Omega$. Some of the conditions expressed by
(\ref{fs})--(\ref{g}) will be required to hold for every element of $\mca$, while the remaining ones will emerge from the condition of criticality.

A domain $\Omega$ contained in the upper half of the $(X,Y)$ plane is called an
\emph{$L$-periodic strip-like domain} if its boundary consists of the real axis $\mcb$ and a curve
$\mcs$ described in parametric form by (\ref{fs1}) such that (\ref{fs2}) holds. We consider the space
${\mathcal A}$ of pairs $(\Omega,\psi)$, where
$\Omega$ is an $L$-periodic strip-like domain of class $C^{2,\alpha}$, and
$\psi \in C^{2,\alpha}_L({\R^2};\,\R)$ satisfying (\ref{g3})--(\ref{g4}).  The subscript
``$L$" is used to indicate periodicity in the $X$-variable, with period $L>0$, while $\alpha\in (0,1)$ is a H\"older exponent.
The behaviour of $\psi$ outside a neighbourhood of $\overline{\Omega}$ is of no importance.
The only restriction we impose on the geometry of
the free surface ${\mathcal S}$ is the requirement that the curve is not self-intersecting.  No restrictions
are imposed on the pattern of the streamlines so that we can handle overhanging profiles and critical layers.

For any pair $(\Omega,\psi)$ in $\mca$, the periodicity in the $X$-variable
permits us to restrict ourselves to a cell $\Omega^\dagger$ bounded below by the real axis ${\mathcal B}$,
above by the free surface $\mcs$ and laterally by two vertical lines situated at horizontal distance $L$. For
definiteness, and to ensure that $\Omega^\dagger$ is a connected set, one may choose one of the vertical
lines in the definition of  $\Omega^\dagger$ (and hence the other one too) so as to pass through the
lowest point of $\mcs$.
Consider on the space $\mca$ the functional
 \begin{equation}\label{L}
{\mathcal L}(\Omega,\psi)=\iint_{\Omega^\dagger} \Big\{ |\nabla\psi|^2\,+\,2\Upsilon\psi\,-\,2gY\,+\,Q\,\Big\}\,d\X\,,
\end{equation}
where $\X=(X,Y)$.

\begin{theorem}\label{t1}
Any critical point $(\Omega,\psi)$ of the functional ${\mathcal L}$ over the space $\mca$
 is a solution to the governing equations (\ref{fs})--(\ref{g}).
\end{theorem}

\begin{proof}
We only need to show that a critical point $(\Omega,\psi)$  of ${\mathcal L}$ on $\mca$  satisfies (\ref{g1}) and (\ref{g2}).
We first investigate, for a fixed domain $\Omega$ the rate of change of ${\mathcal L}$ with respect to variations of
the {\it dependent} variable $\psi$ which do
not change the boundary values on ${\mathcal S}$ and ${\mathcal B}$. For smooth functions $\varphi: \Omega \to \R$ that are $L$-periodic
in the horizontal variable and vanish in a neighbourhood of ${\mathcal S}$ and of ${\mathcal B}$, the first variation $\delta{\mathcal L}(\Omega,\psi,\varphi)$
of ${\mathcal L}$ at $(\Omega,\psi)$ in direction of $\varphi$ is defined by
$$\delta{\mathcal L}(\Omega,\psi,\varphi):=\frac{d}{d\e}\,{\mathcal L}(\Omega,\psi+\e \varphi)\Big|_{\e=0}$$
cf. Section 2.1 in \cite{GH}. A {\it weak extremal} $\psi$ of ${\mathcal L}$ is a  solution of the Euler-Lagrange
equation $\delta{\mathcal L}(\Omega,\psi,\varphi)=0$
for all $\varphi$ of the type described above, and a straight-forward computation shows that the weak extremals
$\psi \in\,C^{2,\alpha}_L(\Omega;\,\R)$
are precisely the solutions to the {\it Euler-Lagrange equation}
\begin{equation}\label{E}
\Delta\psi=\Upsilon\quad\hbox{in}\quad \Omega\,.
\end{equation}

We now consider variations $\X \mapsto {\frak D}_\e(\X)$ of the {\it independent} variables $\X$, allowing modifications
of the domain $\Omega$, but such that the modified domains are still $L$-periodic strip-like domains. This leads
to the notion of {\it strong inner extremals} that will not only satisfy an equation  of Euler-Lagrange-type but also
a free boundary condition, cf.\ Section 3.2 of \cite{GH}. Given a vector field ${\frak f}$ of class $C^{2,\alpha}_L(\R^2;\,\R^2)$,
the support of which is contained in the upper half-plane $\{ \X:\ Y > 0\}$, consider the parameter-dependent family of mappings
${\frak D}_\e: R^2 \to \R^2$ defined by
$${\frak D}_\e({\mathbb X}):=\X\,+\,\e\,{\frak f}(\X).$$
Let $\Omega_\e={\frak D}_\e(\Omega)$. For $|\e| <\e_0$ with $\e_0>0$ sufficiently small, ${\frak D}_\e$ is a diffeomorphism
from $\overline{\Omega}$ onto $\overline{\Omega_\e}$. For a given $\psi \in C^{2,\alpha}_L(\overline{\Omega};\,\R)$ we can now define a
{\it strong inner variation in the direction of ${\frak f}$} by
$$\psi_\e(\X)=\psi({\frak D}_\e^{-1}(\X))\quad\hbox{for}\quad \X \in \overline{\Omega_\e}\quad\hbox{and}\quad \e \in (-\e_0,\e_0)\,.$$
Note that, for any $\e \in (-\e_0,\e_0)$, the pair $(\Omega_\e, \psi_\e)$ still belongs to $\mca$.
Correspondingly, the expression
$$\partial{\mathcal L}(\Omega,\psi,{\frak f}):=\frac{d}{d\e}\,\iint_{(\Omega_\e)^\dagger} F\big(\X,\,\psi_\e(\X),\,\nabla\psi_\e(\X)\big)\,d\X\Big|_{\e=0}$$
where
\begin{equation}\label{F}
F(X,\,Y,\,z,\,p_1,\,p_2)=p_1^2\,+\,p_2^2\,+\,2\Upsilon z\,-\,2gY\,+Q,
\end{equation}
is the integrand in the functional ${\mathcal L}$, is called the {\it inner variation} of the
functional ${\mathcal L}$ at $\psi$ in the direction of the vector field ${\frak f}$.
A  mapping $\psi \in C^{2,\alpha}_L({\overline \Omega}, \R)$ is said
to be a {\it strong inner extremal} of ${\mathcal L}$ if $\partial{\mathcal L}(\Omega,\psi,{\frak f})=0$ holds for all
such vector fields ${\frak f}$. Also, a mapping $\psi \in C^{2,\alpha}_L({\overline \Omega}, \R)$ is said
to be an {\it inner extremal} of ${\mathcal L}$ if $\partial{\mathcal F}(\Omega,\psi,{\frak f})=0$ holds for the subclass
of vector fields ${\frak f}$ that vanish in a neighbourhood of ${\mathcal S}$ and of ${\mathcal B}$. Any inner
extremal $\psi \in C^{2,\alpha}_L(\overline{\Omega};\,\R)$ will satisfy the {\it Noether equation}
\begin{equation}\label{el2}
(\Delta\psi -\Upsilon)\,\nabla\psi=0\quad\hbox{in}\quad \Omega\,,
\end{equation}
cf. Section 3.1 and Proposition 1, Section 3.2 in \cite{GH}.
Note the equivalence of (\ref{E}) and (\ref{el2}) if stagnation points do not occur. However, in
our setting stagnation points are permissible.

Moreover, a strong inner extremal satisfies (\ref{el2}) as well as the boundary conditions
\begin{equation}\label{bcv}
\left\{\begin{array}{l}
\nu_1\,\Big( p_1\,\displaystyle\frac{\partial F}{\partial p_1}\,-\,F\Big) \,+\,\nu_2\,p_1\,\displaystyle\frac{\partial F}{\partial p_2}=0,\\[0.3cm]
\nu_1\, p_2\,\displaystyle\frac{\partial F}{\partial p_1} \,+\,\nu_2\,\Big(p_2\,\displaystyle\frac{\partial F}{\partial p_2}\,-\,F\Big)=0,
\end{array}\right.\quad\hbox{on}\quad {\mathcal S}\,,
\end{equation}
where $\nu=(\nu_1,\nu_2)$ is the outer unit normal to ${\mathcal S}$ and $F$ is defined in (\ref{F}). The condition $\psi=0$ on ${\mathcal S}$,
which must hold since $(\Omega,\,\psi) \in {\mathcal A}$, ensures that $\nu$ is collinear with $\nabla_\X\psi$ at all points
$\X \in {\mathcal S}$ where $\nabla_\X\psi(\X) \neq (0,0)$. At every such point, a mere substitution into (\ref{bcv}) in combination with
(\ref{g3}) leads us to
(\ref{g2}). On the other hand, at points $\X \in {\mathcal S}$ where $\nabla_\X\psi(\X) = (0,0)$, inspection of (\ref{bcv}) yields $F=0$, which
in this case is the same as (\ref{g2}). This completes the proof of Theorem \ref{t1}. \end{proof}

\begin{remark}
{\rm While Theorem \ref{t1} is specifically formulated
for flows with constant vorticity, there is a counterpart for travelling waves on flows with
{\it general vorticity distributions}, in which case equation (\ref{g1}) is replaced by
\[\Delta\psi=\Upsilon(\psi),\]
where now $\Upsilon$ is an arbitrary function of one variable. In this case, we would work on the same space
$\mca$, but with the functional $\mcl$ in (\ref{L}) replaced by
 \[
{\mathcal L}(\Omega,\psi)=\iint_{\Omega^\dagger} \Big\{ |\nabla\psi|^2\,+\,2\,\Xi(\psi)\,-\,2gY\,+\,Q\,\Big\}\,d\X\,,
\]
and, correspondingly, the integrand $F$ in (\ref{F}) being replaced by \[
F(X,\,Y,\,z,\,p_1,\,p_2)=p_1^2\,+\,p_2^2\,+\,2\,\Xi(z)\,-\,2gY\,+Q\,,
\]
where $\Xi(z)=\displaystyle\int_0^z\Upsilon(s)\,ds$. Even discontinuous vorticity functions $\Upsilon$ are permissible if one lowers the regularity requirements and uses a suitable weak formulation (see \cite{CS2, VZ}).
The setting of weak solutions is well-suited for the use of geometric methods
to investigate the behaviour of the surface wave profile near stagnation points cf. \cite{VW1, VW2}. Note that in \cite{CSS},
an equivalent form of the functional ${\mathcal L}$, expressed in terms of a hodograph transform
that involves the stream function, was considered under the restrictive assumptions that
there are no stagnation points in the flow and that the free surface is the graph of a function. The present setting accommodates
overhanging free surface profiles as well as stagnation points in the flow.
}\end{remark}

\subsection{An alternative convenient variational formulation} Since one of the great advantages of Lagrangian dynamics is
the freedom it allows in the choice of coordinates and
since variational formulations involving as few dependent and independent variables as possible are preferable, we will express
the functional ${\mathcal L}$ in  terms of the function $v$, introduced in Section 2.1.
In terms of $v$, the governing equations (\ref{fs})--(\ref{g}) were reformulated in Section 2.2
as a single pseudodifferential equation (\ref{vara}) coupled with the scalar constraint (\ref{meana}).
We  will show that the functional $\Lambda$ introduced in the variational formulation of the
equations (\ref{vara})-(\ref{meana}) of Section 2.2, which were discussed in Section 2.3, corresponds to
an energy-type functional in the
physical variables.

In the notation of the previous subsection, with $\Omega$ an arbitrary $2\pi/k$-periodic strip-like domain,
let us express ${\mathcal L}(\Omega,\psi^\Omega)$ in terms of the corresponding function $v$, where
$(\Omega,\,\psi^\Omega) \in \mca$ and $\psi^\Omega$ is the unique solution of (2.2a)-(2.2c).
We would like to emphasize that, throughout the rest of this section, we always deal with
pairs $(\Omega,\psi^\Omega)$, rather than with arbitrary elements $(\Omega,\psi)$ of $\mca$.
However, for notational convenience we often use the notation $\psi$ instead of $\psi^\Omega$.
In what follows, we will prove the formula

\begin{align}
{\mathcal L}(\Omega,\psi^\Omega)&=\displaystyle\int_{-\pi}^\pi\Big(Q\,v\,-\,g\,v^2\,-\,\frac{\Upsilon^2}{3}\,v^3\Big)\,\Big(\frac{1}{k}\,+\,\mcc_{kh}(v')\Big)\,dx \nonumber\\&\qquad
+\,\displaystyle\int_{-\pi}^\pi\Big(m\,-\,\frac{\Upsilon}{2}\,v^2\Big)\,\Big(\frac{m}{kh}\,-\,\frac{\Upsilon}{2kh}\,[v^2]\,
-\,\Upsilon\,\mcc_{kh}(vv')\Big)\,dx\,,       \label{iale}
\end{align}
which is the same as $\Lambda(w,h)$ in \eqref{ev}.
In order to do so,
it is convenient to choose $\Omega^\dagger$ so that one of its lateral boundaries passes through a
lowest point (trough) of $\mcs$, and we denote by $\mcs^\dagger$ and $\mcb^\dagger$
the top and bottom boundaries of $\Omega^\dagger$. It is also convenient to choose the
conformal mapping $U+iV$ so that $\mcs^\dagger$ is the image of the horizontal
line segment $\{(x,0):x\in[-\pi,\pi]$. Then $\mcb^\dagger$ will be the conformal image of a
segment $\{(x,-kh):x\in[\beta-\pi,\beta+\pi]\}$, for some $\beta\in\mathbb{R}$. Firstly,
denoting by $\nu$ the outward unit normal at the boundary
$\partial\Omega$ of $\Omega$, Green's formula yields
$$\iint_{\Omega^\dagger} |\nabla\psi|^2\,d\X = -\,\Upsilon\iint_{\Omega^\dagger} \psi\,d\X \,-\,
m\int_{\mathcal B^\dagger} \frac{\partial\psi}{\partial \nu}\,d\sigma$$
by periodicity in the $X$-variable and by taking (\ref{g1}), (\ref{g3}) and (\ref{g4}) into account.
The outer normal on ${\mathcal B}$
is $\nu=(0,-1)$, so that
$$
\frac{\partial\psi}{\partial \nu}(X,0)=-\,\psi_Y(X,0)\quad\hbox{on}\quad {\mathcal B}.$$
Recall that the conformal mapping is $X=U(x,y),\ Y=V(x,y)$.
Since $\psi_X(X,0)=0$ by (\ref{g4}), from (\ref{x}) we infer that
$$\xi_y(x,-kh)=\psi_Y(X,0)\,U_x(x,-kh)$$
as $V_y=U_x$ by the Cauchy-Riemann equations. Therefore, using the periodicity of $\xi$,
$$\int_{\mathcal B^\dagger} \frac{\partial\psi}{\partial \nu}\,d\sigma=-\int_{\beta-\pi}^{\beta+\pi} \xi_y(x,-kh)\,dx
=-\int_{-\pi}^\pi \xi_y(x,-kh)\,dx,$$
so that
\begin{equation}\label{e1}
\iint_{\Omega^\dagger} |\nabla\psi|^2\,d\X = -\,\Upsilon\iint_{\Omega^\dagger} \psi\,d\X +
m\int_{-\pi}^\pi \xi_y(x,-kh)\,dx\,.
\end{equation}
Similarly,
$$\begin{array}{l}
\displaystyle\iint_{\Omega^\dagger} \psi\,d\X =
\displaystyle\iint_{\Omega^\dagger}  \Big(\Delta\,\frac{Y^2}{2}\Big)\ \psi \,d\X  = \displaystyle\iint_{\Omega^\dagger}\frac{Y^2}{2}\,\Delta\psi\,d\X +
\displaystyle\int_{\partial\Omega^\dagger} \psi \,\frac{\partial}{\partial \nu}\,\Big(\frac{Y^2}{2}\Big)\,d\sigma
-\,\displaystyle\int_{\partial\Omega^\dagger} \frac{Y^2}{2}\,\frac{\partial \psi}{\partial \nu}\,d\sigma \\[0.35cm]
\qquad = \displaystyle\frac{\Upsilon}{2}\,\iint_{\Omega^\dagger} Y^2\,d\X\,-\,\int_{\mathcal S^\dagger}\frac{Y^2}{2}\,\frac{\partial \psi}{\partial \nu}\,d\sigma.
\end{array}$$
But $\nu=\displaystyle\frac{(-\,V_x,U_x)}{\sqrt{U_x^2\,+\,V_x^2}}$ is the outer normal on ${\mathcal S}$,
so that on ${\mathcal S}$ we have
$$\displaystyle\frac{\partial\psi}{\partial \nu}=(\psi_X,\psi_Y)\cdot \nu=\frac{\psi_XU_y+\psi_YV_y}{\sqrt{U_x^2+V_x^2}}\Big|_{(x,0)}
=\frac{\xi_y}{\sqrt{U_x^2+V_x^2}}\Big|_{(x,0)}$$
in view of (\ref{x}) and the Cauchy-Riemann equations $U_x=V_y$, $U_y=-\,V_x$. Hence
$$\int_{\mathcal S^\dagger}\frac{Y^2}{2}\,\frac{\partial \psi}{\partial \nu}\,d\sigma=\int_{-\pi}^\pi \frac{\xi_y}{\sqrt{U_x^2+V_x^2}}\,
\frac{V^2}{2}\,{\sqrt{U_x^2+V_x^2}}\Big|_{(x,0)}\,dx=\int_{-\pi}^\pi \frac{V^2}{2}\,\xi_y\Big|_{(x,0)}\,dx$$
as $d\sigma=\sqrt{U_x^2+V_x^2}\Big|_{(x,0)}\,dx$. Moreover, using the Cauchy-Riemann equations for $U+iV$, the divergence theorem, the
periodicity in the $X$-variable and the fact that $V(x,-kh)=0$ from (\ref{v}), we get
$$\iint_{\Omega^\dagger} Y^2\,d\X = \iint_{\Omega^\dagger} \nabla_\X \cdot \Big(0,\,\frac{Y^3}{3}\Big)\,d\X =
\displaystyle\frac{1}{3}\,\int_{-\pi}^\pi V^3(x,0)\,V_y(x,0)\,dx=\frac{1}{3}\,\int_{-\pi}^\pi V^3(x,0)\,U_x(x,0)\,dx.$$
 Combining the last four displayed equations, we obtain
\begin{equation}\label{e2}
\displaystyle\iint_{\Omega^\dagger} \psi\,d\X =\,\displaystyle\frac{\Upsilon}{6}\,\int_{-\pi}^\pi V^3(x,0)\,U_x(x,0)\,dx
\,-\,\displaystyle\frac{1}{2}\,\int_{-\pi}^\pi V^2(x,0)\,\xi_y(x,0)\,dx\,.
\end{equation}
Proceeding as above, we also get
$$\displaystyle\iint_{\Omega^\dagger} 1\,d\X =  \displaystyle\int_{-\pi}^\pi V(x,0)\,U_x(x,0)\,dx\,,\qquad
\displaystyle\iint_{\Omega^\dagger} Y\,d\X = \displaystyle\frac{1}{2}\,\int_{-\pi}^\pi V^2(x,0)\,U_x(x,0)\,dx\,.$$
Substituting the last three equations and (\ref{e1}) into (\ref{L}) yields
\begin{eqnarray*}
{\mathcal L}(\Omega,\psi^\Omega)
&=& \displaystyle\,m\,\int_{-\pi}^\pi \xi_y(x,-kh)\,dx\,
-\frac{\Upsilon}{2}\,\int_{-\pi}^\pi V^2(x,0)\,\xi_y(x,0)\,dx\,\\[0.35cm]
&& \,+\,\int_{-\pi}^\pi   \left\{    \displaystyle\frac{\Upsilon^2}{6}\,  V^3(x,0)
-\,\displaystyle\,g\, V^2(x,0)\,\, \,+\,\displaystyle\,Q\, V(x,0) \right\}
\,U_x(x,0)\,dx\,.
\end{eqnarray*}
Using (\ref{z}) to express
$$\xi_y=\Upsilon\,VV_y+\zeta_y=\Upsilon VU_x +\zeta_y$$
and using $V=0$ on $y=-hk$, we get
\begin{eqnarray*}
{\mathcal L}(\Omega,\psi^\Omega)
&=& \displaystyle\,m\,\int_{-\pi}^\pi \zeta_y(x,-kh)\,dx\,-
\displaystyle\frac{\Upsilon}{2}\,\int_{-\pi}^\pi V^2(x,0)\,\zeta_y(x,0)\,dx\,\\[0.35cm]
&& \,+\,\int_{-\pi}^\pi   \left\{   - \displaystyle\frac{\Upsilon^2}{3}\,  V^3(x,0)
-\,\displaystyle\,g\, V^2(x,0)\,\, \,+\,\displaystyle\,Q\, V(x,0) \right\}
\,U_x(x,0)\,dx\,.
\end{eqnarray*}
However, the harmonicity (\ref{gc1}) of $\zeta$ in the strip $\mcr_{kh}$ yields
$$\int_{-\pi}^\pi \zeta_y(x,-kh)\,dx=\int_{-\pi}^\pi \zeta_y(x,0)\,dx,$$
so that
\begin{eqnarray*}
{\mathcal L}(\Omega,\psi^\Omega) &=& \displaystyle\int_{-\pi}^\pi\Big\{\,Q\,V(x,0)\,-\,g\,V^2(x,0)\,-\,\frac{\Upsilon^2}{3}\,V^3(x,0)\,
\Big\}\,U_x(x,0)\,dx\\[0.35cm]
&&\,+\,\displaystyle\int_{-\pi}^\pi\Big\{m\,-\,\frac{\Upsilon}{2}\,V^2(x,0)\Big\}\,\zeta_y(x,0)\,dx.
\end{eqnarray*}
Using the relations (\ref{vzc}) and the fact that $U_x(x,0)=V_y(x,0) = \frac{1}{k}+\mcc_{kh}(v')$,
we obtain (\ref{iale}), which is precisely the functional $\Lambda(w,h)$ introduced in Section 2.3.

\end{appendix}

\section*{Acknowledgements}

The research of the first author and partly that of the third author were supported by the ERC Advanced Grant ``Nonlinear studies of water
flows with vorticity", and that of the second author by NSF Grant DMS-1007960. The third author would also like to thank the Isaac Newton Institute for Mathematical Sciences, Cambridge, for support and hospitality during the programme ``Free Boundary Problems and Related Topics", where work on this paper was undertaken.


\end{document}